\documentclass[bj]{imsart}
\usepackage{amsfonts,amsmath,amssymb,amsthm,enumerate,color,hyperref,bbm,tabularx,ifthen,twoopt}
\usepackage{graphicx}
\usepackage[square,numbers]{natbib}

\newcounter{hypA}

\graphicspath{{figures/}{figures_simuls/}}

\def\PE{\mathbb{E}}

\def\rmi{\mathrm{i}}
\def\rme{\mathrm{e}}

\def\rset{\mathbb{R}}

\def \1{\mathbbm{1}}
\DeclareMathOperator{\argmin}{arg\,min}
\DeclareMathOperator{\argmax}{arg\,max}
\DeclareMathOperator*{\med}{med}

\newtheorem{theo}{Theorem}
\newtheorem{prop}[theo]{Proposition}
\newtheorem{lemma}[theo]{Lemma}

\theoremstyle{remark}
\newtheorem{remark}{Remark}

\begin{document}

\begin{frontmatter}

\title{A robust approach for estimating change-points in the mean of an AR(1) process}
\runtitle{Change-points in the mean of an AR(1) process}

\begin{aug}
  \author{\fnms{S.}  \snm{Chakar}\thanksref{a,b,e1}\ead[label=e1,mark]{souhil.chakar@agroparistech.fr}} \thankstext{t2}{Corresponding author},
  \author{\fnms{E.} \snm{Lebarbier}\thanksref{a,b,e2}\ead[label=e2,mark]{emilie.lebarbier@agroparistech.fr}},
  \author{\fnms{C.}  \snm{L\'evy-Leduc}\thanksref{a,b,e3}%
    \ead[label=e3,mark]{celine.levy-leduc@agroparistech.fr}%
}
    \and
  \author{\fnms{S.} \snm{Robin}\thanksref{a,b,e4}\ead[label=e4,mark]{stephane.robin@agroparistech.fr}},

  \runauthor{S. Chakar et al.}


  \address[a]{AgroParisTech, UMR 518 MIA, F-75005 Paris, France.}

  \address[b]{INRA, UMR 518 MIA, F-75005 Paris, France. \printead{e1,e2} \printead{e3,e4}}

\end{aug}

\begin{abstract}
We consider the problem of multiple change-point estimation in the mean of a Gaussian AR(1) process.  Taking into account the dependence structure does not allow us to use the dynamic programming algorithm, which is the only algorithm giving the optimal solution in the independent case. We propose a robust estimator of the autocorrelation parameter, which is consistent and satisfies a central limit theorem. Then, we propose to follow the classical inference approach, by plugging this estimator in the criteria used for change-points estimation. We show that the asymptotic properties of these estimators are the same as those of the classical estimators in the independent framework. The same plug-in approach is then used to approximate the modified BIC and choose the number of segments. This method is implemented in the R package \textsf{AR1seg} and is available from the Comprehensive R Archive Network (CRAN). This package is used in the simulation section in which we show that for finite sample sizes taking into account the dependence structure improves the statistical performance of the change-point estimators and of the selection criterion.
\end{abstract}

\begin{keyword@MSC}
62M10, 62F12, 62F35.
\end{keyword@MSC}

\begin{keyword@KWD}
Auto-regressive model, change-points, robust estimation of the AR(1) parameter, time series, model selection.
\end{keyword@KWD}



\end{frontmatter}

\section{Introduction}




Change-point detection problems arise in many fields, such as
genomics (\cite{braun1998statistical}, \cite{braun2000multiple},
\cite{picard2005statistical}), medical imaging
\cite[]{lavielle2005using}, earth sciences (\cite{Williams:2003},
\cite{Gazeaux:2013}), econometrics (\cite{lai2005autoregressive},
\cite{lavielle}) or climate (\cite{mestre2000methodes},
\cite{climat}). In many of these problems, the observations can not
be assumed to be independent. Indeed the autocovariance structure of
the time series display more complex patterns and might be taken
into account in change-point estimation.


An abundant literature exists about the statistical theory of
change-point detection. Only speaking about Gaussian processes,
various frameworks have been considered ranging from the independent
case with changes in the mean \cite[]{Bas93N}, to more complex
structural changes \cite[]{BP}, dependent processes
\cite[]{lavielle} or processes with changes in all parameters
\cite[]{BKW10}. \\
 \cite{lavielle} and \cite{LM} proved that, if the
number of changes is known, the least-squares estimators of the
change-point locations and of the parameters of each segment are
consistent under very mild conditions on the auto-covariance
structure of the process with changes in the mean. A
quasi-likelihood approach is also proved to provide consistent
estimates for the model with changes in all
parameters by \cite{BKW10}. 
 Many model selection criteria have also
been proposed to estimate the number of changes, mostly in the
independent case (see for example \cite{yao},
\cite{lavielle2005using}, \cite{lebarbier2005detecting} and
\cite{zs}).

Change-point detection also raises algorithmic issues as the determination of the optimal set of change-point locations is a discrete optimization problem. The dynamic programming algorithm introduced by \cite{AugerLawrence} is the only way to recover this optimal segmentation. The computational complexity of this algorithms is quadratic relatively to the length of the series. Only this algorithm and some of its improvements (such as these proposed by \cite{pruned} or \cite{KFE12-JASA}) provide exactly the optimal change-point location estimators. \\
However, the dynamic programming algorithm only applies when ($i$) the loss function (e.g. the negative log-likelihood) is additive with respect to the segments and when ($ii$) no parameter to be estimated is common to several segments. These requirements are met by the least-square criterion (which corresponds to the negative log-likelihood in the Gaussian homoscedastic independent model with changes in the mean) or by the model and criterion considered by 
\cite{BKW10}. In other cases, iterative and stochastic procedures are needed (see \cite{BPalgo} or \cite{LiL12}).


In this paper, we consider the segmentation of an AR(1) process with homogeneous auto-correlation coefficient $\rho^{\star}$:
\begin{equation}\label{eq:modele_new}
y_i=\mu_k^{\star}+\eta_i\;,\;
t_{n,k}^{\star}+1\leq i\leq t_{n,k+1}^{\star}\;,\; 0\leq k\leq m^{\star}\;,\; 1\leq
i\leq n\;,
\end{equation}
where $(\eta_i)_{i\in \mathbb{Z}}$ is a zero-mean stationary AR(1) Gaussian process defined as the solution of
\begin{equation}\label{eq:ar1}
\eta_i = \rho^{\star}\eta_{i-1} + \epsilon_i\;,
\end{equation}
where $|\rho^{\star}|<1$ and the $\epsilon_i$'s are i.i.d. zero-mean Gaussian random variables with variance $\sigma^{\star2}$. We further also assume that $y_0$ is a Gaussian random variable with mean $\mu_0^{\star}$ and variance $\sigma^{\star2}/(1-{\rho^{\star}}^2)$.
Actually, most of the results we provide in this paper hold without the Gaussian assumption.

Note that this model is different from the ones considered by \cite{davis2006structural} and \cite{BKW10}. Indeed, \cite{davis2006structural} considered the segmentation issue of a non-stationary time series which consists of blocks of different autoregressive processes where all the parameters of the autoregressive processes change from one segment to the other. \cite{BKW10} proposed a methodology for estimating the change-points of a non-stationary time series built from a general class of models having piecewise constant parameters. In this framework, all the parameters may change jointly at each change-point. This differs from our model \eqref{eq:modele_new} where the parameters $\rho^{\star}$ and $\sigma^{\star}$ are not assumed to change from one segment to the other.
The direct maximum-likelihood inference for such a process violates both requirements ($i$) and ($ii$). Indeed the log-likelihood is not additive with respect to the segments because of the dependence that exists between data from neighbor segments and the unknown coefficient $\rho^{\star}$ needs to be estimated jointly over all segments. \\
Our aim is to propose a methodology for estimating both the change-point locations $\boldsymbol{t_n^{\star}} = (t_{n,k}^{\star})_{1\leq k\leq m^{\star}}$ and the means $\boldsymbol{\mu}^{\star}=(\mu_k^{\star})_{0\leq k\leq m^{\star}}$, accounting for the existence of the auto-correlation $\rho^{\star}$.

In the sequel, we shall use the following conventions: $t_{n,0}^{\star}=0,t_{n,m^{\star}+1}^{\star}=n$ and assume that there exists $\boldsymbol{\tau^{\star}}=(\tau_k^{\star})_{0\leq k\leq m+1}$ such that, for $0\leq k\leq m+1$ $t_{n,k}^{\star}=\lfloor n\tau_k^{\star}\rfloor$, $\lfloor x\rfloor$ denoting the integer part of $x$. Consequently, $\tau_0^{\star}=0$ and $\tau_{m^{\star}+1}^{\star}=1$.

If $\rho^{\star}$ was known, the series could be decorrelated and the dynamic programming algorithm then used for the segmentation of this decorrelated series. Here, $\rho^{\star}$ is unknown, but is estimated, and this estimator is then used to decorrelate the series.
To this aim, we borrow techniques from robust estimation {\cite[]{MG}}. Briefly speaking, we consider the data observed at the change-point locations as outliers and propose an estimate of $\rho^{\star}$ that is robust to the presence of such outliers. We shall prove that the estimate we propose is consistent and satisfies a central limit theorem.

We shall prove that the resulting change-point estimators satisfy the same asymptotic properties as those proposed by \cite{LM} and \cite{BKW10}. 
Finally, we propose a model selection criterion inspired by the one proposed in \cite{zs} and prove some asymptotic properties of this criterion.

This method is implemented in the R package \textsf{AR1seg} and is available from the Comprehensive R Archive Network (CRAN).

This paper is organized as follows. In Section \ref{sec:correlation}, we propose a robust estimator for $\rho^{\star}$ and establish its asymptotic properties. In Section \ref{sec:change-points}, we prove that the change-point estimators defined in \eqref{eq:tau_n} are consistent in both the Gaussian and the non-Gaussian case. In Section \ref{sec:selection}, we provide a consistent model selection criterion in the non-Gaussian case and derive an approximation of a Gaussian criterion. In Section \ref{sec:simul}, we illustrate by a simulation study the performance of this approach for time series having a finite sample size.

\section{Robust estimation of the parameter $\rho^{\star}$}\label{sec:correlation}

The aim of this section is to provide an estimator of $\rho^{\star}$ which can deal with the
presence of change-points in the data. In the absence of change-points ($m^{\star}=0$ in (\ref{eq:modele_new})), a consistent estimator
of $\rho^{\star}$ could be obtained by using the classical autocorrelation function estimator of $(y_i)_{0\leq i\leq n}$ computed at lag 1. 
Since change-points can be seen as outliers
in the AR(1) process, we shall propose a robust approach
for estimating $\rho^{\star}$. 
\cite{MG} propose a robust estimator of the autocorrelation function of a stationary time series 
based on the robust scale estimator proposed by \cite{CR}.
More precisely, the approach of \cite{MG} would result in the following estimate of $\rho^{\star}$:
\begin{equation*}
\widehat{\rho}_{\textrm{MG}}=\dfrac{Q_{n}^2 \left(y^+\right) - Q_{n}^2 \left(y^-\right)}{Q_{n}^2 \left(y^+\right) + Q_{n}^2 \left(y^-\right)}\;,
\end{equation*}
where $y^+ = (y_{i+1}+y_i)_{0\leq i\leq n-1}$, $y^- = (y_{i+1}-y_i)_{0\leq i\leq n-1}$ 
and $Q_n$ is the scale estimator of \cite{CR} which is such that $Q_n \left(x\right)$
is proportional to the first quartile of
\begin{equation*}
\left\lbrace \vert x_i - x_j \vert ; 0\leq i<j\leq n \right\rbrace \; .
\end{equation*}

The asymptotic properties of this estimator are studied in \cite{levy2011robust} for Gaussian stationary processes with either short-range or long-range dependence.
However, as we shall see in the simulation section we can provide an estimator of $\rho^{\star}$ which is more
robust to the presence of change-points than $\widehat{\rho}_{\textrm{MG}}$. The asymptotic properties of 
this novel robust estimator are given in Proposition \ref{prop:correlation}.

\begin{prop}\label{prop:correlation}
Let $y_0,\dots,y_n$ be $(n+1)$ observations satisfying (\ref{eq:modele_new}) 
and let
\begin{equation}\label{eq:est_rho}
\widetilde{\rho}_n = \frac{\left(\med\limits_{0\leq i\leq n-2} \left|y_{i+2}-y_i\right|\right)^2}
{\left(\med\limits_{0\leq i\leq n-1} \left|y_{i+1}-y_i\right|\right)^2}-1 \;,
\end{equation}
where $\med x_i$ denotes the median. Then, $\widetilde{\rho}_n$ satisfies
the following Central Limit Theorem
\begin{equation}\label{eq:tcl_rho_tilde}
\sqrt{n}(\widetilde{\rho}_n - \rho^{\star})\stackrel{d}{\longrightarrow}\mathcal{N}(0,\tilde{\sigma}^2)\;, \textrm{ as } n\to\infty\;,
\end{equation}
where
\begin{equation*}
\tilde{\sigma}^2=\PE[\Psi(\eta_0,\eta_1,\eta_2)^2]
+2\sum_{k\geq 1}\PE\left[\Psi(\eta_0,\eta_1,\eta_2)\Psi(\eta_k,\eta_{k+1},\eta_{k+2})\right]\;,
\end{equation*}
and the function $\Psi$ is defined by
\begin{multline}\label{eq:def_Psi}
\Psi : (x_0,x_1,x_2)\mapsto\\ -\frac{2\sigma^{\star 2}\Phi^{-1}(3/4)}{\varphi\left(\Phi^{-1}(3/4)\right)}
\left[\1_{\left\{|x_2-x_0|\leq \sqrt{2\sigma^{\star 2}}\Phi^{-1}(3/4)\right\}} - \1_{\left\{|x_1-x_0|\leq \sqrt{\frac{2\sigma^{\star 2}}{1+\rho^{\star}}}\Phi^{-1}(3/4)\right\}}\right]\;,
\end{multline}
where $\Phi$ and $\varphi$ denote the cumulative distribution function and the probability distribution function
of a standard Gaussian random variable, respectively.
\end{prop}

The proof of Proposition \ref{prop:correlation} is given in Appendix.

\begin{remark}
Let us now compare the properties of $\widetilde{\rho}_n$ with the properties of $\widehat{\rho}_n(1)$ where $\widehat{\rho}_n(\cdot)$ denotes the classical estimator of the autocorrelation function computed from $y_0,\dots,y_{n}$ defined in \eqref{eq:modele_new} with $m^{\star}=0$. By \cite[Theorem 7.2.1 and Example 7.2.3]{brockwell}, we get that
$$
\sqrt{n}(\widehat{\rho}_n(1)-\rho^{\star})\stackrel{d}{\longrightarrow}\mathcal{N}\left(0,1-{\rho^{\star}}^2\right)\;,\textrm{ as }
n\to\infty\;.$$
From this result, we can see that $\widetilde{\rho}_n$ converges to $\rho^{\star}$ at the same rate as $\widehat{\rho}_n(1)$
except that our result still holds when $m\neq 0$.
\end{remark}

\begin{remark}
{Note that the asymptotic distribution given in (\ref{eq:tcl_rho_tilde}) allows to define a test of $(H_0):$ `$\rho^{\star}=0$' as the asymptotic variance $\widetilde{\sigma}^2$ does not depend on any unknown parameter under $H_0$.}
\end{remark}

\begin{remark}
Since the estimator (\ref{eq:est_rho}) involves differences of the process $(y_i)$ at different instants, 
it can only be used in the case of stable distributions as defined in \cite{feller:1971}. 
Among them, we can quote the Cauchy, L\'evy and Gaussian distributions, where the Gaussian distribution is the only one
to have a finite second order moment. We give some hints in Appendix \ref{subsec:hints} to explain why, in the case of the
Cauchy distribution, taking $\widetilde{\widetilde{\rho}}_n$ defined as follows leads to an accurate estimator of $\rho^\star$:
\begin{equation}\label{eq:rho_cauchy}
\widetilde{\widetilde{\rho}}_n=
\left\{
\begin{array}{cc}
-1+\sqrt{1+\widetilde{\rho}_n},&\textrm{if } \widetilde{\rho}_n\geq 0\;,\\
-\sqrt{1-\sqrt{1+\widetilde{\rho}_n}},&\textrm{if } \widetilde{\rho}_n<0\;,
\end{array}
\right.
\end{equation}
where $\widetilde{\rho}_n$ is defined by (\ref{eq:est_rho}).
Some simulations are also provided in Section \ref{sec:add_simul} to illustrate the finite sample size 
properties of this estimator.
\end{remark}


\section{Change-points and expectations estimation}\label{sec:change-points}

In this section, the number of change-points $m^{\star}$ is assumed to be known. In the sequel, for notational simplicity, $m^{\star}$ will be denoted by $m$. 
Our goal is to estimate both the change-points and the means in model \eqref{eq:modele_new}. A first idea consists in using the following criterion which is based on a quasi-likelihood conditioned on $y_0$ and to minimize it with respect to $\rho$:
\begin{multline*}
\sum_{k=0}^{m} \sum_{i=t_{k}+2}^{t_{k+1}} \left(y_i- \rho y_{i-1}-\delta_k\right)^2 + \sum_{k=1}^m  \left\{\left(y_{t_k +1}-\frac{\delta_k}{1-\rho}\right)- \rho \left(y_{t_k}-\frac{\delta_{k-1}}{1-\rho}\right)\right\}^2\\
 + \left(y_1- \rho y_{0}-\delta_0\right)^2 .
\end{multline*}
Due to the term that involves both $\delta_{k-1}$ and $\delta_k$, this criterion cannot be efficiently minimized. Therefore, we propose to use an alternative criterion defined as follows:
\begin{equation}\label{eq:SSm_def}
SS_m\left(y,\rho, \boldsymbol{\delta},\boldsymbol{t}\right) = \sum_{k=0}^m \sum_{i=t_k +1}^{t_{k+1}}
\left(y_i-\rho y_{i-1}- \delta_k\right)^2\;.
\end{equation}
Note that $SS_m\left(z,\rho, \left(1-\rho\right) \boldsymbol{\mu},\boldsymbol{t}\right)$ corresponds to $-n/2$ times the log-likelihood of the following model maximized with respect to $\sigma$
%
%
\begin{eqnarray}\label{eq:bkw}
z_i - \mu_k^{\star} & = & \rho^{\star} \left(z_{i-1} - \mu_k^{\star}\right) +\epsilon_i\;,\; t_{n,k}^{\star}+1\leq i\leq t_{n,k+1}^{\star}\;,\; 0\leq k\leq m\;,\; 1\leq i\leq n\;,
\end{eqnarray}
and where $z_0$ is a Gaussian random variable with mean $\mu_0^{\star}$ and variance $\sigma^{ \star2}/(1-{\rho^{\star}}^2)$.
In this model, {which is a subset of a model belonging to the class considered in \cite{BKW10},} the expectation changes are not abrupt anymore as in model \eqref{eq:modele_new}.

The parameter $\rho$, involved in each term of \eqref{eq:SSm_def}, is still a problem in order to minimize $SS_m$ wrt $\rho, \boldsymbol{\delta}$ and $\boldsymbol{t}$. This minimization problem is a complex discrete and global optimization problem. Dynamic Programming \cite{AugerLawrence} cannot be used in this case. Only iterative methods are suitable to this minimization problem, without any guarantee to converge to the global minimum.

However, if $\rho$ is replaced by an estimator $\overline{\rho}_n$, $SS_m (y,\overline{\rho}_n,\boldsymbol{\delta},\boldsymbol{t})$ can be minimized wrt $\boldsymbol{\delta}$ and $\boldsymbol{t}$ by Dynamic Programming. Proposition \ref{Prop:Segment2} gives asymptotic results for the estimators resulting from this method.

\begin{prop}\label{Prop:Segment}
Let $ z = \left(z_0 ,\dots , z_n \right) $ be a finite sequence of real-valued random variables satisfying \eqref{eq:bkw} and $\left(\overline{\rho}_n\right)$ a sequence of real-valued random variables. Let $\boldsymbol{\widehat{\delta}}_n (z , \overline{\rho}_n)$ and $\boldsymbol{\widehat{t}}_n(z, \overline{\rho}_n)$ be defined by 
\begin{eqnarray}
\left(\boldsymbol{\widehat{\delta}}_n(z, \overline{\rho}_n), \boldsymbol{\widehat{t}}_n(z, \overline{\rho}_n)\right) & = & \underset{\left(\boldsymbol{\delta},\boldsymbol{t}\right)\in \mathbb{R}^{m+1}\times \mathcal{A}_{n,m}}{\argmin} SS_m\left(z,\overline{\rho}_n , \boldsymbol{\delta},\boldsymbol{t}\right) \ \textit{,}\label{eq:delta_n,t_n}\\
\boldsymbol{\widehat{\tau}}_n(z, \overline{\rho}_n) & = & \frac{1}{n} \boldsymbol{\widehat{t}}_n(z, \overline{\rho}_n),\label{eq:tau_n}
\end{eqnarray}
\noindent where
\begin{equation}\label{eq:Anm}
\mathcal{A}_{n,m} = \\
\left\lbrace \left(t_0,\dots ,t_{m+1}\right);t_0=0<\dots <t_{m+1}=n, \forall k=1,\dots ,m+1, t_k-t_{k-1}\geq\Delta_n \right\rbrace
\end{equation}
and where $\left(\Delta_n\right)$ is a real sequence such that
{$n^{-1}\Delta_n \underset{n\to\infty}{\longrightarrow} 0$ and $n^{-\alpha}\Delta_n \underset{n\to\infty}{\longrightarrow} +\infty$ with $\alpha>0$.}
Assume that
\begin{equation}\label{eq:hypRhoRate}
\left(\overline{\rho}_n-\rho^{\star}\right) = O_P\left(n^{-1/2}\right),
\end{equation}
as $n$ tends to infinity. 
Then, 
\begin{equation*}
\| \boldsymbol{\widehat{\tau}}_n(z, \overline{\rho}_n) - \boldsymbol{\tau^{\star}} \| = O_P\left(n^{-1}\right), \qquad \| \boldsymbol{\widehat{\delta}}_n(z, \overline{\rho}_n) - \boldsymbol{\delta^{\star}} \|= O_P\left(n^{-1/2}\right),
\end{equation*}
where $\|\cdot\|$ is the Euclidian norm.\\
The results still hold if the $\epsilon_i$'s are only assumed to be centered and to have a finite second order moment.
\end{prop}




\begin{prop}\label{Prop:Segment2}
The results of Proposition \ref{Prop:Segment} still hold under the same assumptions 
when $z$ is replaced with $y$
satisfying \eqref{eq:modele_new}.\\
The results still hold if the $\epsilon_i$'s are only assumed to be centered and to have a finite second order moment.
\end{prop}

The proofs of Propositions \ref{Prop:Segment} and \ref{Prop:Segment2} are given in Sections \ref{subsec:prop:Segment} and 
\ref{subsec:prop:Segment2}, respectively. 
Note that the estimators defined in these propositions have the same asymptotic properties 
as those of the estimators proposed by \cite{LM}. 
In the Gaussian framework, the estimator $\widetilde{\rho}_n$ defined in Section \ref{sec:correlation} satisfies the same properties as $\overline{\rho}_n$ and can thus be used in the criterion $SS_m$ for providing consistent estimators of the change-points and of the means.


\section{Selecting the number of change-points} \label{sec:selection}

We now consider the selection of the number of change-points. We first propose a penalized contrast criterion, which we prove to be consistent in the non-Gaussian case. The penalty has a general form, which needs to be specified for a practical use. Therefore, we also derive an adaptation of the modified BIC criterion proposed by \cite{zs} in the Gaussian context. This criterion does not rely on any tuning parameter and has been shown to be efficient in practical cases (see \cite{picard2011joint}).

\subsection{Consistent model selection criterion}\label{subsec:beta}

We propose to select the number of change-points $m$ as follows 
\begin{equation} \label{Eq:ConsistentCrit}
\widehat{m} = \underset{ 0\leq m \leq m_{\max}}{\argmin} \frac1n SS_m(z, \overline{\rho}_n) + \beta_n m
\end{equation}
where $m_{\max}\geq m^{\star}$, $\left(\beta_n\right)_{n\geq 1}$ is a sequence of positive real numbers, $\overline{\rho}_n$ satisfies the assumptions of Proposition \ref{Prop:Segment} and
\begin{equation} \label{Eq:SSm}
SS_m(z, \rho) = \min_{\boldsymbol{\delta}, \boldsymbol{t} \in \mathcal{A}_{n,m}} SS_m(z, \rho, \boldsymbol{\delta}, \boldsymbol{t}) \; ,
\end{equation}
$\mathcal{A}_{n,m}$ being defined in~(\ref{eq:Anm}).

\begin{prop}\label{Prop:SelBeta}
Under the assumptions of Proposition \ref{Prop:Segment},  and if
\begin{eqnarray*}
\beta_n \underset{n\to\infty}{\longrightarrow} 0 , \quad n^{1/2}\beta_n\underset{n\to\infty}{\longrightarrow} +\infty , \quad \Delta_n\beta_n\underset{n\rightarrow\infty}{\longrightarrow} +\infty \; ,
\end{eqnarray*}
where $\Delta_n$ is defined in Proposition \ref{Prop:Segment}, $\widehat{m}$ defined by \eqref{Eq:ConsistentCrit} converges in probability to $m^{\star}$.\\
The result still holds if the $\epsilon_i$'s are only assumed to be independent, centered and to have a finite second order moment.
\end{prop}

\begin{prop}\label{Prop:SelBeta2}
The result of Proposition \ref{Prop:SelBeta} still holds under the same assumptions 
when $z$ is replaced by $y$ satisfying \eqref{eq:modele_new}.\\
The result still holds if the $\epsilon_i$'s are only assumed to be independent, centered and to have a finite second order moment.
\end{prop}

The proofs of Propositions \ref{Prop:SelBeta} and \ref{Prop:SelBeta2} are given in Sections \ref{subsec:bicbardet} and \ref{subsec:mbicbardet_bis}, respectively.

\begin{remark}
If $\beta_n = n^{-\beta}$, the assumptions of Propositions \ref{Prop:SelBeta} and \ref{Prop:SelBeta2} are fulfilled if and only if $0<\beta < \min \left(\alpha , 1/2 \right)$, where $\alpha$ is defined in Proposition \ref{Prop:Segment}. $\alpha$ stands for the usual bound for the control of the minimal segment length (see \cite{LM}). The $1/2$ bound is the price to pay for the estimation of $\rho^{\star}$.
\end{remark}

\subsection{Modified BIC criterion}
\cite {zs} proposed a modified Bayesian Information Criterion (mBIC) to select the number $m$ of change-points in the particular case of segmentation of an independent Gaussian process $x$. This criterion is defined in a Bayesian context in which a non informative prior is set for the number of segments $m$. 
mBIC is derived from an $O_P(1)$ approximation of the Bayes factor between models with $m$ and $0$ change-points, respectively.
The mBIC selection procedure consists in choosing the number of change-points as:
\begin{equation}\label{eq:Criterion_formula}
\widehat{m} = \argmax_m C_m(x, 0)
\end{equation}
where the criterion $C_m(y, \rho)$ is defined for a process $y$ as
\begin{multline*}
  C_m(y, \rho) = \\
- \frac{n-m+1}{2} \log SS_m(y, \rho) + \log
\Gamma\left(\frac{n-m+1}{2}\right) -\frac{1}{2} \sum_{k=0}^m \log n_k(\boldsymbol{\widehat{t}}(y, \rho)) - m \log n ,
\end{multline*}
where $\Gamma$ is the usual Gamma function. In the latter equation
\begin{equation}\label{eq:nk}
n_k(\boldsymbol{\widehat{t}}(y, \rho))=\widehat{t}_{k+1}(y, \rho)-\widehat{t}_{k}(y, \rho),
\end{equation}
where $\boldsymbol{\widehat{t}}(y, \rho)=(\widehat{t}_1(y, \rho), \dots, \widehat{t}_m(y, \rho))$ is defined as $\boldsymbol{\widehat{t}}(y, \rho)= \underset{\boldsymbol{t}\in\mathcal{A}_{n,m}}{\argmin} \min_\delta SS_m(y, \rho, \delta, \boldsymbol{t})$.

Note that, in model \eqref{eq:bkw}, the criterion could be directly applied to the decorrelated series $v^{\star} = \left(v^{\star}_i\right)_{1\leq i\leq n} = \left(y_i - \rho^{\star} y_{i-1}\right)_{1\leq i\leq n}$ since
$$
C_m(y, \rho^{\star}) = C_m(v^{\star}, 0).
$$ 
We propose to use the same selection criterion, replacing $\rho^{\star}$ by some relevant estimator $\overline{\rho}_n$. 
The following two propositions show that this plug-in approach result in the same asymptotic properties under both Model \eqref{eq:bkw} and \eqref{eq:modele_new}. 

\begin{prop} \label{Prop:mBICBardet}
For any positive $m$, for a process $z$ satisfying \eqref{eq:bkw} and under the assumptions of Proposition \ref{Prop:Segment}, we have
\begin{eqnarray*}
C_m(z, \overline{\rho}_n) & = & C_m(z, \rho^{\star}) + O_P(1),\textrm{ as } n\to\infty\;.
\end{eqnarray*}
\end{prop}

\begin{prop} \label{Prop:mBIC}
For any positive $m$, for a process $y$ satisfying \eqref{eq:modele_new} and under the assumptions of Proposition \ref{Prop:Segment2} , we have
\begin{eqnarray*}
C_m(y, \overline{\rho}_n) & = & C_m(y, \rho^{\star}) + O_P(1),\textrm{ as } n\to\infty\;.
\end{eqnarray*}
\end{prop}

The proofs of Propositions \ref{Prop:mBICBardet} and \ref{Prop:mBIC} are given in Appendix.

In practice, we propose to take $\overline{\rho}_n = \widetilde{\rho}_n$ 
which satisfies the condition of Proposition \ref{Prop:mBIC} 
to estimate the number of segments by
\begin{eqnarray} \label{Eq:BIC}
\widehat{m} & = & \argmax_m \left[- \left(\frac{n-m+1}{2} \right) \log SS_m(y, \widetilde{\rho}_n) + \log
\Gamma\left(\frac{n-m+1}{2}\right) \right. \nonumber \\
  & & \left. -\frac{1}{2} \sum_{k=0}^m \log n_k(\boldsymbol{\widehat{t}}(y, \widetilde{\rho}_n)) - m \log n \right],
\end{eqnarray}
where $SS_m(\cdot,\cdot)$ and $n_k(\cdot,\cdot)$ are defined in \eqref{Eq:SSm} and \eqref{eq:nk}, respectively.

\begin{remark}
Since the definition of the original mBIC criterion is intrinsically related to normality, we did not study precisely the quality of our approximation without the normality assumption.
\end{remark}


\section{Numerical experiments}\label{sec:simul}


\subsection{Practical implementation}\label{sec:post-proc}


Our decorrelation procedure introduces spurious change-points in the
series, at distance $1$ of the true change-points (see Figure
\ref{fig:exp_vs_exp_diff}, top). Since these artefacts may affect
our procedure, we propose a post-processing to the estimated
change-points $\boldsymbol{\widehat{t}_n}$, which consists in
removing segments of length 1:
\begin{equation}\label{eq:PT}
PP\left( \boldsymbol{\widehat{t}}_n \right) = \left\lbrace \widehat{t}_{n,k}\in \boldsymbol{\widehat{t}}_n \right\rbrace \setminus \left\lbrace \widehat{t}_{n,i} \textit{ such that } \widehat{t}_{n,i} = \widehat{t}_{n,i-1} +1 \textit{ and } \widehat{t}_{n,i+1} \neq \widehat{t}_{n,i} +1 \right\rbrace \; .
\end{equation}
This post-processing results in a smaller number of change-points. Figure \ref{fig:exp_vs_exp_diff} summarizes the whole processing.

\begin{figure}
\centering
\includegraphics[width=.75\textwidth]{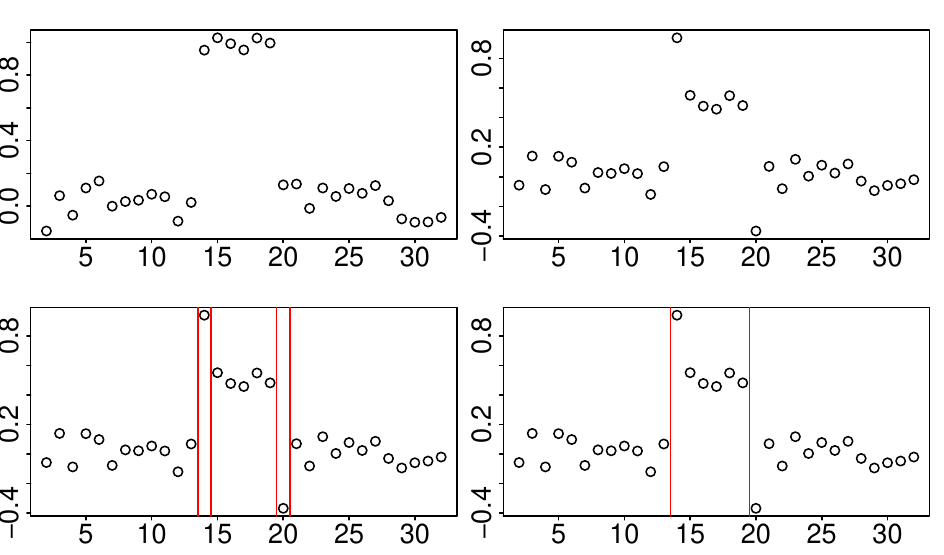}
\caption{\footnotesize{Top left: a series around two changes, with $\rho^{\star}=0.5$. Top right: the decorrelated series in the same region. Bottom left: before post-processing, two pairs of adjacent change-points are found. Bottom right: post-processing removes the last change-point of each pair of adjacent ones.}}
\label{fig:exp_vs_exp_diff}
\end{figure}

In practice, it may also be useful to have some guidance on how to check that the assumptions underpinning our approach are satisfied for a given data set. A possible approach is to subtract the estimated piecewise constant function from the original series. If the model is the expected one, this new series should be a realization of an AR(1) Gaussian process. Hence, the residuals built by decorrelation of this series should be Gaussian and independent. One way to check this is to perform a gaussianity test and a Portmanteau test on this series of residuals.


\subsection{Simulation design}\label{sec:design}

To assess the performance of the proposed method, we used a
simulation design inspired from the one conceived by \cite{simul}.
We considered series of length $n \in \{100, 200, 400, 800, 1600\}$
with autocorrelation at lag 1, denoted by $\rho^{\star}$, ranging from
$-.9$ to $.9$ (by steps of $.1$) and residual standard deviation
$\sigma^{\star}$ between $.1$ and $.6$ (by steps of $.1$). All series were
affected by $m^{\star} = 6$ change-points located at
fractions $1/6 \pm 1/36, 3/6 \pm 2/36, 5/6 \pm 3/36$
of their length. Each combination was replicated $S=100$ times.
The mean within each segment alternates between 0 and 1, starting with $\mu_1 = 0$.

\paragraph{\textbf{Estimation of $\rho^{\star}$.}}

For each generated series, two different estimates $\overline{\rho}_n$ of $\rho^{\star}$ were computed: the original estimate $\overline{\rho}_n = \widehat{\rho}_{\textrm{MG}}$ proposed by \cite{MG} and our revised version $\overline{\rho}_n = \widetilde{\rho}_n$.
We carried the same study on series with no change-point (centered series). 

\paragraph{\textbf{Estimation of the segmentation parameters.}}
%

For each generated series, we estimated the change-point locations
$\widehat{\boldsymbol{\tau}}_n(y,
\overline{\rho}_n)$ using Proposition \ref{Prop:Segment} for each
$m$ from $1$ to $m_{\max} =75$ and with different choices of $\overline{\rho}_n$:
$\widetilde{\rho}_n$ (our estimator), $\rho^{\star}$ (the true value) and
zero (which does not take into account for the autocorrelation). For
each choice of $\overline{\rho}_n$, we then selected the number of
change-points $\hat{m}$ using \eqref{Eq:BIC}. Actually, the last
choice $\overline{\rho}_n=0$ corresponds to the classical
least-squares framework. In addition, we shall also use the
post-processing described in Section \ref{sec:post-proc} for the
cases where $\overline{\rho}_n=\widetilde{\rho}_n$ and $\rho^{\star}$. \\
To study the quality of the proposed model
selection criterion, we computed the distribution of $\widehat{m}$
for each estimate $\overline{\rho}_n \in \{\widetilde{\rho}_n,
\rho^{\star}, 0\}$ with post-processing or not for the first two estimates
of $\rho^{\star}$.

In order to assess the performance of the estimation of the
change-point locations,  we computed the Hausdorff distance defined
in the segmentation framework as follows, see \cite{boysen:2009} and \cite{harchaoui:2010}:


\begin{equation}\label{eq:hausdorff}
d\left(\boldsymbol{\tau}^{\star} , \widehat{\boldsymbol{\tau}}_n\left(y,
\overline{\rho}_n\right)\right) = \max\left( d_1
\left(\boldsymbol{\tau}^{\star}, \widehat{\boldsymbol{\tau}}_n\left(y,
\overline{\rho}_n\right)\right),d_2 \left(\boldsymbol{\tau}^{\star},
\widehat{\boldsymbol{\tau}}_n\left(y, \overline{\rho}_n\right)\right)\right)\;,
\end{equation}
where
\begin{eqnarray}
d_1 \left(\boldsymbol{a},\boldsymbol{b}\right) & = & \underset{b\in\boldsymbol{b}}{\sup}
\underset{a\in\boldsymbol{a}}{\inf} \left\vert a - b \right\vert
\label{eq:hausd_1},  \\
 \text{and \qquad} d_2
\left(\boldsymbol{a},\boldsymbol{b}\right) & = & d_1
\left(\boldsymbol{b},\boldsymbol{a}\right).\label{eq:hausd_2}
\end{eqnarray}
$d_1$ close to zero means that an estimated change-point is likely to be close to a true change-point. A small value of $d_2$ means that a true change-point is likely to be close to each estimated change-point. A perfect segmentation results in both null $d_1$ and $d_2$. Over-segmentation results in a small $d_1$ and a large $d_2$. Under-segmentation results in a large $d_1$ and a small $d_2$, provided that the estimated change-points are correctly located.

\subsection{Results}

\paragraph{\textbf{Estimation of $\rho^{\star}$.}}

In Figure \ref{fig:rho_no_break}, we compare the performance of our
robust estimator of $\rho^{\star}$: $\widetilde{\rho}_n$ with the ones of
the estimator $\widehat{\rho}_{\textrm{MG}}$ in the case where there
are no change-points in the observations. More precisely, in
this case, the observations $y$ are generated under the
model (\ref{eq:modele_new}) with $\mu_k^{\star}=0$, for all $k$.  We
observe that the estimator proposed by \cite{MG} performs better than our robust estimator.
However, it is not the case anymore in the presence of change-points
in the data as we can see in Figure \ref{fig:rho_with_break}. In the
latter case, our robust estimator $\widetilde{\rho}_n$ outperforms
the estimator $\widehat{\rho}_{\textrm{MG}}$ for almost all values
of $\rho^{\star}$.





\begin{figure}[!h]
\centering
\includegraphics[width=0.7\textwidth]{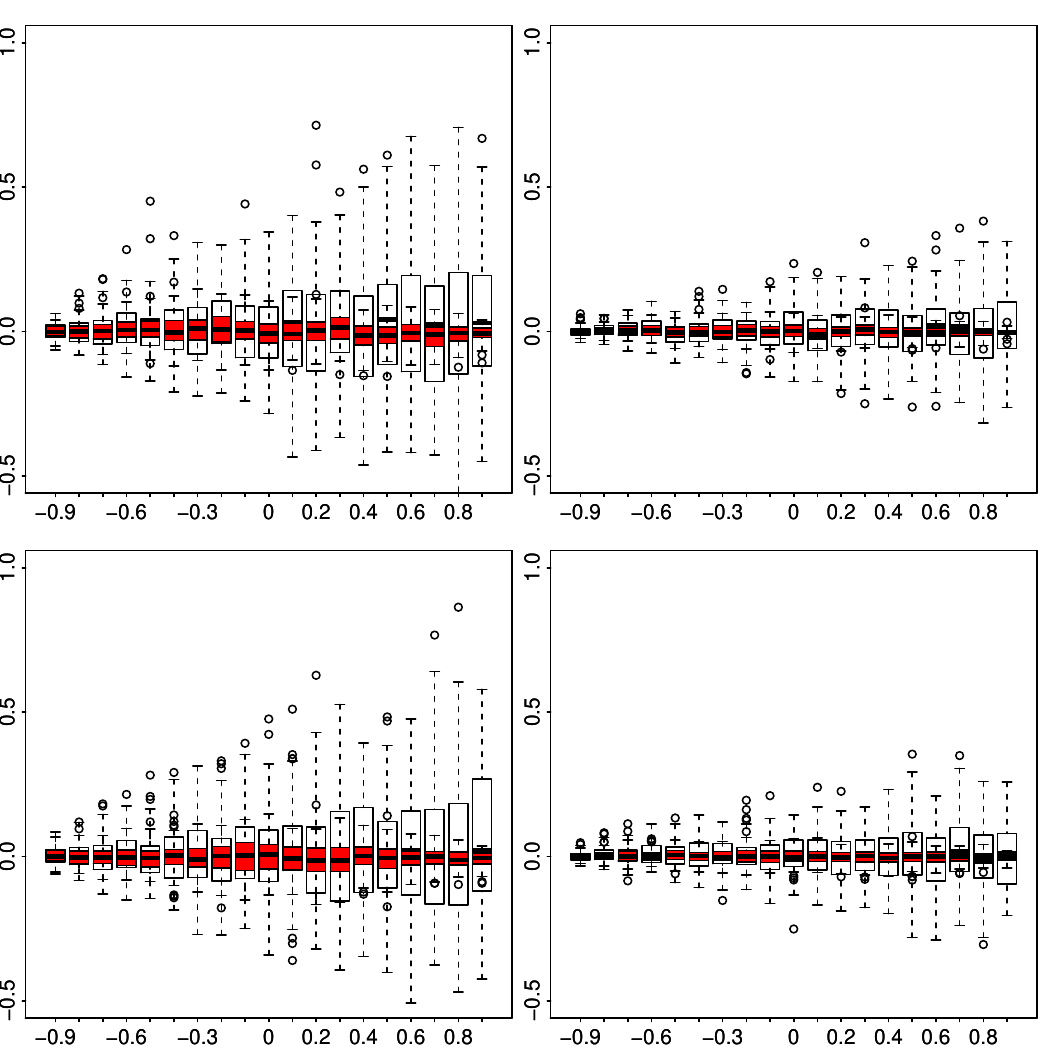}
\caption{\footnotesize{Boxplots of
$\widehat{\rho}_{\textrm{MG}}-\rho^{\star}$ in red and
$\widetilde{\rho}_n-\rho^{\star}$ in black for different values of
$\rho^{\star}$ in the case where there are no change-points in the data
with $n=400$ (plots on the left), $n=1600$ (plots on the right),
$\sigma^{\star}=0.2$ (top) and $\sigma^{\star}=0.6$ (bottom).}}
\label{fig:rho_no_break}
\end{figure}
\begin{figure}[!h]
\centering
\includegraphics[width=0.7\textwidth]{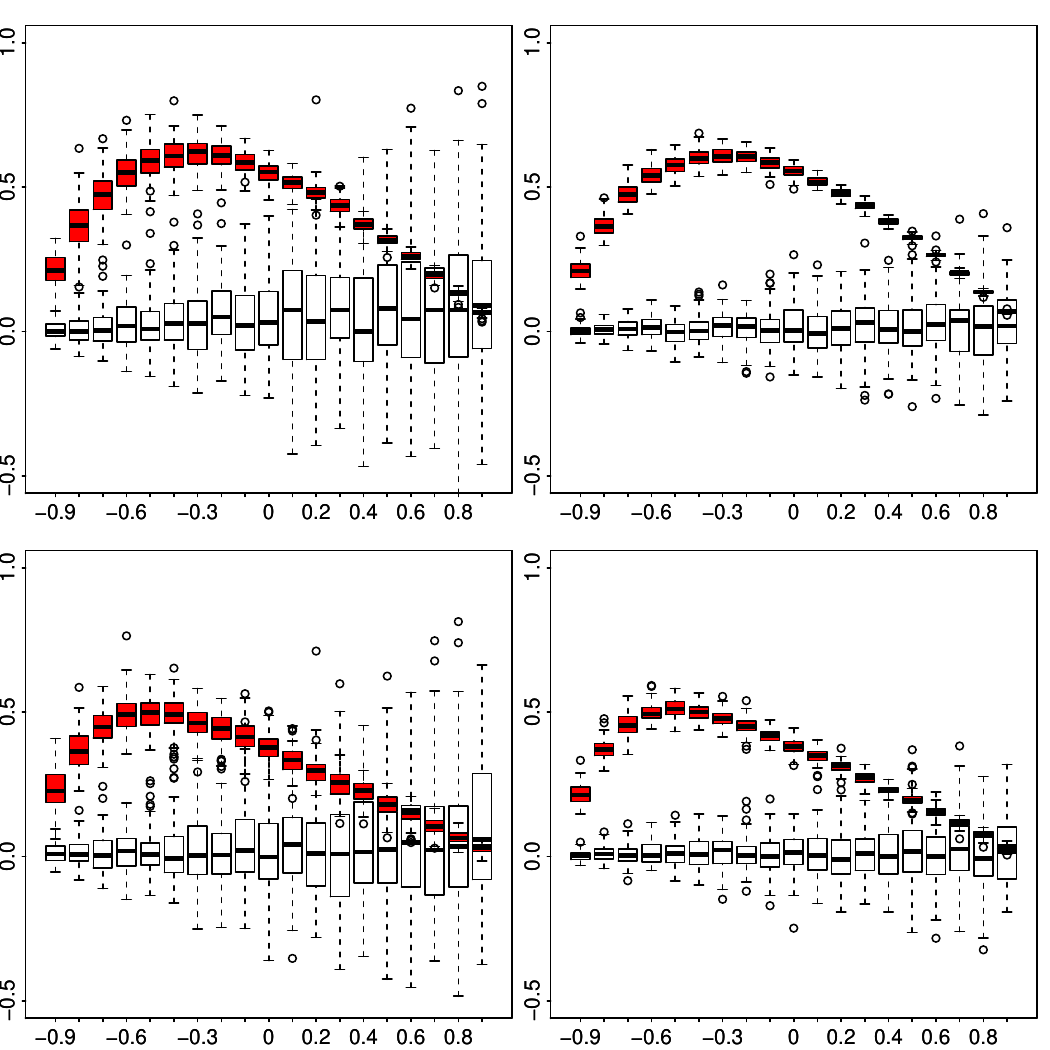}
\caption{\footnotesize{Boxplots of
$\widehat{\rho}_{\textrm{MG}}-\rho^{\star}$ in red and
$\widetilde{\rho}_n-\rho^{\star}$ in black for different values of
$\rho^{\star}$ in the case where there are change-points in the data with
$n=400$ (plots on the left), $n=1600$ (plots on the right),
$\sigma^{\star}=0.2$ (top) and $\sigma^{\star}=0.6$ (bottom).}}
\label{fig:rho_with_break}
\end{figure}

\paragraph{\textbf{Model selection.}}

In Figures \ref{fig:Kchap_sig01} and \ref{fig:Kchap_sig05}, we
compare the estimated number of change-points $\widehat{m}$ in
two different configurations of signal-to-noise ratio
($\sigma^{\star}=0.1$ and $\sigma^{\star}=0.5$) and with three different values
of $\rho^{\star}$ ($\rho^{\star}=0.3$, $0.6$ and $0.8$). In this figures, the
notation LS, Robust and Oracle correspond to the cases where
$\overline{\rho}_n=0$, $\overline{\rho}_n=\tilde{\rho}_n$ and
$\overline{\rho}_n=\rho^{\star}$, respectively. Moreover, we use the
notation -P when the post-processing described in Section \ref{sec:post-proc} is used.  In the situations where $\sigma^{\star}$ and
$\rho^{\star}$ are small, all the methods provide an accurate estimation
of the number of change-points. In the other cases,
LS tends to strongly overestimate the number of change-points.
Robust and Oracle tend to select twice the true number of change-points due to the artifactual
presence of change-points in the decorrelated series as explained in Section \ref{sec:post-proc}. This is corrected by the post-processing and Robust-P provides the correct number of change-points in most of the considered configurations.
Moreover, we also observe that the performance of Robust and
Robust-P are similar to these of Oracle and Oracle-P: the robust
decorrelation procedure we propose performs as well as if $\rho^{\star}$
was known for $n=1600$. It has to be noted that the post-processing would not
improve the performance on LS {so we did not considered it}.

\begin{figure}[!h]
\includegraphics[width=0.32\textwidth,height=5cm]{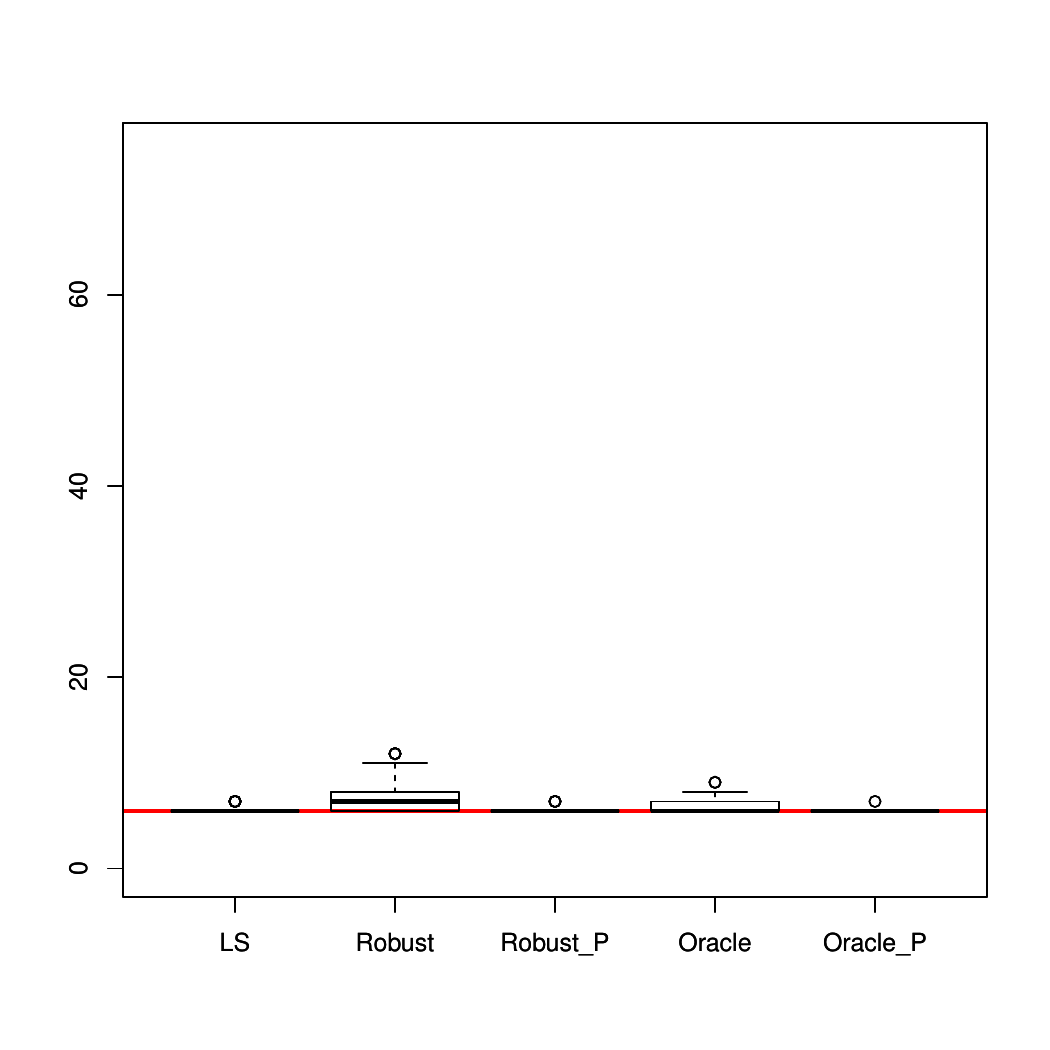}
\includegraphics[width=0.32\textwidth,height=5cm]{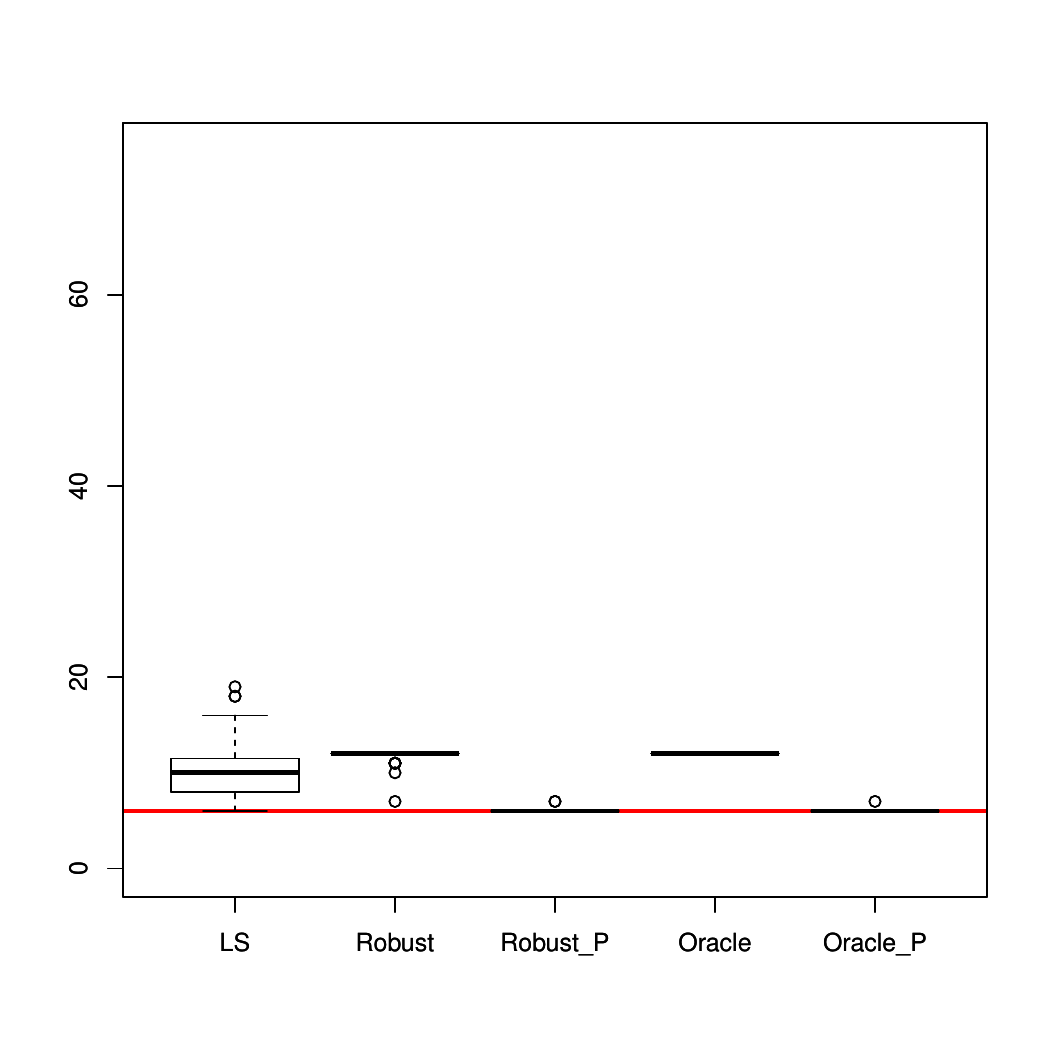}
\includegraphics[width=0.32\textwidth,height=5cm]{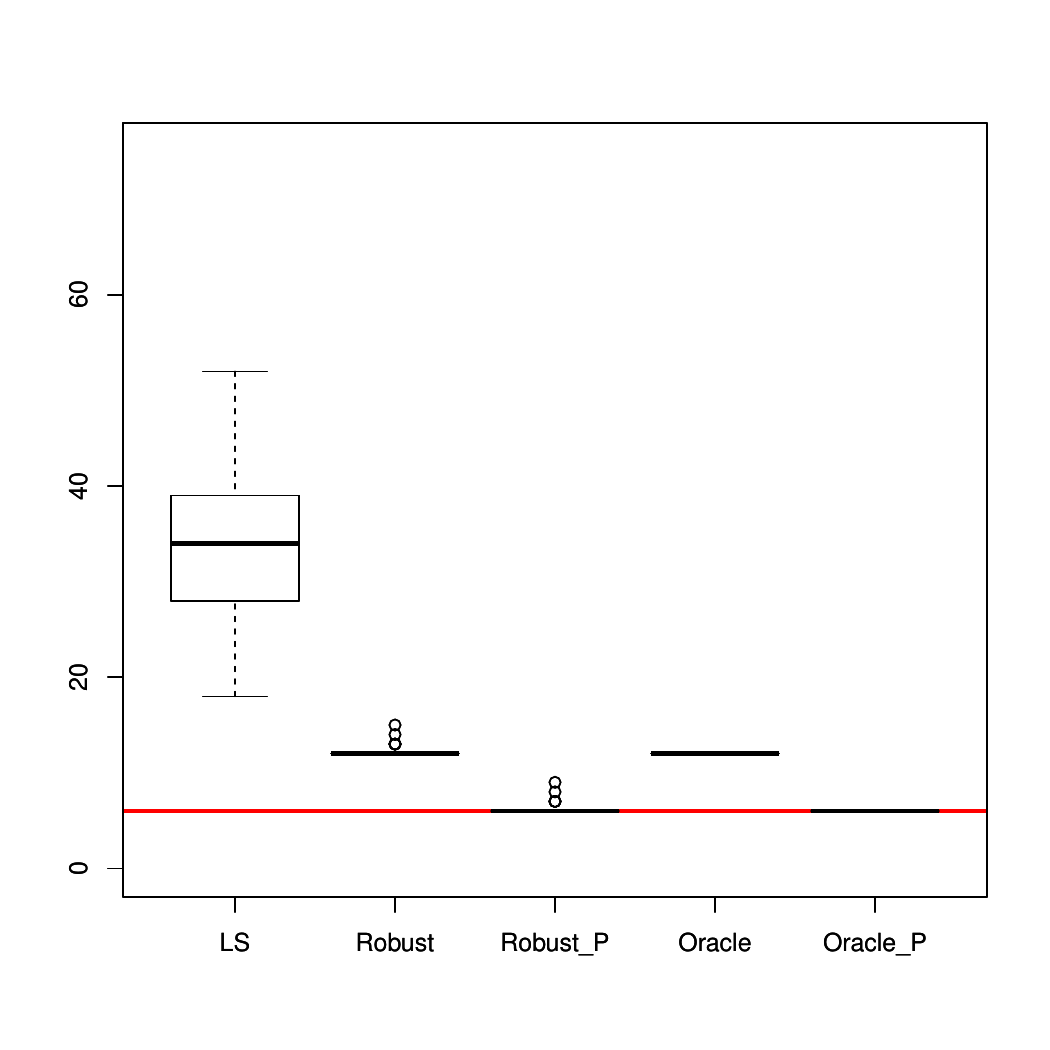}
\caption{\footnotesize{Boxplots for the estimated number of change-points for $n=1600$ when $\overline{\rho}_n=0$ (LS), $\overline{\rho}_n=\tilde{\rho}_n$ (Robust and
Robust-P with post-processing) and $\overline{\rho}_n=\rho^{\star}$ (Oracle and Oracle-P with post-processing)
with $\sigma^{\star}=0.1$ and $\rho^{\star}=0.3$ (left),  $\rho^{\star}=0.6$ (middle) and $\rho^{\star}=0.8$ (right). The true number of change-points is equal to 6 (red horizontal line).}}
\label{fig:Kchap_sig01}
\end{figure}

\begin{figure}[!h]
\includegraphics[width=0.32\textwidth,height=5cm]{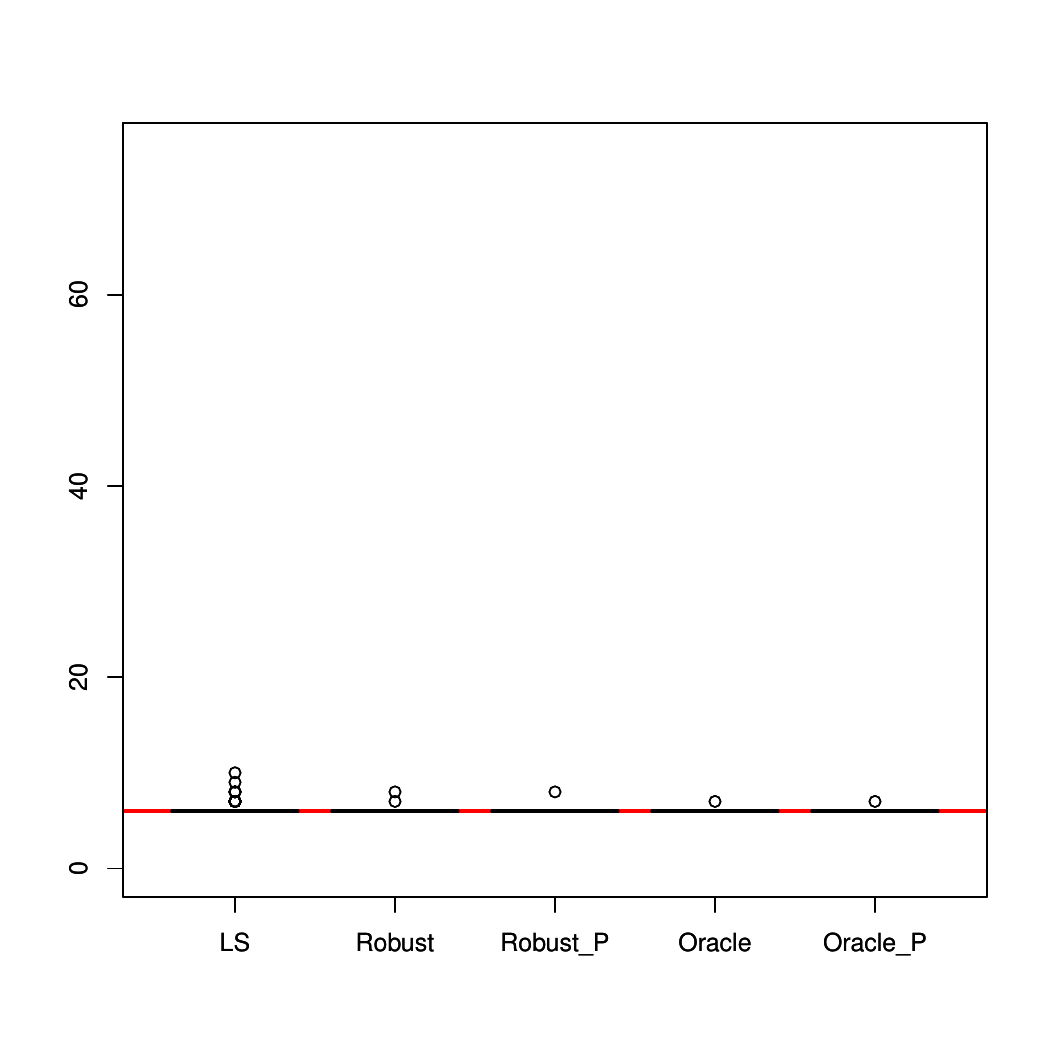}
\includegraphics[width=0.32\textwidth,height=5cm]{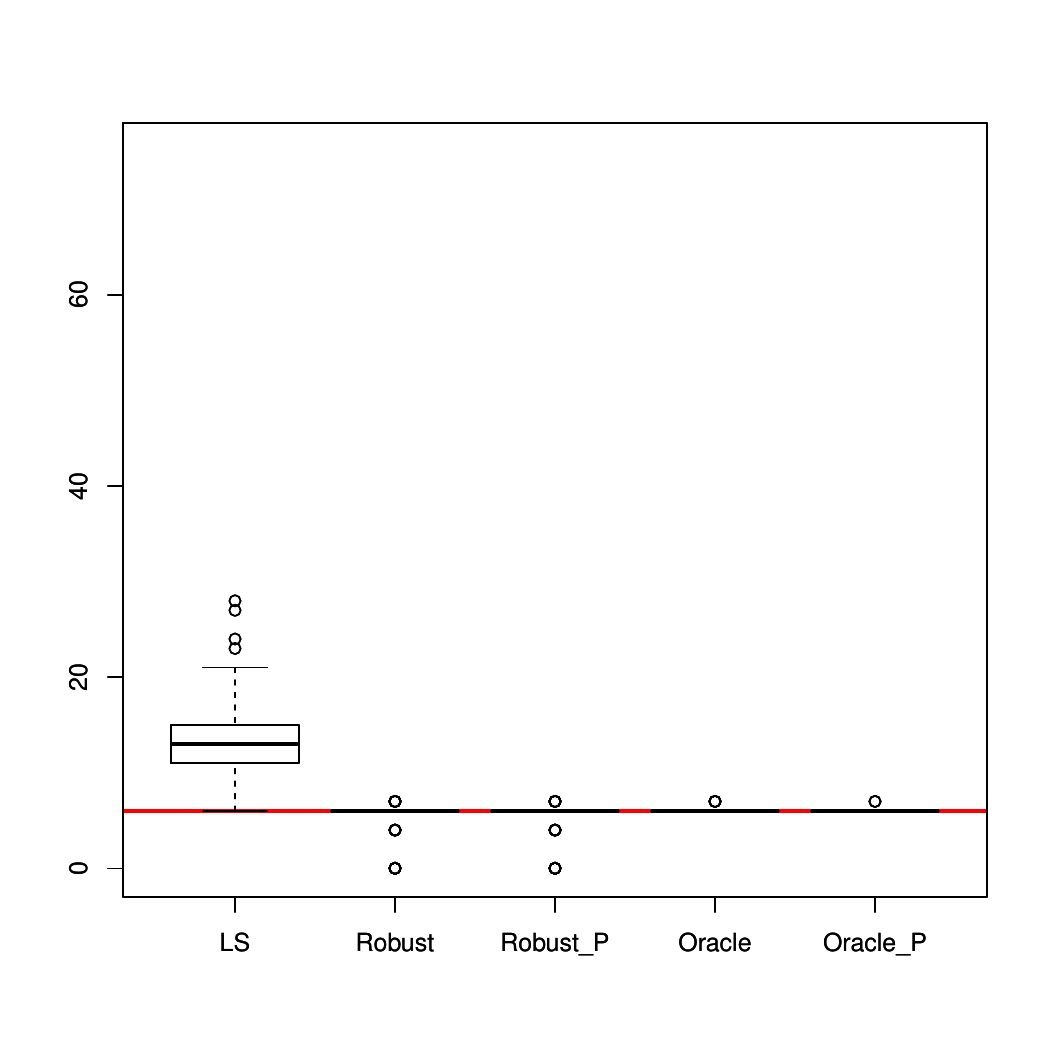}
\includegraphics[width=0.32\textwidth,height=5cm]{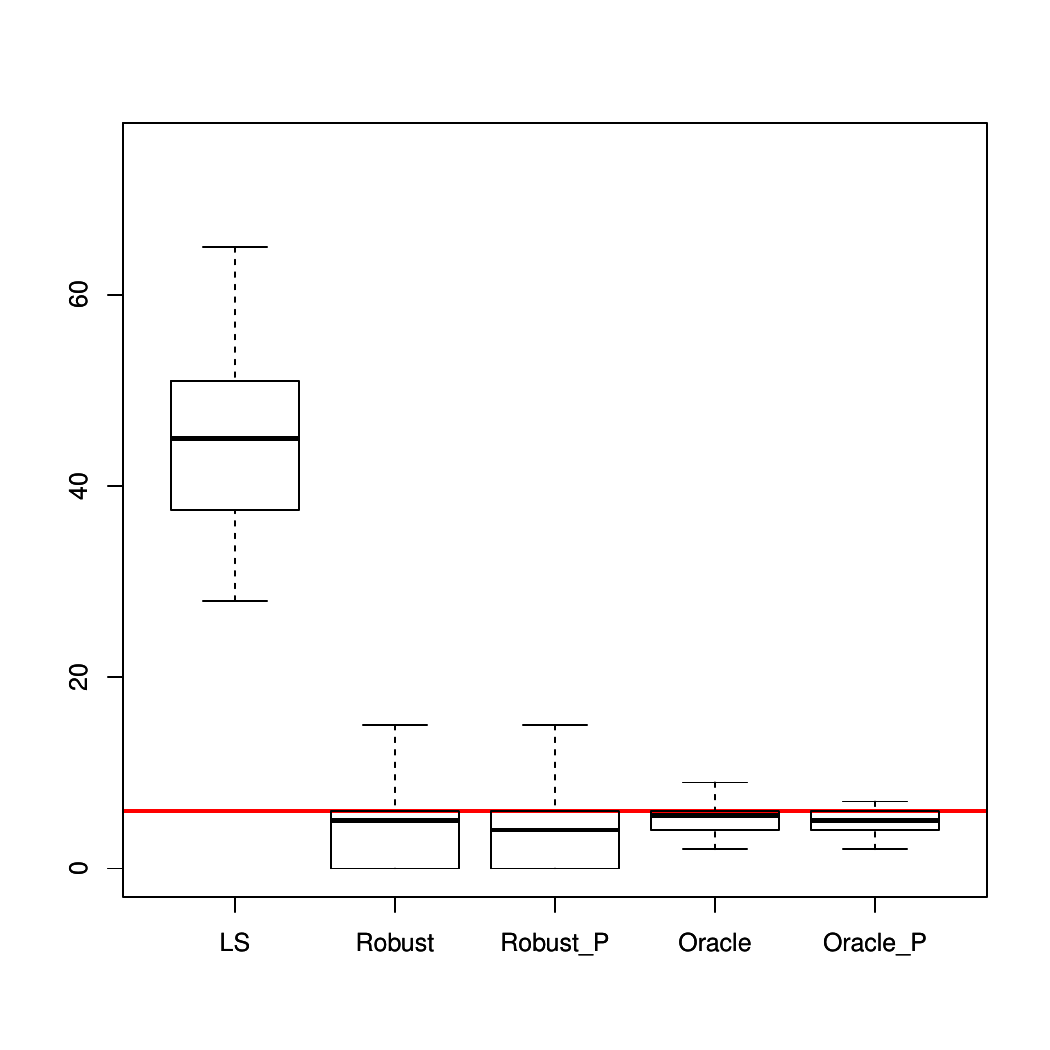}
\caption{\footnotesize{Boxplots for the estimated number of change-points for $n=1600$ when $\overline{\rho}_n=0$ (LS), $\overline{\rho}_n=\tilde{\rho}_n$ (Robust and
Robust-P with post-processing) and $\overline{\rho}_n=\rho^{\star}$ (Oracle and Oracle-P with post-processing)
with $\sigma^{\star}=0.5$ and $\rho^{\star}=0.3$ (left),  $\rho^{\star}=0.3$ (middle) and $\rho^{\star}=0.8$ (right). The true number of change-points is equal to 6 (red horizontal line).}}
\label{fig:Kchap_sig05}
\end{figure}

\paragraph{\textbf{Change-point locations.}}
In Figures \ref{fig:hausd_1_sig05} and \ref{fig:hausd_2_sig05} are
displayed the boxplots of the two parts $d_1$ and
$d_2$ of the Hausdorff distance
defined in (\ref{eq:hausd_1}) and (\ref{eq:hausd_2}), respectively for different values of $\rho^{\star}$ when
$\sigma^{\star}=0.5$. $d_2$ is displayed in Figure \ref{fig:hausd_2_sig01} for $\sigma^{\star}=0.1$; for this value of
$\sigma^{\star}$, $d_1$ was found null for all methods and all values of $\rho^{\star}$.


When the noise is small ($\sigma^{\star} = 0.1$), the robust procedure we propose performs well for the whole range of correlation. On the contrary, the performance of LS are deprecated when the correlation increases, whereas these of LS$^{\star}$ still provide accurate change-point locations. This shows that the least-square approach only fails because it turns to overestimate the number of change-points. 
This is all the more true for LS when the variance of the noise is
large ($\sigma^{\star} = 0.5$). When the problem
gets difficult (both $\sigma^{\star}$ and $\rho^{\star}$ large), our robust
procedure tends to underestimate the number of
change-points (which was expected) and the
estimated change-points are close to true ones.



\begin{figure}[!h]
\includegraphics[width=0.32\textwidth,height=5cm]{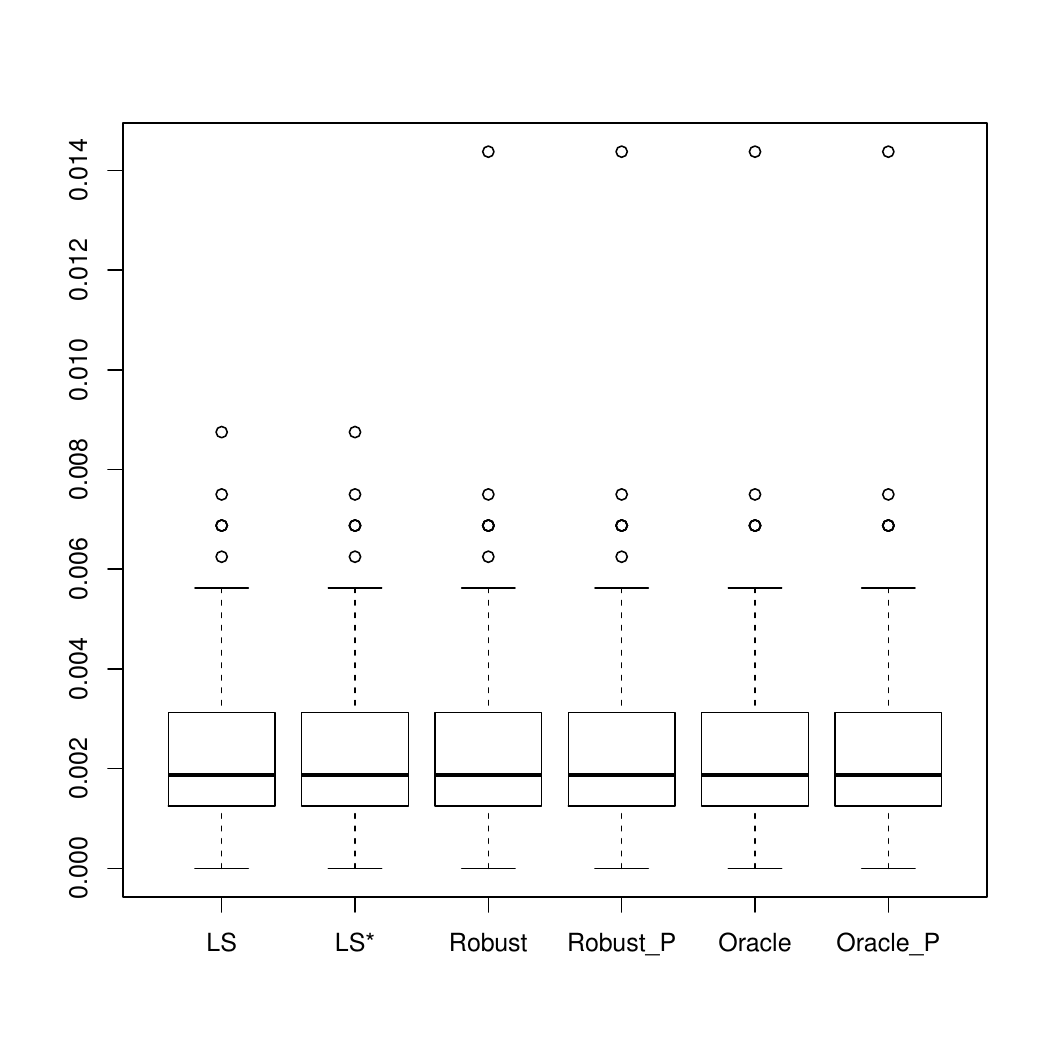}
\includegraphics[width=0.32\textwidth,height=5cm]{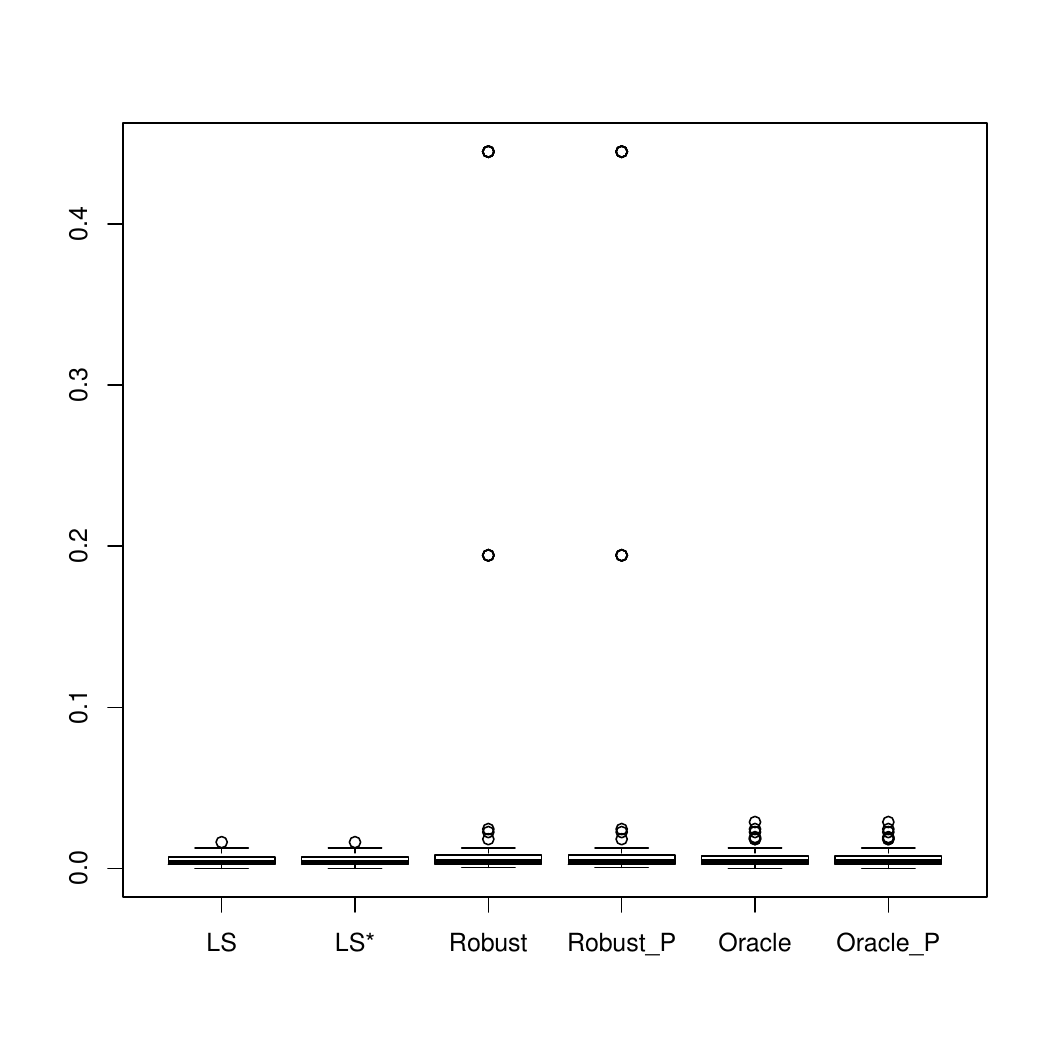}
\includegraphics[width=0.32\textwidth,height=5cm]{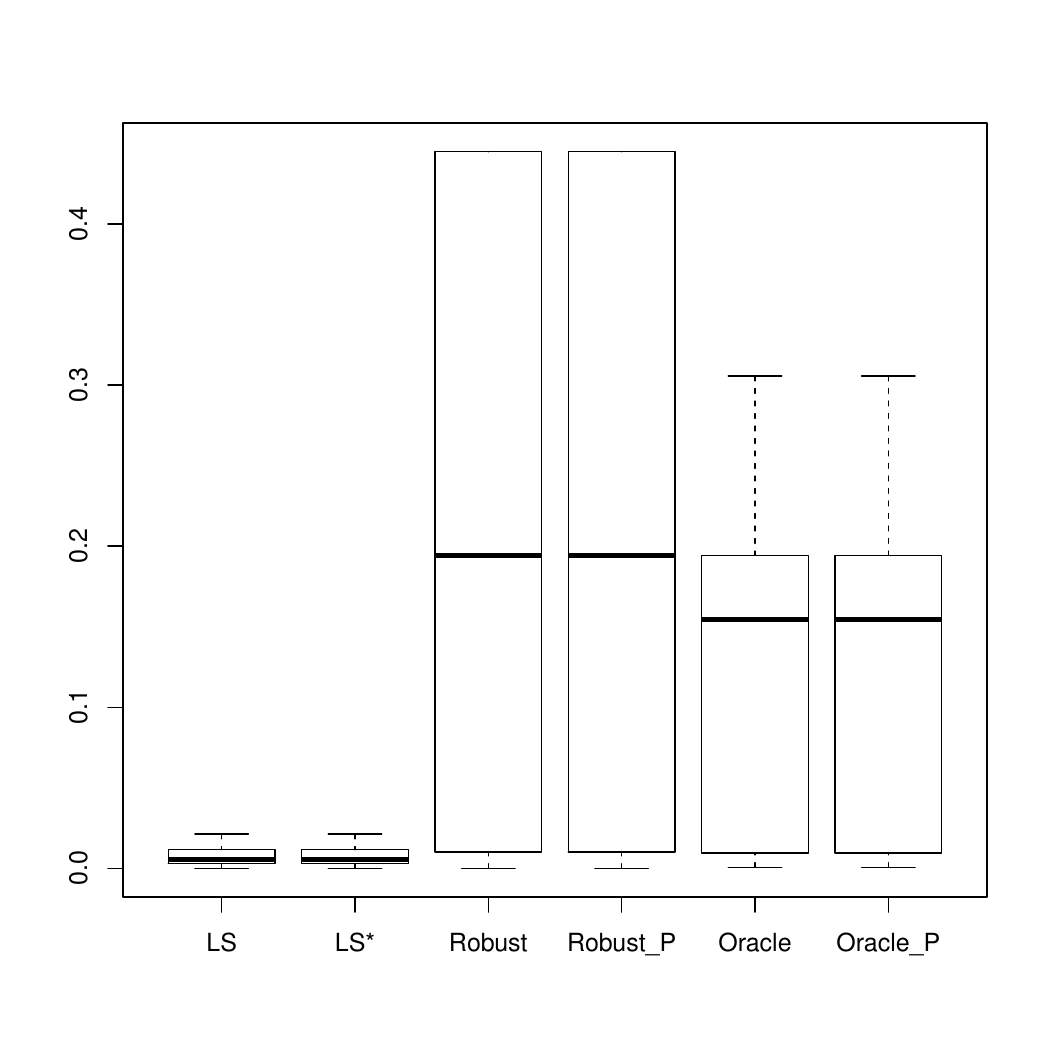}
\caption{\footnotesize{Boxplots for the first part of the Hausdorff distance ($d_1$)  for $n=1600$ when $\overline{\rho}_n=0$ (LS and LS* when the true number of
change-points is known), $\overline{\rho}_n=\tilde{\rho}_n$ (Robust and
Robust-P with post-processing) and $\overline{\rho}_n=\rho^{\star}$ (Oracle and Oracle-P with post-processing)
with $\sigma^{\star}=0.5$ and $\rho^{\star}=0.3$ (left),  $\rho^{\star}=0.6$ (middle) and $\rho^{\star}=0.8$ (right).}}
\label{fig:hausd_1_sig05}
\end{figure}

\begin{figure}[!h]
\includegraphics[width=0.32\textwidth,height=5cm]{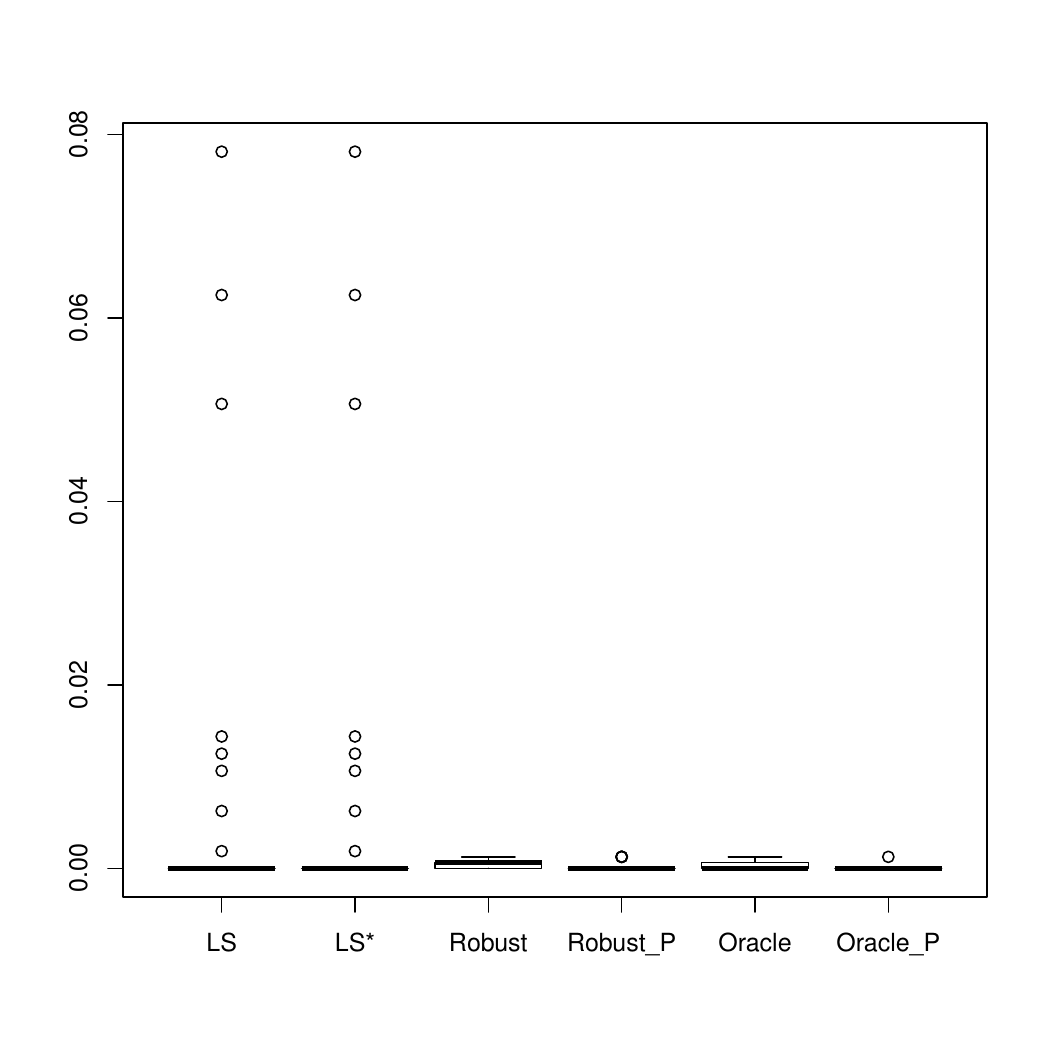}
\includegraphics[width=0.32\textwidth,height=5cm]{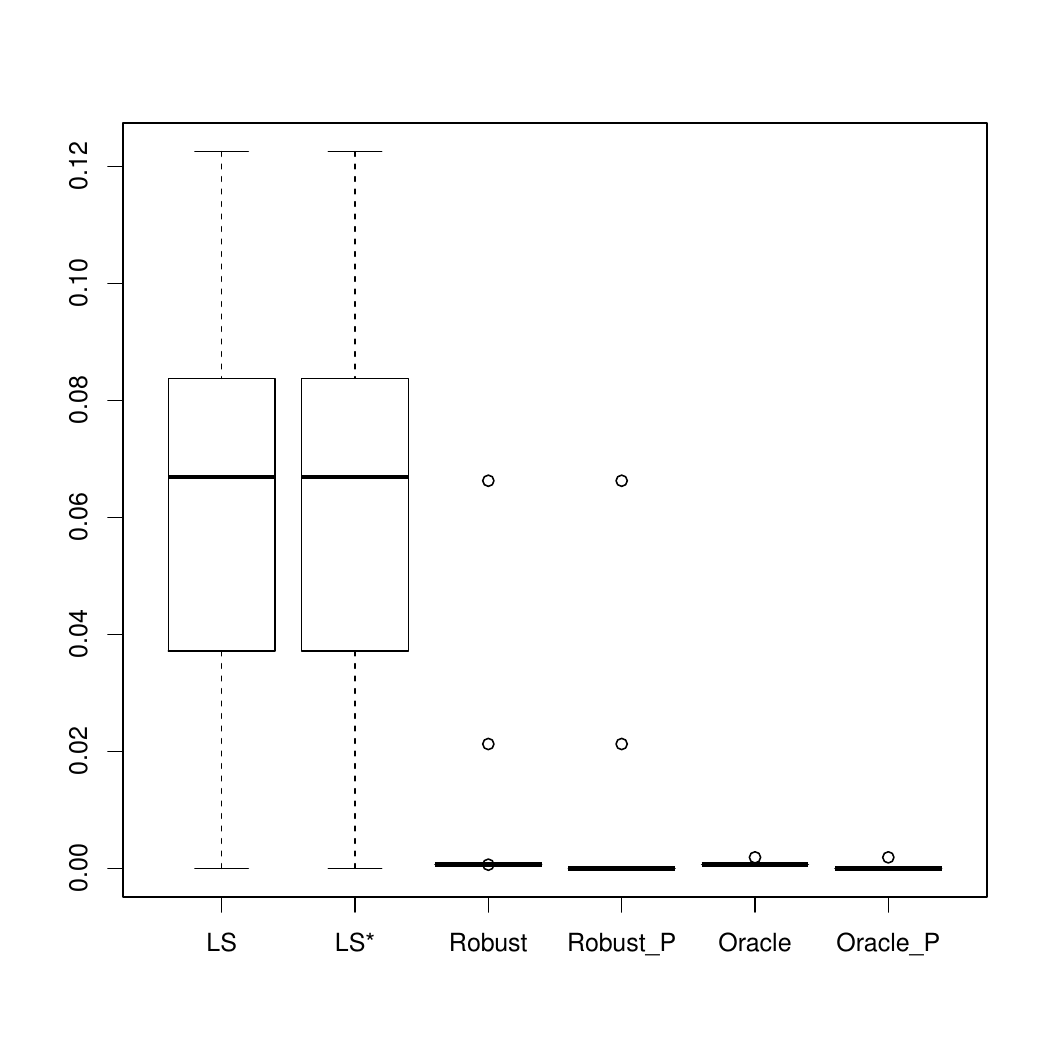}
\includegraphics[width=0.32\textwidth,height=5cm]{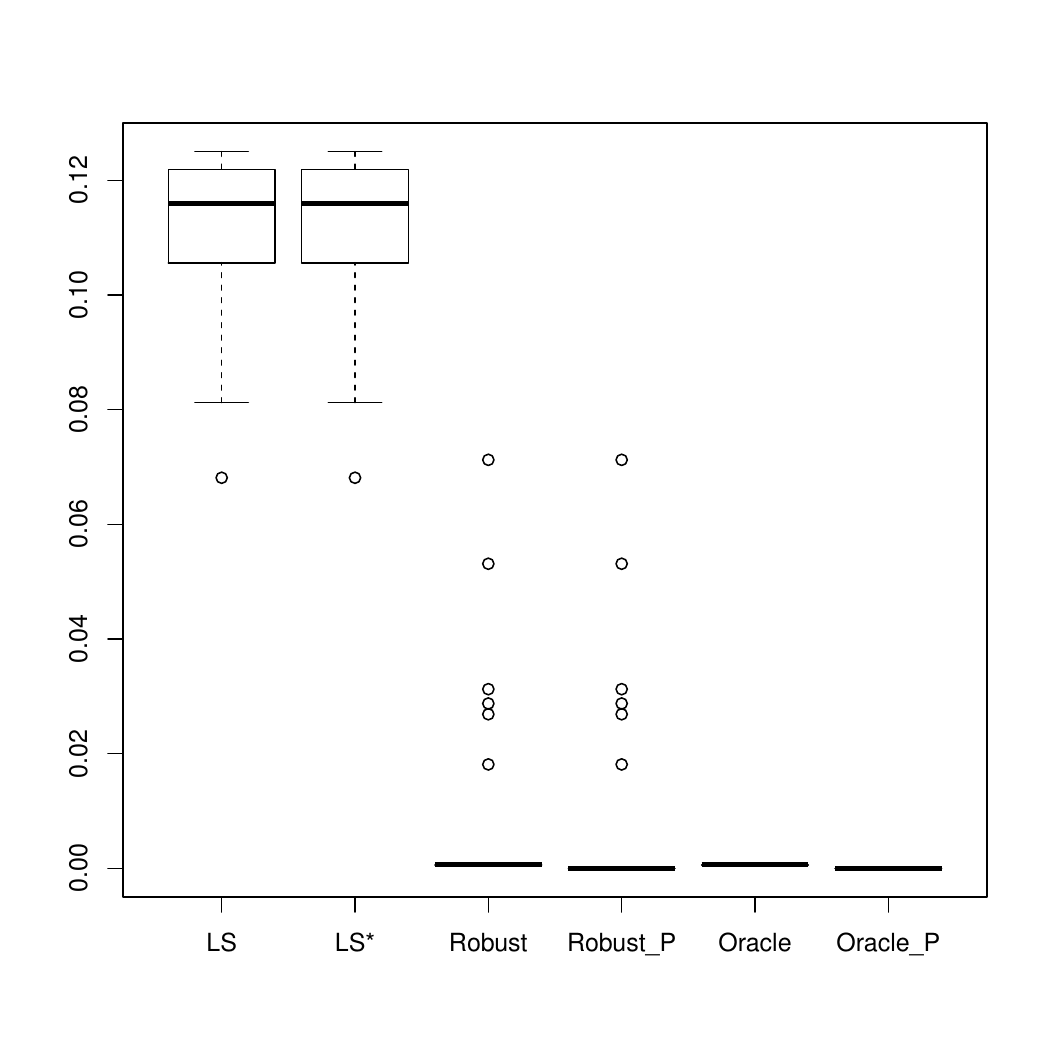}
\caption{\footnotesize{Boxplots for the second part of the Hausdorff distance ($d_2$) for $n=1600$ when $\overline{\rho}_n=0$ (LS and LS* when the true number of
change-points is known), $\overline{\rho}_n=\tilde{\rho}_n$ (Robust and
Robust-P with post-processing) and $\overline{\rho}_n=\rho^{\star}$ (Oracle and Oracle-P with post-processing)
with $\sigma^{\star}=0.1$ and $\rho^{\star}=0.3$ (left),  $\rho^{\star}=0.6$ (middle) and $\rho^{\star}=0.8$ (right).}}
\label{fig:hausd_2_sig01}
\end{figure}

\begin{figure}[!h]
\includegraphics[width=0.32\textwidth,height=5cm]{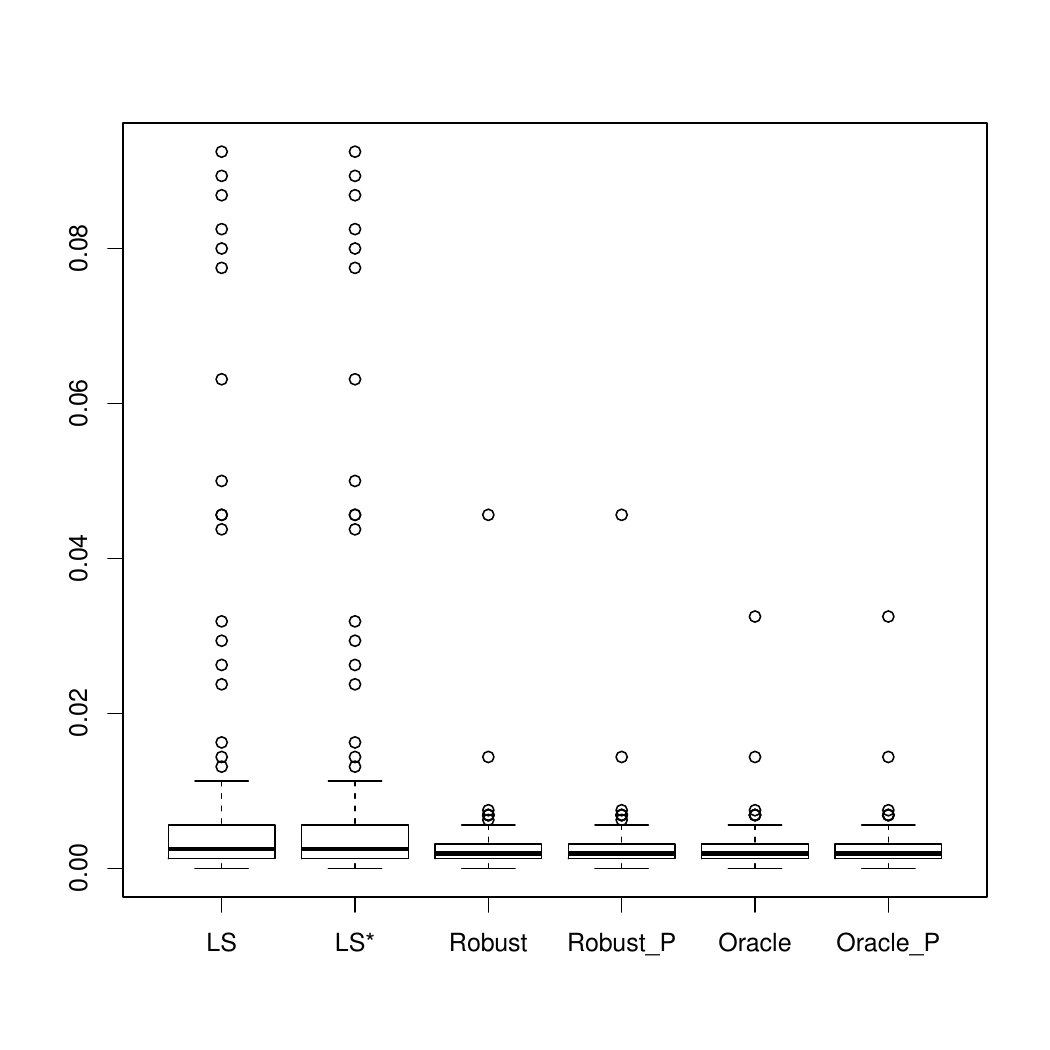}
\includegraphics[width=0.32\textwidth,height=5cm]{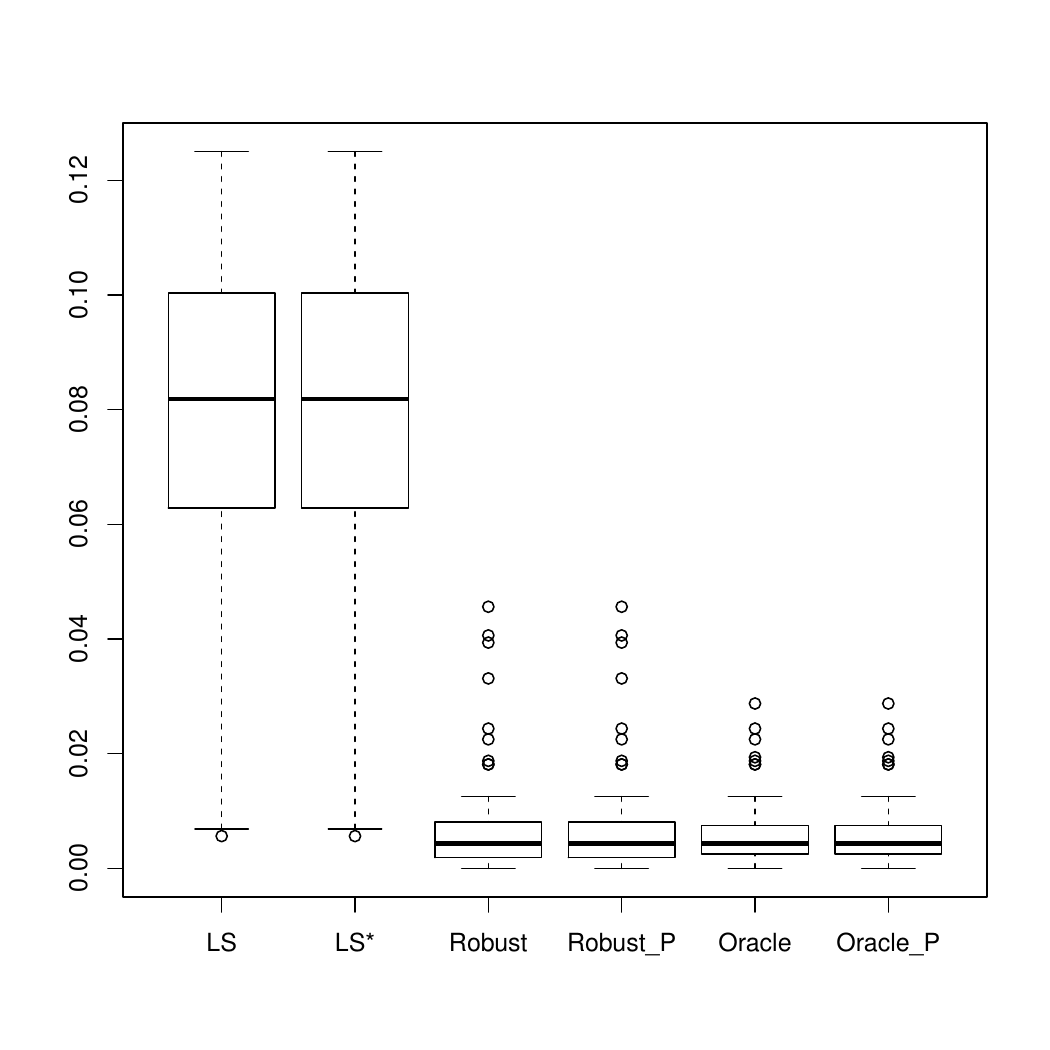}
\includegraphics[width=0.32\textwidth,height=5cm]{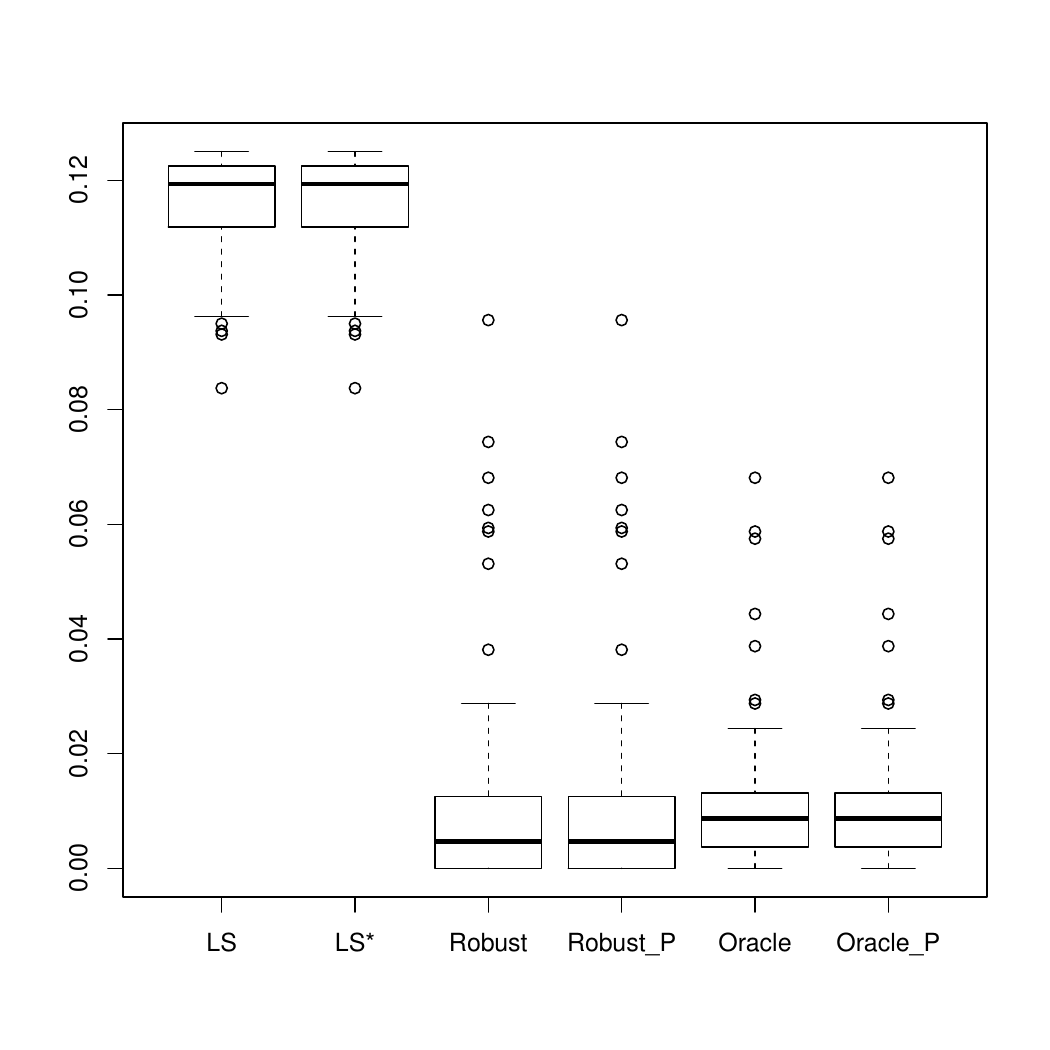}
\caption{\footnotesize{Boxplots for the second part of the Hausdorff distance ($d_2$) when $\overline{\rho}_n=0$ (LS and LS* when the true number of
change-points is known), $\overline{\rho}_n=\tilde{\rho}_n$ (Robust and
Robust-P with post-processing) and $\overline{\rho}_n=\rho^{\star}$ (Oracle and Oracle-P with post-processing)
with $\sigma^{\star}=0.5$ and $\rho^{\star}=0.3$ (left),  $\rho^{\star}=0.6$ (middle) and $\rho^{\star}=0.8$ (right).}}
\label{fig:hausd_2_sig05}
\end{figure}

An other way to illustrate the performance of the estimation of the change-point locations is the histograms of these estimates. We provide these plots only for LS, Robust-P and Oracle-P, because Post-processing does not change significantly LS estimates, and, furthermore, Robust (resp. Oracle) method's histograms with or without Post-Processing are very similar, see Figures \ref{fig:densplot1} and \ref{fig:densplot5}.

These figures illustrate that in case of over-estimation of the number of changes by LS method, the additional change-points seem to be uniformly distributed.

\begin{figure}[!h]
\includegraphics[width=0.9\textwidth,height=6cm]{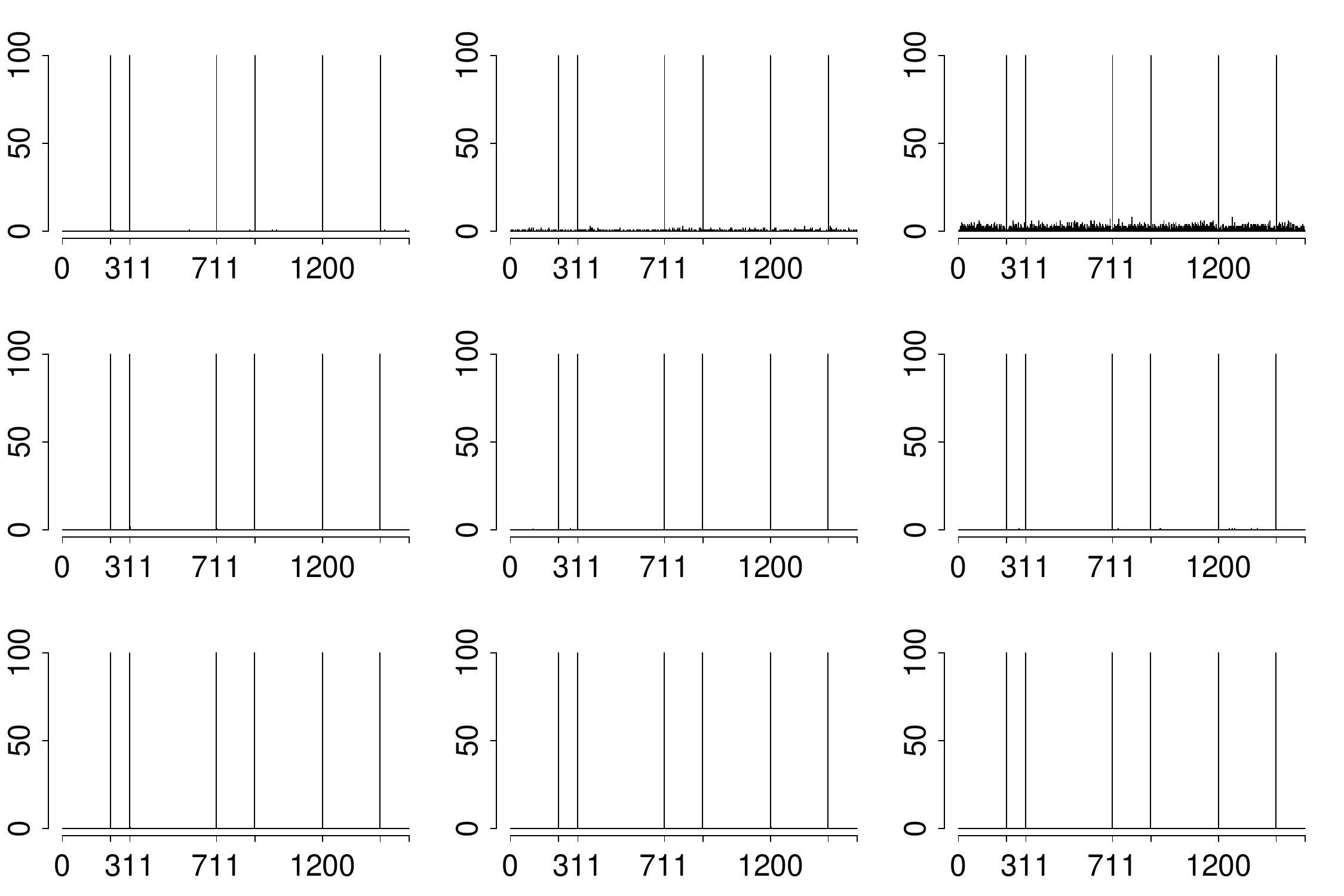}
\caption{\footnotesize{Frequencies of each possible change-point estimator, with $\sigma^{\star}=0.1$ and $n=1600$. Tick-marks on bottom-side axis represent the true change-point locations. $\overline{\rho}_n=0$ (LS, top line), $\overline{\rho}_n=\tilde{\rho}_n$ (Robust-P, middle line) and $\overline{\rho}_n=\rho^{\star}$ (Oracle-P, bottom line) with $\rho^{\star}=0.3$ (left),  $\rho^{\star}=0.6$ (middle) and $\rho^{\star}=0.8$ (right).}}
\label{fig:densplot1}
\end{figure}

\begin{figure}[!h]
\includegraphics[width=0.9\textwidth,height=6cm]{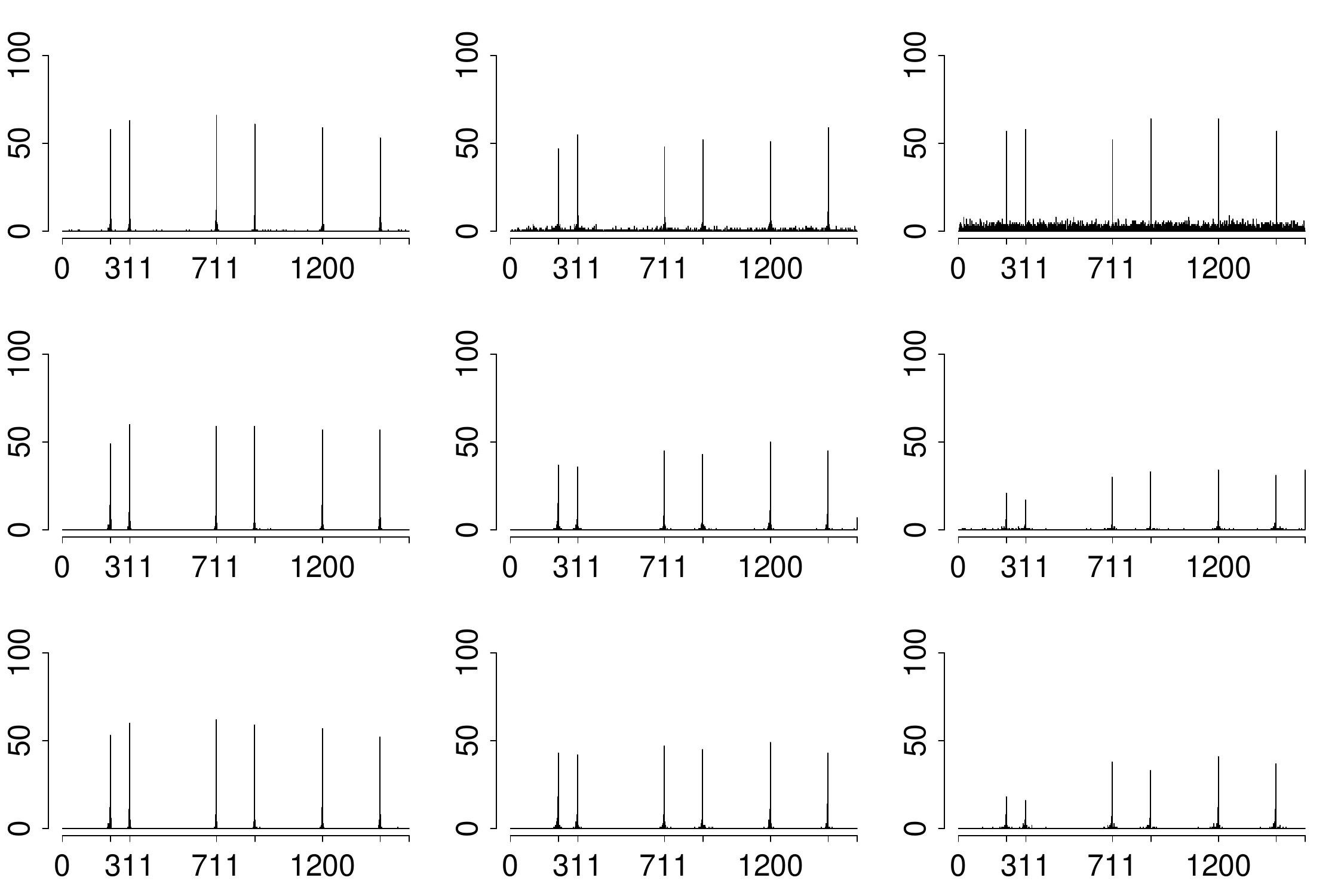}
\caption{\footnotesize{Frequencies of each possible change-point estimator, with $\sigma^{\star}=0.5$ and $n=1600$. Tick-marks on bottom-side axis represent the true change-point locations.$\overline{\rho}_n=0$ (LS, top line), $\overline{\rho}_n=\tilde{\rho}_n$ (Robust-P, middle line) and $\overline{\rho}_n=\rho^{\star}$ (Oracle-P, bottom line) with $\rho^{\star}=0.3$ (left),  $\rho^{\star}=0.6$ (middle) and $\rho^{\star}=0.8$ (right).}}
\label{fig:densplot5}
\end{figure}

\subsection{Additional simulation studies}\label{sec:add_simul}
\subsubsection{Comparison with Bardet et al. \cite{BKW10}}
The quasi-maximum likelihood method proposed by \cite{BKW10}, when applied to a Gaussian AR(1) process with changes in the mean $\left(y_0,\dots ,y_n\right)$ , consists in the minimization wrt $\boldsymbol{\rho}=\left(\rho_0,\dots ,\rho_m\right), \boldsymbol{\sigma}=\left(\sigma_0,\dots ,\sigma_m\right), \boldsymbol{\delta}=\left(\delta_0,\dots ,\delta_m\right)$ and $\boldsymbol{t}=\left(t_0,\dots ,t_m\right)$ of the following function:
\begin{equation}\label{eq:bardetFunc}
\left(\boldsymbol{\rho}, \boldsymbol{\sigma} , \boldsymbol{\delta},\boldsymbol{t}\right) \mapsto \sum_{k=0}^m \left\lbrace \left(t_{k+1}-t_k\right)\log\left(\sigma_k^2\right)+\frac{1}{\sigma_k^2}\sum_{i=t_k +1}^{t_{k+1}}\left(y_i-\rho_k y_{i-1}- \delta_k\right)^2 \right\rbrace \; .
\end{equation}
Indeed, in the class of models considered in \cite{BKW10}, changes in all the parameters are possible at each change-point. Using this method to estimate the change-point locations for data satisfying Model \eqref{eq:modele_new} or \eqref{eq:bkw} boils down to ignore the stationarity of $\left(\eta_i\right)_{i\geq 0}$ as defined in \eqref{eq:ar1}. It can lead to a poor estimation of change-point locations, especially when there are many changes close to each other.
To illustrate this fact, we compared our estimator of change-point locations to the estimates given by the minimization of \eqref{eq:bardetFunc}. We generated $100$ series of length $400$, under Model \eqref{eq:modele_new}, with $\rho^{\star}=0.3$ and $\sigma^{\star}=0.4$. The number of change-points, their locations and the means within segments are the same as in Section \ref{sec:design}. The number of changes is assumed to be known and we did not post-process the estimates. Simulations show that using the method of \cite{BKW10} in this case can lead to a poor estimation of close change-points, while our method is less affected by the length of segments (see Figure \ref{fig:BardetVsUs}). For example, the boundaries of the smallest segment are recovered in less than half of the simulations when minimizing \eqref{eq:bardetFunc}.
\begin{figure}
\centering
\includegraphics[width=.45\textwidth]{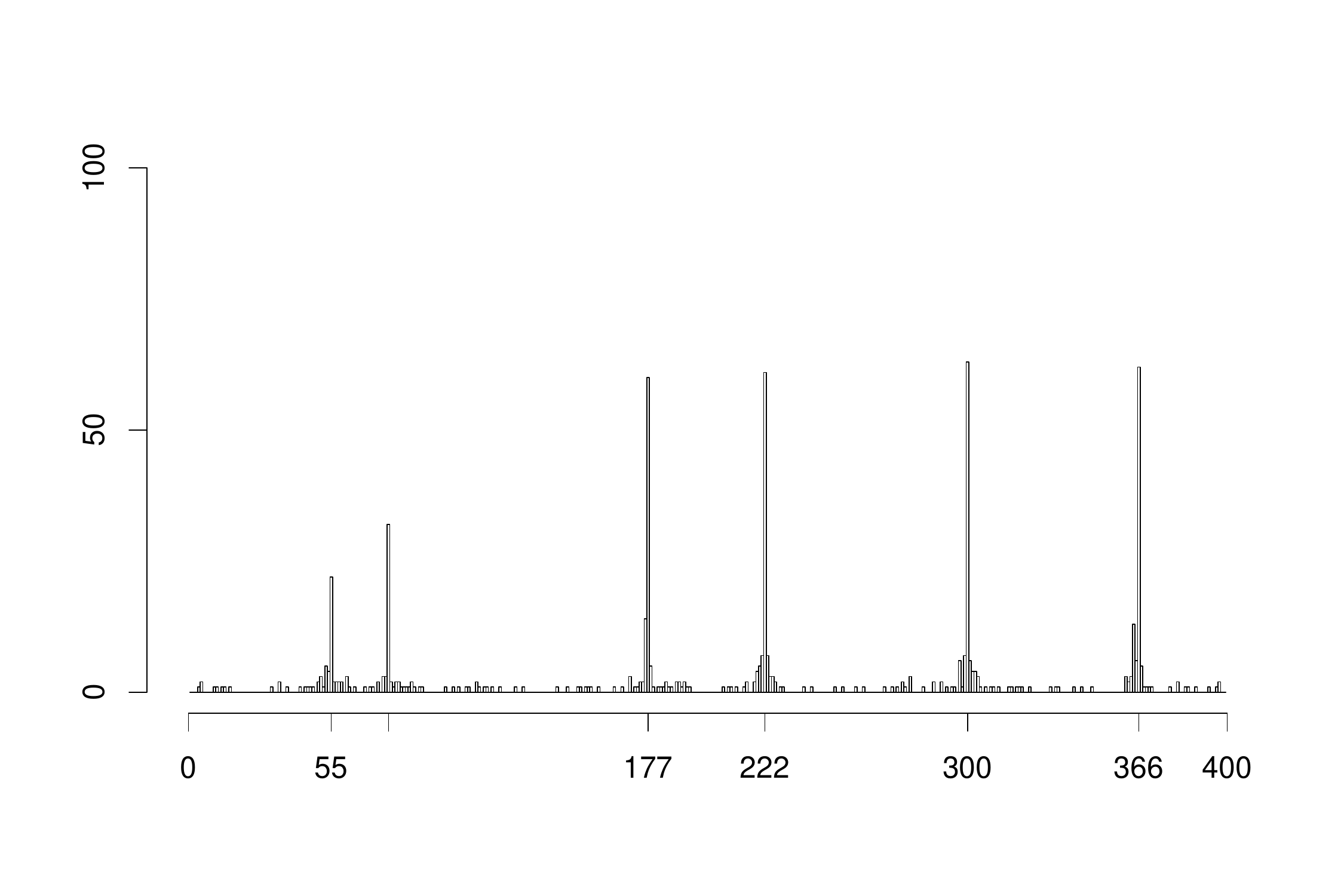}
\includegraphics[width=.45\textwidth]{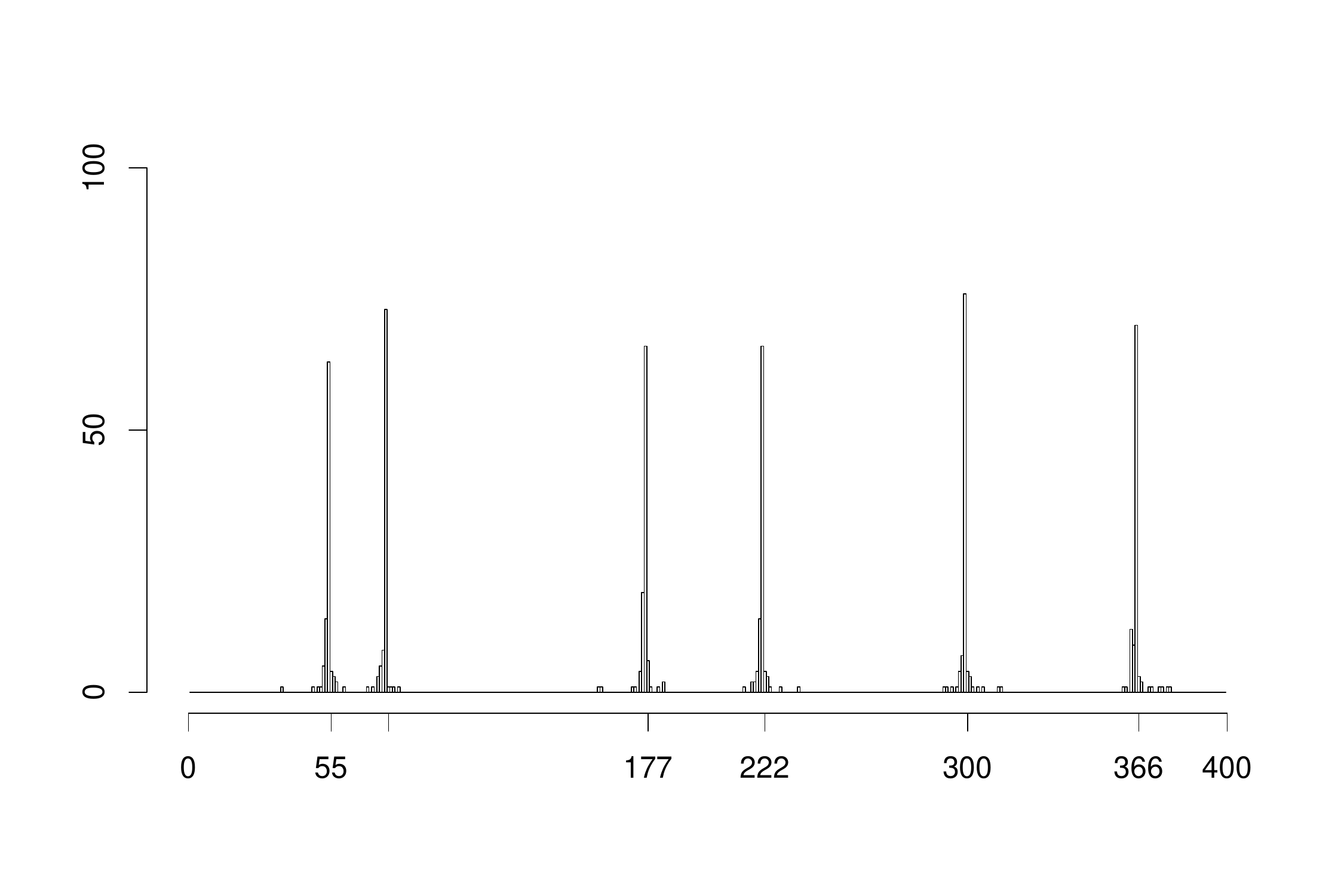}
\caption{\footnotesize{Frequencies of each possible change-point location estimate. Tick-marks on bottom-side axis represent the true change-point locations. Left: Estimation by the minimization of \eqref{eq:bardetFunc}. Right: Our method.}} \label{fig:BardetVsUs}
\end{figure}
\subsubsection{Robustness to model mis-specification}\label{sec:AR2} 
In this section, we study the
behaviour of our proposed robust procedure (Robust-P) 
when the signal is corrupted by an AR(2) Gaussian
process, e.g. in Model \ref{eq:modele_new}, $\eta_i$ is a zero-mean
stationary process such that
$$
\eta_i = \phi_1^{\star} \eta_{i-1}+\phi_2^{\star} \eta_{i-2}+ \varepsilon_i ,
$$
where $\vert \phi_2 \vert < 1$, $\phi_1+\phi_2<1$ and
$\phi_2-\phi_1<1$. We considered series of fixed length $n=1600$, a
residual standard deviation $\sigma^{\star}=0.1$, $\phi_1^{\star}=0.3$ and
$\phi_2^{\star}$ in $\{-0.9,-0.8,-0.7,\dots,0.5,0.6\}$ 
We used the same segmentation design as in
subsection \ref{sec:post-proc}. Each combination was
replicated $100$ times. All the results are
displayed in Figure \ref{fig:K-rho-h1-h2-AR2}. \\
The procedure performs well when $\phi_2^{\star}$ belongs to the interval
$[-0.5,0.2]$ as expected (similar to the case of AR(1)): the
estimated segmentation is close to the true one. When
$\phi_2^{\star}>0.2$, it tends to over-estimate the number of
change-points. The true change-points are detected ($d_1$ is close
to zero, e.g. the decorrelation procedure with the obtained negative
estimation of ${\rho}^{\star}$ leads to an increasing in the mean
differences) but false change-points are added (large $d_2$). When
$\phi_2^{\star}<-0.5$, under-segmentation is observed: the decorrelation
procedure with a large estimated value of ${\rho}^{\star}$ leads to a
difficult segmentation problem.

\begin{figure}
\centering
\includegraphics[width=.32\textwidth, height=0.2\textheight]{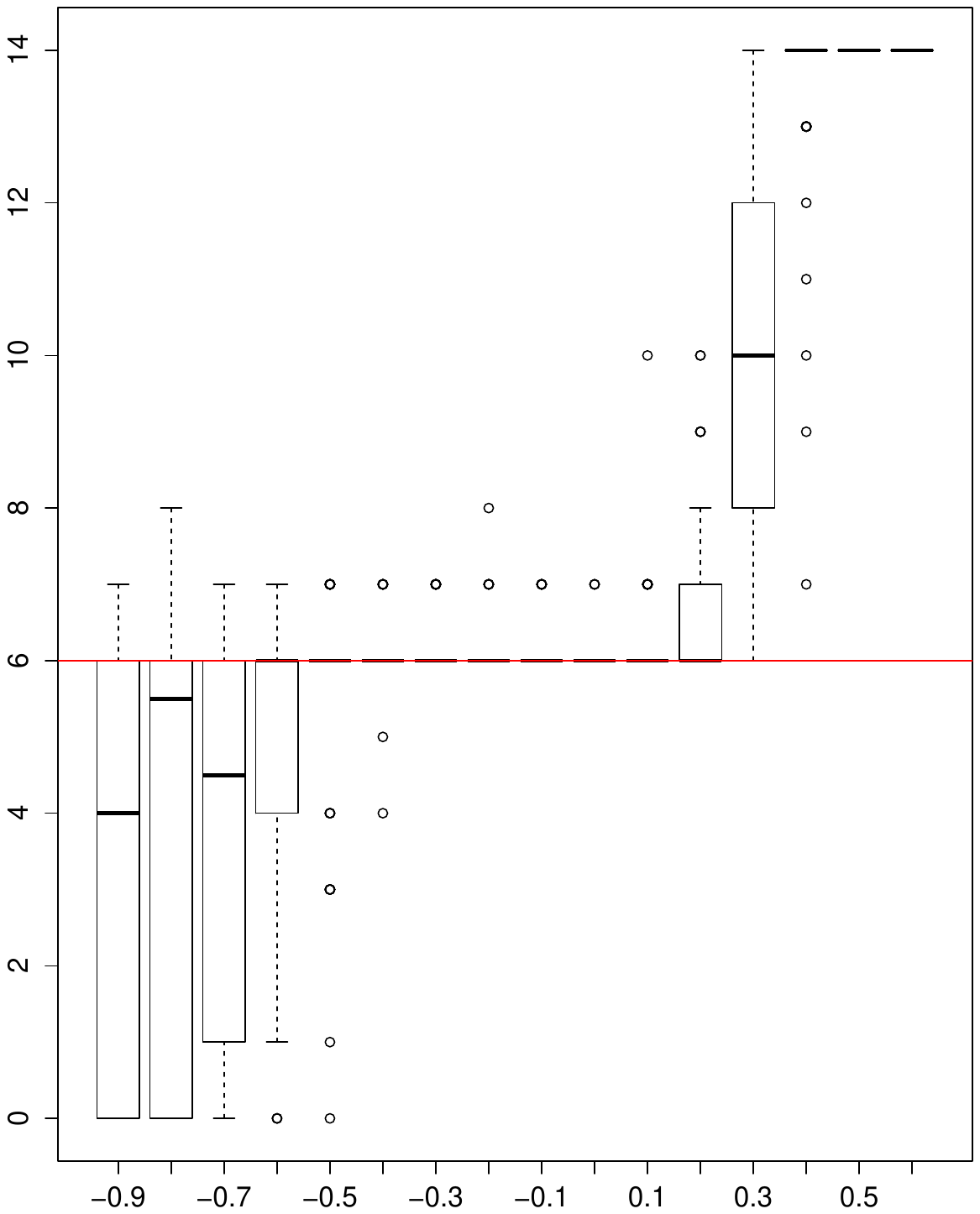}
\includegraphics[width=.32\textwidth, height=0.2\textheight]{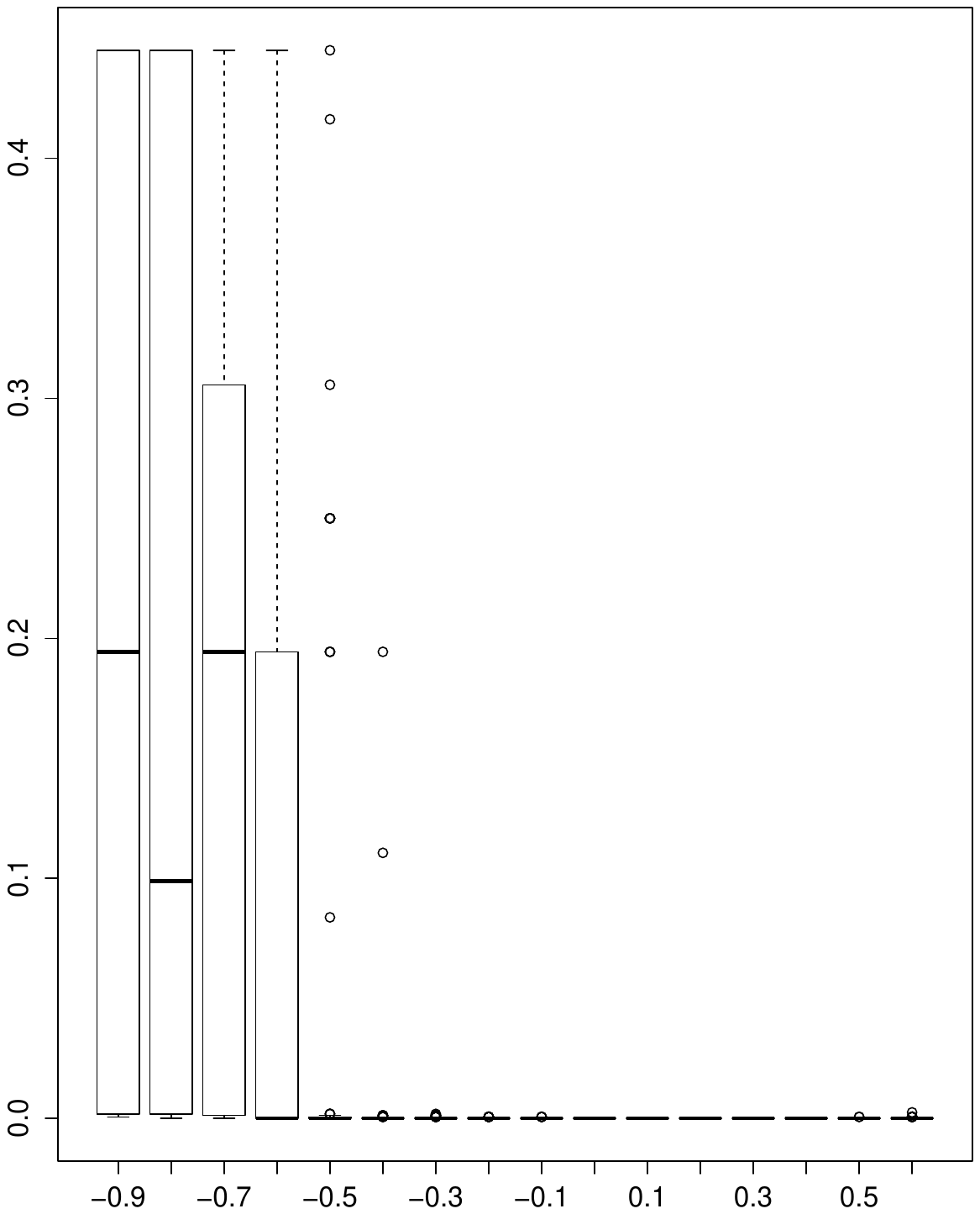}
\includegraphics[width=.32\textwidth, height=0.2\textheight]{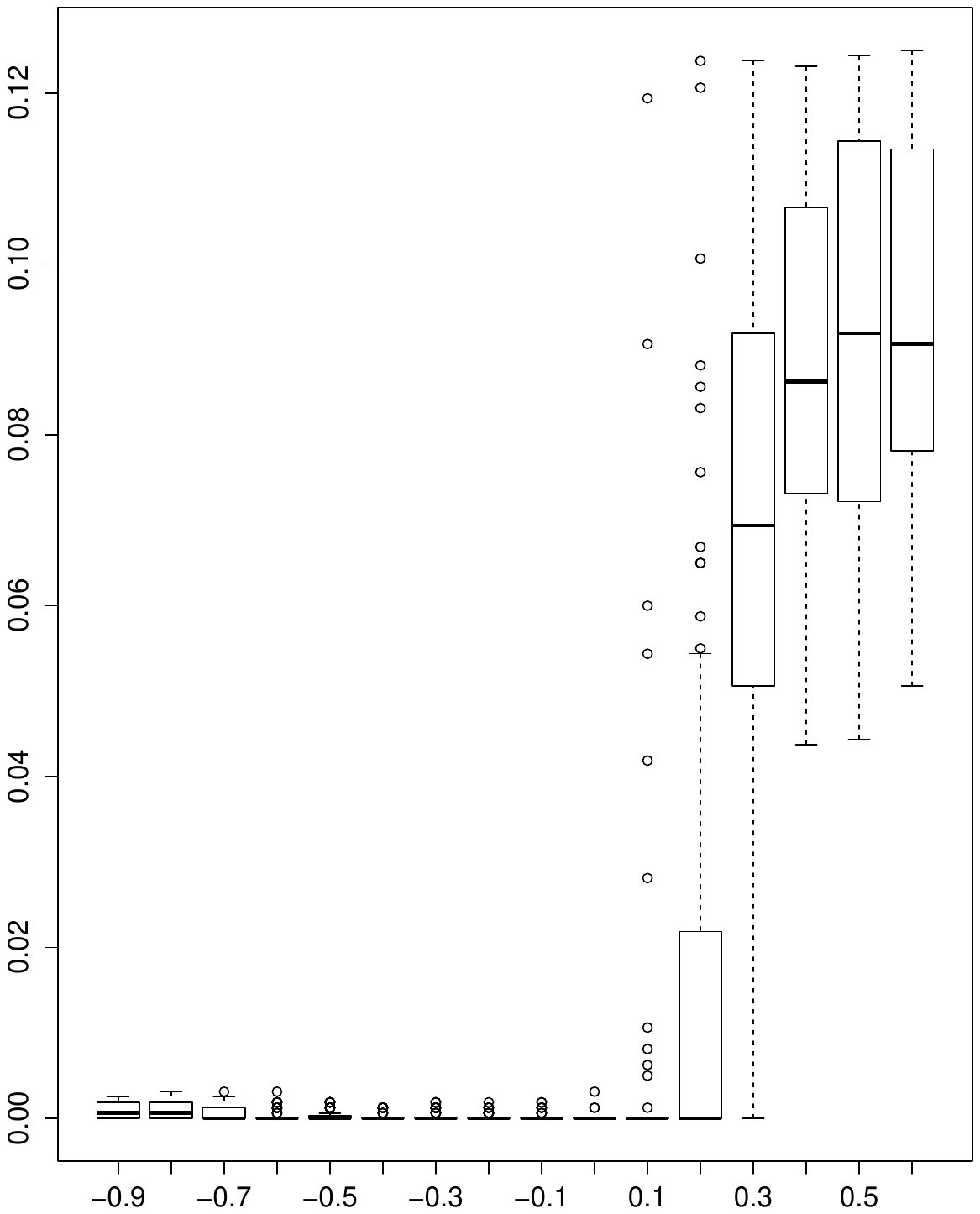}
\caption{\footnotesize{Left: Boxplots for the estimated number
of change-points. 
Center and right: Boxplots for the first part of the Hausdorff distance ($d_1$) and for the second part of the Hausdorff distance ($d_2$) with $n=1600$, $\sigma^{\star}=0.1$ and $\phi_1^{\star}=0.3$ wrt different values of $\phi_2^{\star}$.}} \label{fig:K-rho-h1-h2-AR2}
\end{figure}

\subsubsection{Estimator of $\rho^{\star}$ in the case of the Cauchy distribution}
In Section \ref{sec:correlation}, an analogous estimator of $\rho^{\star}$ in the case of Cauchy distributed observations is proposed. 
We follow the simulation design described in Subsection \ref{sec:design}, where the Gaussian random variables are replaced by
Cauchy random variables. More precisely, the expectation parameters are replaced by the location parameters of the Cauchy 
distribution and $\sigma^{\star}$ is replaced by the scale parameter of the Cauchy distribution. We can see from Figure \ref{fig:Cauchy} that $\widetilde{\widetilde{\rho}}_n$ is an accurate estimator of $\rho^{\star}$ except when $\rho^{\star}$ is close to zero.
\begin{figure}
\centering
\includegraphics[width=.4\textwidth]{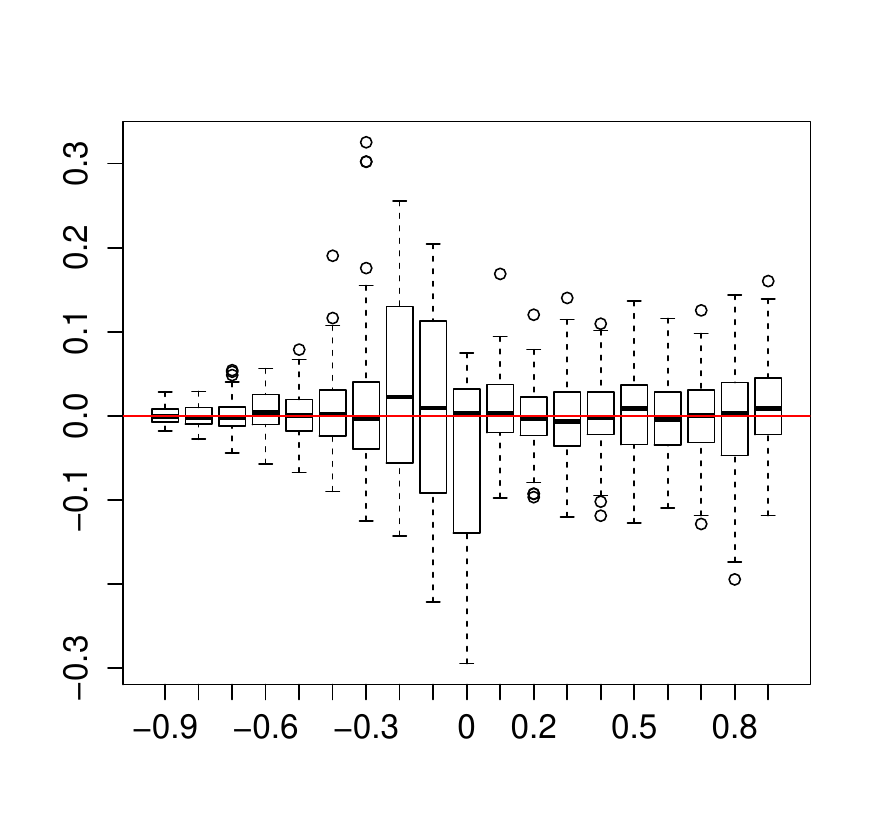}
\caption{\footnotesize{Boxplots of $\widetilde{\widetilde{\rho}}_n-\rho^{\star}$ for different values of $\rho^{\star}$  
when $n=1600$ and $\sigma^{\star}=0.1$.}} \label{fig:Cauchy}
\end{figure}
When this estimator of $\rho^{\star}$ is used in our change-point estimation method, it leads to poor estimations of the change-points
since the Cauchy distribution does not have finite second order moment (simulations not shown).




\section{Conclusion}

In this paper, we propose a novel approach for estimating multiple change-points
in the mean of a Gaussian AR(1) process. Our approach is based on two main stages.
The first one consists in building a robust estimator of the autocorrelation parameter
which is used for whitening the original series. In the second stage, we apply
the inference approach commonly used to estimate change-points in the mean
of independent random variables. 
In the course of this study, we have shown that our approach, which is implemented in the R 
package \textsf{AR1seg}, is a very efficient
technique both on a theoretical and practical point of view.
More precisely, it has two main features which make it very
attractive. Firstly, the estimators that we propose have the same asymptotic properties as the classical estimators 
in the independent framework which means that the performances of our estimators are not affected by the dependence assumption.
Secondly, from a practical point of view, \textsf{AR1seg} is computationally
efficient and exhibits better performance on finite sample size data than 
existing approaches which do not take into account the dependence structure of the observations.


%

\bibliographystyle{apalike}
\bibliography{bibliographie1}

\begin{thebibliography}{}

\bibitem[Arcones, 1994]{arcones1994limit}
Arcones, M.~A. (1994).
\newblock Limit theorems for nonlinear functionals of a stationary {G}aussian
  sequence of vectors.
\newblock {\em The Annals of Probability}, pages 2242--2274.

\bibitem[Auger and Lawrence, 1989]{AugerLawrence}
Auger, I.~E. and Lawrence, C.~E. (1989).
\newblock Algorithms for the optimal identification of segment neighborhoods.
\newblock {\em Bulletin of mathematical biology}, 51(1):39--54.

\bibitem[Bai and Perron, 1998]{BP}
Bai, J. and Perron, P. (1998).
\newblock Estimating and testing linear models with multiple structural
  changes.
\newblock {\em Econometrica}, pages 47--78.

\bibitem[Bai and Perron, 2003]{BPalgo}
Bai, J. and Perron, P. (2003).
\newblock Computation and analysis of multiple structural change models.
\newblock {\em Journal of Applied Econometrics}, 18(1):1--22.

\bibitem[Bardet et~al., 2012]{BKW10}
Bardet, J.-M., Kengne, W.~C., Wintenberger, O., et~al. (2012).
\newblock Detecting multiple change-points in general causal time series using
  penalized quasi-likelihood.
\newblock {\em Electronic Journal of Statistics}, pages 1--50.

\bibitem[Basseville and Nikiforov, 1993]{Bas93N}
Basseville, M. and Nikiforov, N. (1993).
\newblock {\em The Detection of abrupt changes - {T}heory and applications}.
\newblock Prentice-Hall: Information and System sciences series.

\bibitem[Boysen et~al., 2009]{boysen:2009}
Boysen, L., Kempe, A., Munk, A., Liebscher, V., and Wittich, O. (2009).
\newblock Consistencies and rates of convergence of jump penalized least
  squares estimators.
\newblock {\em The Annals of Statistics}, 37(1):157--183.

\bibitem[Braun et~al., 2000]{braun2000multiple}
Braun, J.~V., Braun, R., and M{\"u}ller, H.-G. (2000).
\newblock Multiple changepoint fitting via quasilikelihood, with application to
  {DNA} sequence segmentation.
\newblock {\em Biometrika}, 87(2):301--314.

\bibitem[Braun and Muller, 1998]{braun1998statistical}
Braun, J.~V. and Muller, H.-G. (1998).
\newblock Statistical methods for {DNA} sequence segmentation.
\newblock {\em Statistical Science}, pages 142--162.

\bibitem[Brockwell and Davis, 2009]{brockwell}
Brockwell, P. and Davis, R. (2009).
\newblock {\em Time series: theory and methods}.
\newblock Springer Verlag.

\bibitem[Cs{\"o}rg{\'o} and Mielniczuk, 1996]{CM}
Cs{\"o}rg{\'o}, S. and Mielniczuk, J. (1996).
\newblock The empirical process of a short-range dependent stationary sequence
  under {G}aussian subordination.
\newblock {\em Probability Theory and Related Fields}, 104:15--25.

\bibitem[Davis et~al., 2006]{davis2006structural}
Davis, R.~A., Lee, T. C.~M., and Rodriguez-Yam, G.~A. (2006).
\newblock Structural break estimation for nonstationary time series models.
\newblock {\em Journal of the American Statistical Association},
  101(473):223--239.

\bibitem[Durrett, 2010]{durrett2010probability}
Durrett, R. (2010).
\newblock {\em Probability: theory and examples}.
\newblock Cambridge university press.

\bibitem[Feller, 1971]{feller:1971}
Feller, W. (1971).
\newblock {\em An Introduction to Probability Theory and Its Applications, Vol.
  2}.
\newblock Wiley, New York, NY, second edition.

\bibitem[Gazeaux et~al., 2013]{Gazeaux:2013}
Gazeaux, J., Williams, S., King, M., Bos, M., Dach, R., Deo, M., Moore, A.,
  Ostini, L., Petrie, E., Roggero, M., Teferle, F., Olivares, G., and Webb, F.
  (2013).
\newblock Detecting offsets in {GPS} time series: First results from the
  detection of offsets in {GPS} experiment.
\newblock {\em Journal of Geophysical Research (Solid Earth)}, 118(5).

\bibitem[Harchaoui and L\'evy-Leduc, 2010]{harchaoui:2010}
Harchaoui, Z. and L\'evy-Leduc, C. (2010).
\newblock Multiple change-point estimation with a total variation penalty.
\newblock {\em Journal of the American Statistical Association}, 105(492).

\bibitem[Killick et~al., 2012]{KFE12-JASA}
Killick, R., Fearnhead, P., and Eckley, I.~A. (2012).
\newblock Optimal detection of changepoints with a linear computational cost.
\newblock {\em Journal of the American Statistical Association},
  107(500):1590--1598.

\bibitem[Lai et~al., 2005a]{lai2005autoregressive}
Lai, T.~L., Liu, H., and Xing, H. (2005a).
\newblock Autoregressive models with piecewise constant volatility and
  regression parameters.
\newblock {\em Statistica Sinica}, 15:279--301.

\bibitem[Lai et~al., 2005b]{simul}
Lai, W., Johnson, M., Kucherlapati, R., and Park, P. (2005b).
\newblock Comparative analysis of algorithms for identifying amplifications and
  deletions in array {CGH} data.
\newblock {\em Bioinformatics}, 21(19):3763.

\bibitem[Lavielle, 1999]{lavielle}
Lavielle, M. (1999).
\newblock Detection of multiple changes in a sequence of dependent variables.
\newblock {\em Stochastic Processes and their Applications}, 83(1):79--102.

\bibitem[Lavielle, 2005]{lavielle2005using}
Lavielle, M. (2005).
\newblock Using penalized contrasts for the change-point problem.
\newblock {\em Signal Processing}, 85(8):1501--1510.

\bibitem[Lavielle and Moulines, 2000]{LM}
Lavielle, M. and Moulines, E. (2000).
\newblock Least-squares estimation of an unknown number of shifts in a time
  series.
\newblock {\em Journal of time series analysis}, 21(1):33--59.

\bibitem[Lebarbier, 2005]{lebarbier2005detecting}
Lebarbier, {\'E}. (2005).
\newblock Detecting multiple change-points in the mean of {G}aussian process by
  model selection.
\newblock {\em Signal processing}, 85(4):717--736.

\bibitem[L{\'e}vy-Leduc et~al., 2011]{levy2011robust}
L{\'e}vy-Leduc, C., Boistard, H., Moulines, E., Taqqu, M.~S., and Reisen, V.~A.
  (2011).
\newblock Robust estimation of the scale and of the autocovariance function of
  {G}aussian short-and long-range dependent processes.
\newblock {\em Journal of Time Series Analysis}, 32(2):135--156.

\bibitem[Li and Lund, 2012]{LiL12}
Li, S. and Lund, R. (2012).
\newblock Multiple changepoint detection via genetic algorithms.
\newblock {\em Journal of Climate}, 25(2):674--686.

\bibitem[Lu et~al., 2010]{climat}
Lu, Q., Lund, R., and Lee, T. (2010).
\newblock An {MDL} approach to the climate segmentation problem.
\newblock {\em The Annals of Applied Statistics}, 4(1):299--319.

\bibitem[Ma and Genton, 2000]{MG}
Ma, Y. and Genton, M.~G. (2000).
\newblock Highly robust estimation of the autocovariance function.
\newblock {\em Journal of Time Series Analysis}, 21(6):663--684.

\bibitem[Mestre, 2000]{mestre2000methodes}
Mestre, O. (2000).
\newblock {\em M{\'e}thodes statistiques pour l'homog{\'e}n{\'e}isation de
  longues s{\'e}ries climatiques}.
\newblock PhD thesis.

\bibitem[Picard et~al., 2011]{picard2011joint}
Picard, F., Lebarbier, E., Budinsk{\'a}, E., and Robin, S. (2011).
\newblock Joint segmentation of multivariate {G}aussian processes using mixed
  linear models.
\newblock {\em Computational Statistics \& Data Analysis}, 55(2):1160--1170.

\bibitem[Picard et~al., 2005]{picard2005statistical}
Picard, F., Robin, S., Lavielle, M., Vaisse, C., and Daudin, J.-J. (2005).
\newblock A statistical approach for array {CGH} data analysis.
\newblock {\em BMC bioinformatics}, 6(1):27.

\bibitem[Rigaill, 2010]{pruned}
Rigaill, G. (2010).
\newblock Pruned dynamic programming for optimal multiple change-point
  detection.
\newblock {\em Arxiv preprint arXiv:1004.0887}.

\bibitem[Rousseeuw and Croux, 1993]{CR}
Rousseeuw, P.~J. and Croux, C. (1993).
\newblock Alternatives to the median absolute deviation.
\newblock {\em Journal of the American Statistical Association},
  88(424):1273--1283.

\bibitem[Van~der Vaart, 2000]{van}
Van~der Vaart, A. (2000).
\newblock {\em Asymptotic statistics}.
\newblock Number~3. Cambridge Univ Pr.

\bibitem[Williams, 2003]{Williams:2003}
Williams, S. (2003).
\newblock Offsets in {G}lobal {P}ositioning {S}ystem time series.
\newblock {\em Journal of Geophysical Research (Solid Earth)}, 108.

\bibitem[Yao, 1988]{yao}
Yao, Y. (1988).
\newblock Estimating the number of change-points via {S}chwarz' criterion.
\newblock {\em Statistics \& Probability Letters}, 6(3):181--189.

\bibitem[Zhang and Siegmund, 2007]{zs}
Zhang, N. and Siegmund, D. (2007).
\newblock A modified {B}ayes information criterion with applications to the
  analysis of comparative genomic hybridization data.
\newblock {\em Biometrics}, 63(1):22--32.

\end{thebibliography}

\newpage

\appendix

\section{Proofs} \label{App:Proof}

\subsection{Proof of Proposition \ref{prop:correlation}}

Let $F_1$ and $F_2$ denote the cumulative distribution functions (cdf) of 
$\left(\left|y_{i+1}-y_i\right|\right)$ for $i\neq t_{n,1}^{\star},\dots,t_{n,m^{\star}}^{\star}$ and $\left(|y_{i+2}-y_i|\right)$
for $i\neq t_{n,1}^{\star}-1,\dots,t_{n,m^{\star}}^{\star}-1$,
respectively. By  (\ref{eq:modele_new}), 
$(y_i-\PE(y_i))_{0\leq i\leq n}$ are $(n+1)$ observations of a AR(1) stationnary Gaussian process thus for any $i\neq t_{n,1}^{\star},\dots,t_{n,m^{\star}}^{\star}$, 
$(y_{i+1}-y_i)$ and for any $i\neq t_{n,1}^{\star}-1,\dots,t_{n,m^{\star}}^{\star}-1$, $(y_{i+2}-y_i)$ are 
zero-mean Gaussian random variables with variances equal to
$2\sigma^{\star 2}/(1+\rho^{\star})$ and $2\sigma^{\star 2}$, respectively.
Hence, for all $t$ in $\rset$,
\begin{equation}\label{eq:def:F1:F2}
F_1:t\mapsto 2\Phi\left(t\sqrt{\frac{1+\rho^{\star}}{2\sigma^{\star 2}}}\right)-1
\textrm{ and } F_2:t\mapsto 2\Phi\left(t\sqrt{\frac{1}{2\sigma^{\star 2}}}\right)-1\;,
\end{equation}
where $\Phi$ denotes the cumulative distribution function of a standard Gaussian random variable.

Let also denote by $F_{1,n}$ and $F_{2,n-1}$ the empirical cumulative distribution functions of
$\left(|y_{i+1}-y_i|\right)_{0\leq i\leq n-1}$ and $\left(|y_{i+2}-y_i|\right)_{0\leq i\leq n-2}$,
respectively. Observe that for all $t$ in $\rset$,
\begin{multline}\label{eq:cdf1}
\sqrt{n}(F_{1,n}(t)-F_1(t))=\frac{1}{\sqrt{n}}\sum_{i=0}^{n-1} \left(\1_{\{|y_{i+1}-y_i|\leq t\}}-F_1(t)\right)\\
=\frac{1}{\sqrt{n}}\sum_{i\in \{t_{n,1}^{\star},\dots,t_{n,m^{\star}}^{\star}\}}\left(\1_{\{|y_{i+1}-y_i|\leq t\}}-F_1(t)\right)
+\frac{1}{\sqrt{n}}\sum_{\stackrel{0\leq i\leq n-1}{i \notin \{t_{n,1}^{\star},\dots,t_{n,m^{\star}}^{\star}\}}} \left(\1_{\{|y_{i+1}-y_i|\leq t\}}-F_1(t)\right)\\
=\frac{1}{\sqrt{n}}\sum_{0\leq i\leq n-1} \left(\1_{\{|z_i|\leq t\}}-F_1(t)\right)+R_n(t)\;,
\end{multline}
where $\sup_{t\in\rset} |R_n(t)|=o_p(1)$, the $z_i=y_{i+1}-y_i$ except for $i =t_{n,1}^{\star},\dots,t_{n,m^{\star}}^{\star}$, where
$z_i=\eta_{i+1}-\eta_i$, $(\eta_i)$ being defined in (\ref{eq:ar1}).

Thus, by using the theorem of \cite{CM}, we obtain that the first term in the rhs of (\ref{eq:cdf1}) 
converges in distribution to a zero-mean Gaussian process $G$ in the
space of c\`adl\`ag functions equipped with the uniform norm. Since the second term in the rhs 
tends uniformly to zero in probability, we get that $\sqrt{n}(F_{1,n}-F_1)$ converges in distribution to a zero-mean Gaussian process in the
space of c\`adl\`ag functions equipped with the uniform norm
and that the same holds for $\sqrt{n-1}(F_{2,n-1}-F_2)$.

By Lemma 21.3 of \cite{van} the quantile function $T: F\mapsto F^{-1}(1/2)$
is Hadamard differentiable at $F$ tangentially to the set of c\`adl\`ag functions $h$ that are continuous at
$F^{-1}(1/2)$ with derivative $T'_F(h)=-h(F^{-1}(1/2))/F'(F^{-1}(1/2))$.
By applying the functional delta method (Theorem 20.8 in \cite{van}), we get that
$\sqrt{n}(T(F_{1,n})-T(F_1))$ converges in distribution to $T'_{F_1}(G)$. Moreover, by the 
continuous mapping theorem, it is the same for $T'_{F_1}\left\{\sqrt{n}(F_{1,n}-F_1)\right\}$. Thus,
\begin{multline}\label{eq:expansion_F1}
\sqrt{n}\left(F^{-1}_{1,n}(1/2)-F_1^{-1}(1/2)\right)
=T'_{F_1}\left\{\sqrt{n}(F_{1,n}-F_1)\right\}+o_p(1)\\
=-\frac{1}{\sqrt{n}}\frac{\sum_{i=0}^{n-1}\left(\1_{\{|y_{i+1}-y_i|\leq F_1^{-1}(1/2)\}}-1/2\right)}
{F_1'(F_1^{-1}(1/2))}+o_p(1)\;.
\end{multline}
In the same way,
\begin{multline}\label{eq:expansion_F2}
\sqrt{n-1}\left(F^{-1}_{2,n-1}(1/2)-F_2^{-1}(1/2)\right)=\\
-\frac{1}{\sqrt{n-1}}\frac{\sum_{i=0}^{n-2}\left(\1_{\{|y_{i+2}-y_i|\leq F_2^{-1}(1/2)\}}-1/2\right)}
{F_2'(F_2^{-1}(1/2))}+o_p(1)\;,
\end{multline}

By applying the Delta method \cite[Theorem 3.1]{van} with the transformation $f(x)=x^2$, we get
\begin{multline}\label{eq:expansion_F1_square}
\sqrt{n}\left(F^{-1}_{1,n}(1/2)^2-F_1^{-1}(1/2)^2\right)=\\
-\frac{2 F_1^{-1}(1/2)}{\sqrt{n}}\frac{\sum_{i=0}^{n-1}\left(\1_{\{|y_{i+1}-y_i|\leq F_1^{-1}(1/2)\}}-1/2\right)}
{F_1'(F_1^{-1}(1/2))}+o_p(1)\;,
\end{multline}
\begin{multline}\label{eq:expansion_F2_square}
\sqrt{n-1}\left(F^{-1}_{2,n-1}(1/2)^2-F_2^{-1}(1/2)^2\right)=\\
-\frac{2 F_2^{-1}(1/2)}{\sqrt{n-1}}\frac{\sum_{i=0}^{n-2}\left(\1_{\{|y_{i+2}-y_i|\leq F_2^{-1}(1/2)\}}-1/2\right)}
{F_2'(F_2^{-1}(1/2))}+o_p(1)\;,
\end{multline}
Note that by (\ref{eq:def:F1:F2}), we obtain that
\begin{equation}\label{eq:F1-1:F2-1}
F_1^{-1}(1/2)=\sqrt{\frac{2\sigma^{\star 2}}{1+\rho^{\star}}}\Phi^{-1}(3/4) \textrm{ and } F_2^{-1}(1/2)=\sqrt{2\sigma^{\star 2}}\Phi^{-1}(3/4)\;.
\end{equation}
Moreover,
\begin{equation}\label{eq:F1prime:F2prime}
F_1'(F_1^{-1}(1/2))=2\sqrt{\frac{1+\rho^{\star}}{2\sigma^{\star 2}}}\varphi\left(\Phi^{-1}(3/4)\right)   
\textrm{ and } F_2'(F_2^{-1}(1/2))=2\sqrt{\frac{1}{2\sigma^{\star 2}}}\varphi\left(\Phi^{-1}(3/4)\right)\;,
\end{equation}
where $\varphi$ denotes the p.d.f of a standard Gaussian random variable.

Observe that $\sqrt{n}(\widetilde{\rho}_n - \rho^{\star})$ can be rewritten as follows:
\begin{multline}\label{eq:expansion_rhotilde_1}
\sqrt{n}(\widetilde{\rho}_n - \rho^{\star})=\sqrt{n}\;\frac{F^{-1}_{2,n}(1/2)^2-(1+\rho^{\star}) F^{-1}_{1,n}(1/2)^2}{F^{-1}_{1,n}(1/2)^2}\\
=\sqrt{n}\;\frac{\left(F^{-1}_{2,n-1}(1/2)^2-F_2^{-1}(1/2)^2\right) -(1+\rho^{\star}) \left(F^{-1}_{1,n}(1/2)^2-F_1^{-1}(1/2)^2\right)}{F^{-1}_{1,n}(1/2)^2}\\
+\sqrt{n}\;\frac{F_2^{-1}(1/2)^2-(1+\rho^{\star})F_1^{-1}(1/2)^2}{F^{-1}_{1,n}(1/2)^2}\;.
\end{multline}
By (\ref{eq:F1-1:F2-1}) the last term in the rhs of (\ref{eq:expansion_rhotilde_1}) is equal to zero. Thus,
\begin{multline*}
\sqrt{n}(\widetilde{\rho}_n - \rho^{\star})=\\
\frac{1}{\sqrt{n-1}}\sum_{i=0}^{n-2} \left\{a_2\left(\1_{\{|y_{i+2}-y_i|\leq F_2^{-1}(1/2)\}}-1/2\right)
-a_1 (1+\rho^{\star})\left(\1_{\{|y_{i+1}-y_i|\leq F_1^{-1}(1/2)\}}-1/2\right)\right\}+o_p(1)\;,
\end{multline*}
where, by (\ref{eq:F1prime:F2prime}), 
\begin{eqnarray*}
a_2 & = &-\frac{2 F_2^{-1}(1/2)}{F_2'(F_2^{-1}(1/2))}=-2\sigma^{\star 2}\frac{\Phi^{-1}(3/4)}{\varphi\left(\Phi^{-1}(3/4)\right)}\\
\textrm{ and } a_1 & = & -\frac{2F_1^{-1}(1/2)}{F_1'(F_1^{-1}(1/2))}=-\frac{2\sigma^{\star 2}}{1+\rho^{\star}}\frac{\Phi^{-1}(3/4)}{\varphi\left(\Phi^{-1}(3/4)\right)}  \;.
\end{eqnarray*}
By (\ref{eq:F1-1:F2-1}), $\sqrt{n}(\widetilde{\rho}_n - \rho^{\star})$ can thus be rewritten as follows:
\begin{equation*}
\sqrt{n}(\widetilde{\rho}_n - \rho^{\star})=\frac{1}{\sqrt{n-1}}\sum_{0\leq i\leq n-2}
\Psi(\eta_i,\eta_{i+1},\eta_{i+2})+o_p(1)\;.
\end{equation*}
where $\Psi$ is defined in (\ref{eq:def_Psi}) and 
$(\eta_i)$ is defined in (\ref{eq:ar1}).
Since $\Psi$ is a function on $\rset^3$ with Hermite rank greater than 1 and $(\eta_i)_{i\geq 0}$ is a stationary AR(1) Gaussian process,
(\ref{eq:tcl_rho_tilde}) follows by applying \cite[Theorem 4]{arcones1994limit}.

\subsection{Hints for \eqref{eq:rho_cauchy}}\label{subsec:hints}
Note that if $X$ has a Cauchy($x_0$,$\gamma$) distribution then the characteristic function $\varphi_X$
of $X$ can be written as
$
\varphi_X(t)=\rme^{\rmi x_0 t-\gamma |t|}.
$
Moreover, the cdf $F_X$ of $X$ is such that $F_X^{-1}(3/4)=x_0+\gamma$.
Thus, $\eta_i=\sum_{k\geq 0} (\rho^{\star})^k \varepsilon_{i-k}$ has a 
Cauchy$\left(\frac{x_0}{1-\rho^{\star}},\frac{\gamma}{1-|\rho^{\star}|}\right)$ distribution and
$(\rho^{\star} -1)\eta_i$ has a Cauchy$\left(-x_0,\frac{\gamma|\rho^{\star} -1|}{1-|\rho^{\star}|}\right)$ distribution.
Since $\eta_{i+1}-\eta_i=(\rho^{\star} -1)\eta_i+\varepsilon_i$ is a sum of two independent Cauchy random variables,
it is distributed as a Cauchy$\left(0,\gamma\left(1+\left|\frac{\rho^{\star}-1}{1-|\rho^{\star}|}\right|\right)\right)$ distribution.
In the same way, 
$\eta_{i+2}-\eta_i=({\rho^{\star}}^2 -1)\eta_i+\rho^{\star}\varepsilon_i+\varepsilon_{i+2}$ is a sum 
of three independent Cauchy random variables and has thus a Cauchy$\left(0,2\gamma(1+|\rho^{\star}|)\right)$.
Let $F_1$ and $F_2$ denote the cdf of $(\eta_{i+1}-\eta_i)$ and $(\eta_{i+2}-\eta_i)$, respectively.
By using the properties of the cdf of a Cauchy distribution, we get, on the one hand, that
$F_2^{-1}(3/4)=2\gamma(1+|\rho^{\star}|)$ and, on the other hand, that
$$
F_1^{-1}(3/4)=
\left\{
\begin{array}{cc}
2\gamma,&\textrm{if } \rho^{\star}>0\;,\\
\frac{2\gamma}{1+\rho^{\star}},&\textrm{if } \rho^{\star}<0\;.
\end{array}
\right.
$$
From this we get that
\begin{equation}\label{eq:lim_F2_F1}
\left(\frac{F_2^{-1}(3/4)}{F_1^{-1}(3/4)}\right)^2-1=
\left\{
\begin{array}{cc}
\rho^{\star}(2+\rho^{\star}),&\textrm{if } \rho^{\star}>0\;,\\
{\rho^{\star}}^2({\rho^{\star}}^2-2),&\textrm{if } \rho^{\star}<0\;.
\end{array}
\right.
\end{equation}
The definition of $\widetilde{\widetilde{\rho}}_n$ comes by inverting
these last two functions.


\subsection{Proof of Proposition \ref{Prop:Segment}}\label{subsec:prop:Segment}

In the sequel, we need the following definitions, notations and remarks. Observe that \eqref{eq:bkw} can be rewritten as follows:
\begin{equation}\label{eq:modele_matriciel}
z = \rho^* Bz + T\left(\boldsymbol{t}_n^{\star}\right) \boldsymbol{\delta}^{\star} + \epsilon\;,
\end{equation}
where
\begin{equation}\label{eq:YXE}
z = \left( \begin{array}{c} z_1 \\\vdots \\z_{n} \end{array}\right)\;,
\qquad
Bz = \left( \begin{array}{c} z_0 \\ \vdots \\ z_{n-1} \end{array}\right)\;, 
\qquad
\boldsymbol{\delta}^{\star} = \left( \begin{array}{c} \delta^{\star}_0 \\ \vdots \\ \delta^{\star}_{m} \end{array}\right)\;, 
\qquad 
\epsilon = \left( \begin{array}{c} \epsilon_1 \\ \vdots \\ \epsilon_n \end{array}\right)\;,
\end{equation}
where $\delta_k^{\star} = (1-\rho^{\star}) \mu^{\star}_k$, for $0 \leq k \leq m$,
and $T\left(\boldsymbol{t}\right)$ is an $n \times (m+1)$ matrix where the $k$th column is $( \underset{t_{k-1}}{\underbrace{0,\dots , 0}} \; \underset{t_k-t_{k-1}}{\underbrace{1,\dots ,1}} \; \underset{n - t_k}{\underbrace{0,\dots ,0}} )'$.
Let us define the exact and estimated decorrelated series by
\begin{eqnarray}\label{eq:decor}
w^{\star} & = & z - \rho^{\star} Bz\;, \\
\overline{w} & = & z - \overline{\rho}_n Bz\label{eq:overline_x}\;.
\end{eqnarray}
For any vector subspace $E$ of $\mathbb{R}^{n}$, let $\pi_E$ denote the orthogonal projection of $\mathbb{R}^{n}$ on $E$. 
Let also $\Vert \cdot \Vert$ be the euclidian norm on $\mathbb{R}^{n}$, $\langle \cdot , \cdot \rangle$ the canonical scalar product on  $\mathbb{R}^{n}$ and 
$\Vert \cdot \Vert_{\infty}$ the sup norm. 
For $x$ a vector of $\rset^n$ and $\boldsymbol{t}\in\mathcal{A}_{n,m}$, let
\begin{equation}\label{eq:Jnm}
J_{n,m}\left(x,\boldsymbol{t}\right)
= \frac{1}{n} \left( \Vert \pi_{E_{ \boldsymbol{t}_n^{\star} } }\left( x\right) \Vert^2 - \Vert \pi_{ E_{\boldsymbol{t}}} \left( x \right) \Vert^2 \right)\;,
\end{equation}
written $J_n\left(x,\boldsymbol{t}\right)$ in the sequel for notational simplicity. In \eqref{eq:Jnm}, $E_{\boldsymbol{t}_n^{\star}}$ and $E_{\boldsymbol{t}}$ correspond to the
linear subspaces of $\mathbb{R}^{n}$ generated by the columns of $T\left(\boldsymbol{t}_n^{\star}\right)$ and $T\left(\boldsymbol{t}\right)$, respectively. We shall use the same decomposition
as the one introduced in \cite{LM}:
\begin{equation}\label{eq:Jn_decomp}
J_n\left(x,\boldsymbol{t}\right) =  K_n\left(x,\boldsymbol{t}\right) + V_n\left(x,\boldsymbol{t}\right) + W_n\left(x,\boldsymbol{t}\right)\;,
\end{equation}
where
\begin{eqnarray}
K_n\left(x,\boldsymbol{t}\right) & = & \frac{1}{n}\left\Vert \left(\pi_{E_{\boldsymbol{t}_n^{\star}}}- \pi_{E_{\boldsymbol{t}}}\right)\mathbb{E}x\right\Vert^2\;, \label{eq:Kn} \\ 
V_n\left(x,\boldsymbol{t}\right) & = & \frac{1}{n}\left(\left\Vert \pi_{E_{\boldsymbol{t}_n^{\star}}} \left( x-\mathbb{E}x\right)\right\Vert^2 - \left\Vert \pi_{E_{\boldsymbol{t}}} \left( x-\mathbb{E}x\right)\right\Vert^2\right)\;, \label{eq:Vn}\\
W_n\left(x,\boldsymbol{t}\right) & = & \frac{2}{n}\left( \left\langle \pi_{E_{\boldsymbol{t}_n^{\star}}} \left( x-\mathbb{E}x\right), \pi_{E_{\boldsymbol{t}_n^{\star}}} \left( \mathbb{E}x \right)\right\rangle - \left\langle \pi_{E_{\boldsymbol{t}}} \left( x-\mathbb{E}x\right), \pi_{E_{\boldsymbol{t}}} \left( \mathbb{E}x \right)\right\rangle  \right)\;. \label{eq:Wn}
\end{eqnarray}
We shall also use the following notations:
\begin{eqnarray}
\underline{\lambda} & = & \underset{1\leq k\leq m}{\min} \left\vert \delta_k^{\star}-\delta_{k-1}^{\star} \right\vert\;,\label{eq:underline_lambda}\\
\overline{\lambda} & = & \underset{1\leq k\leq m}{\max} \left\vert \delta_k^{\star}-\delta_{k-1}^{\star} \right\vert\;,\label{eq:overline_lambda}\\
\Delta_{\boldsymbol{\tau}^{\star}} & = & \underset{1\leq k\leq m+1}{\min} \left(\tau_k^{\star} - \tau_{k-1}^{\star}\right)\;,\label{eq:Delta_tau*}\\
\mathcal{C}_{\nu, \gamma, n, m} & = & \left\lbrace \boldsymbol{t} \in \mathcal{A}_{n,m}; \nu\underline{\lambda}^{-2}\leq \Vert \boldsymbol{t} - \boldsymbol{t}_n^{\star}\Vert\leq n\gamma\Delta_{\boldsymbol{\tau}^{\star}}\right\rbrace\;,\label{eq:C_alpha_gamma_n}\\
\mathcal{C}_{\nu, \gamma, n, m}' & = & \mathcal{C}_{\nu, \gamma, n, m}\cap \left\lbrace \boldsymbol{t}\in\mathcal{A}_{n,m}; \forall k = 1,\dots, m, t_k\geq t_{n,k}^{\star}\right\rbrace\;,
\label{eq:C'_alpha_gamma_n} \\
\mathcal{C}_{\nu, \gamma, n, m}'\left(\mathcal{I}\right) & = & \left\lbrace \boldsymbol{t}\in \mathcal{C}_{\nu, \gamma, n, m}' ; \right. \nonumber\\
 & & \left. \forall k\in\mathcal{I}, \nu\underline{\lambda}^{-2}\leq t_k - t_{n,k}^{\star}\leq n\gamma\Delta_{\boldsymbol{\tau}^{\star}} \textit{ and } \forall k\notin \mathcal{I}, t_k - t_{n,k}^{\star} < \nu\underline{\lambda}^{-2}\right\rbrace\label{eq:C'alpha_gamma_n_I}\;,
\end{eqnarray}
for any $\nu>0$, $0<\gamma <1/2$ and $\mathcal{I}\subset\left\lbrace 1,\dots ,m\right\rbrace$.
We shall also need the following lemmas in order to prove Proposition \ref{Prop:Segment} which are proved below.
\begin{lemma}\label{lem:rateXY}
Let $\left(z_0 ,\dots ,z_n\right)$ be defined by \eqref{eq:modele_new} or \eqref{eq:bkw}, then
\begin{eqnarray}\label{eq:rateXY}
\Vert Bz \Vert & = & O_P\left(n^{1/2}\right)\;,\\
\Vert z \Vert & = & O_P\left(n^{1/2}\right)\;,
\end{eqnarray}
as $n$ tends to infinity, where $Bz$ and $z$ are defined in \eqref{eq:YXE}.
\end{lemma}

\begin{lemma}\label{lem:BoundedUnifBound}
Let $\left(z_0 ,\dots ,z_n\right)$ be defined by \eqref{eq:modele_new} or \eqref{eq:bkw} then,
for all $\boldsymbol{t}\in\mathcal{A}_{n,m}$,
\begin{equation*}
\left\vert J_n \left(\overline{w}, \boldsymbol{t}\right) - J_n \left(w^{\star}, \boldsymbol{t}\right)\right\vert \leq 2\frac{\left\vert \rho^{\star} - \overline{\rho} \right\vert}{n} \left\Vert Bz\right\Vert \left(\left\vert \rho^{\star} + \overline{\rho} \right\vert \left\Vert Bz\right\Vert + 2 \left\Vert z\right\Vert\right) = {O_P\left(n^{-1/2}\right)} = o_P\left(1\right)\;,
\end{equation*}
 {as $n\to\infty$}, where $J_n$ is defined in \eqref{eq:Jnm}, $Bz$ and $z$ are defined in \eqref{eq:YXE}, $w^{\star}$ is defined in \eqref{eq:decor} and $\overline{w}$ is defined in \eqref{eq:overline_x}.
\end{lemma}

\begin{lemma}\label{lem:consistency}
Under the assumptions of Proposition \ref{Prop:Segment},
$\Vert\boldsymbol{\overline{\tau}_n}-\boldsymbol{\tau}^{\star}\Vert_{\infty}$ converges in probability to $0$, as $n$ tends to infinity.
\end{lemma}

\begin{lemma}\label{lem:almost72LM}
Under the assumptions of Proposition \ref{Prop:Segment} and for any $\nu>0$, $0<\gamma <1/2$ and $\mathcal{I}\subset\left\lbrace 1,\dots ,m\right\rbrace$, 
\begin{equation*}
P\left(\min_{\boldsymbol{t}\in \mathcal{C}_{\nu, \gamma, n, m}'\left(\mathcal{I}\right)} \left(\frac{1}{2} K_n\left(w^{\star},\boldsymbol{t}\right) + V_n\left(w^{\star},\boldsymbol{t}\right) + W_n\left(w^{\star},\boldsymbol{t}\right)\right)\leq 0 \right) \longrightarrow 0\;,\; \textrm{as } n\to\infty\;,
\end{equation*}
where $\mathcal{C}_{\nu, \gamma, n, m}'\left(\mathcal{I}\right)$ is defined in \eqref{eq:C'alpha_gamma_n_I} and $w^{\star}$ is defined in \eqref{eq:decor}.
\end{lemma}

\begin{lemma}\label{lem:almost72LM2}
Under the assumptions of Proposition \ref{Prop:Segment} and for any $\nu>0$, $0<\gamma <1/2$ and $\mathcal{I}\subset\left\lbrace 1,\dots ,m\right\rbrace$, 
\begin{equation*}
P\left(\min_{\boldsymbol{t}\in \mathcal{C}_{\nu, \gamma, n, m}'\left(\mathcal{I}\right)} J_n\left(\overline{w},\boldsymbol{t}\right)\leq 0\right) \longrightarrow 0
\;,\; \textrm{as } n\to\infty\;,
\end{equation*}
where $\mathcal{C}_{\nu, \gamma, n, m}'\left(\mathcal{I}\right)$ is defined in \eqref{eq:C'alpha_gamma_n_I} and $\overline{w}$ is defined in \eqref{eq:overline_x}.
\end{lemma}

\begin{lemma}\label{lem:rateT}
Under the assumptions of Proposition \ref{Prop:Segment},
\begin{equation*}
\Vert\boldsymbol{\widehat{\tau}_n}(z, \overline{\rho}_n)-\boldsymbol{\tau}^{\star} \Vert_{\infty} = O_P\left(n^{-1}\right)\;.
\end{equation*}
\end{lemma}


\begin{proof}[Proof of Lemma \ref{lem:rateXY}]
Without loss of generality, assume $\left(z_0 ,\dots ,z_n \right)$ is defined by \eqref{eq:bkw}.
$ \Vert z \Vert^2 - \Vert Bz \Vert^2 = z_n^2 - z_0^2 = O_P\left(1\right) $
thus we only need to prove \eqref{eq:rateXY}. Observe that 
$
\Vert Bz\Vert^2 = \sum_{i=0}^{n-1}z_i^2 \leq 2 \sum_{i=0}^{n-1}(z_i-\PE(z_i))^2+2 \sum_{i=0}^{n-1}\PE(z_i)^2.
$
Since $((z_i-\PE(z_i))^2)$ is stationary with autocovariance function $\gamma$ such that $\gamma(h)\to 0$ as $h\to\infty$, 
\cite[Theorem 7.1.1]{brockwell} implies that 
$
\Vert Bz\Vert^2=O_P (n).
$
\end{proof}

\begin{proof}[Proof of Lemma \ref{lem:BoundedUnifBound}]
By \eqref{eq:decor}, $\overline{w}=w^{\star}+(\rho^{\star}-\overline{\rho}_n)Bz$. Thus, by \eqref{eq:Jnm}, we get
\begin{multline}\label{eq:Jnm_diff}
J_n \left(\overline{w}, \boldsymbol{t}\right) - J_n \left(w^{\star}, \boldsymbol{t}\right) = \\
\frac{\left(\rho^{\star} - \overline{\rho}_n\right)^2}{n} \left\Vert \pi_{E_{\boldsymbol{t}^{\star}}} \left( Bz \right)\right\Vert^2 + \frac{2\left(\rho^{\star} - \overline{\rho}_n\right)}{n} \left\langle \pi_{E_{\boldsymbol{t}^{\star}}}\left(z-\rho^{\star} Bz \right), \pi_{E_{\boldsymbol{t}^{\star}}}\left(Bz\right)\right\rangle\\ 
-\frac{\left(\rho^{\star} - \overline{\rho}_n\right)^2}{n} \left\Vert \pi_{E_{\boldsymbol{t}}}\left(Bz\right)\right\Vert^2 - \frac{2\left(\rho^{\star} - \overline{\rho}_n\right)}{n}\left\langle \pi_{E_{\boldsymbol{t}}}\left(z-\rho^{\star} Bz \right), \pi_{E_{\boldsymbol{t}}}\left(Bz\right)\right\rangle\;.
\end{multline}
Observe that the sum of the first two term in the rhs of \eqref{eq:Jnm_diff} can be rewritten as follows:
\begin{multline*}
\frac{1}{n} (\rho^{\star} - \overline{\rho}_n) \left\langle  \pi_{E_{\boldsymbol{t}^{\star}}}(Bz),(\rho^{\star} - \overline{\rho}_n)\pi_{E_{\boldsymbol{t}^{\star}}}(Bz)+2\pi_{E_{\boldsymbol{t}^{\star}}}(z-\rho^{\star} Bz)\right\rangle\\
=\frac{1}{n} (\rho^{\star} - \overline{\rho}_n) \left\langle  \pi_{E_{\boldsymbol{t}^{\star}}}(Bz),\pi_{E_{\boldsymbol{t}^{\star}}}\left(2z-(\rho^{\star}+\overline{\rho}_n)Bz\right)\right\rangle\;.
\end{multline*}
Since the same can be done for the last two terms in the rhs of (\ref{eq:Jnm_diff}), the Cauchy-Schwarz inequality and the $1$-Lipschitz property of 
projections give
\begin{equation*}
\left\vert J_n \left(\overline{w}, \boldsymbol{t}\right) - J_n \left(w^{\star}, \boldsymbol{t}\right)\right\vert \leq 2\frac{\left\vert \rho^{\star} - \overline{\rho}_n \right\vert}{n} \left\Vert Bz\right\Vert \left(\left\vert \rho^{\star} + \overline{\rho}_n \right\vert \left\Vert Bz\right\Vert + 2 \left\Vert z\right\Vert\right)\;.
\end{equation*}
The conclusion follows from (\ref{eq:hypRhoRate}) and Lemma \ref{lem:rateXY}.
\end{proof}

\begin{proof}[Proof of Lemma \ref{lem:consistency}]
\cite[proof of Theorem 3]{LM} give the following bounds for any $\boldsymbol{t} \in\mathcal{A}_{n,m} $:
\begin{eqnarray}
K_n\left(w^{\star},\boldsymbol{t} \right) & \geq & \underline{\lambda}^2 \min\left(\frac{1}{n}\max_{1\leq k \leq m}\left\vert t_k - t_{n,k}^{\star}\right\vert , \Delta_{\boldsymbol{\tau}^{\star}}\right)\;, \label{eq:LMbounds}\\
V_n\left(w^{\star},\boldsymbol{t}\right) & \geq & -\frac{2\left(m+1\right)}{n\Delta_n}\left(\max_{1\leq s\leq n} \left(\sum_{i=1}^s \epsilon_i\right)^2 + \max_{1\leq s\leq n} \left(\sum_{i=n-s}^{n} \epsilon_i\right)^2\right)\;, \label{eq:LMboundV}\\
\left\vert  W_n\left(w^{\star},\boldsymbol{t} \right) \right\vert & \leq & \frac{3\left(m+1\right)^2 \overline{\lambda}}{n} \left(\max_{1\leq s\leq n} \left\vert\sum_{i=1}^s \epsilon_i\right\vert + \max_{1\leq s\leq n} \left\vert \sum_{i=n-s}^{n} \epsilon_i\right\vert \right)\;,
\end{eqnarray}
where $\Delta_{\boldsymbol{\tau}^{\star}}$, $\underline{\lambda}$ and $\overline{\lambda}$ are defined in (\ref{eq:Delta_tau*}), (\ref{eq:underline_lambda}) and 
(\ref{eq:overline_lambda}), respectively.
For any $\nu>0$, define, as in \cite[proof of Theorem 3]{LM},
\begin{equation}\label{eq:Cn_alpha}
\mathcal{C}_{n,m,\nu} = \left\lbrace \boldsymbol{t}\in \mathcal{A}_{n,m} ; \left\Vert \boldsymbol{t} - \boldsymbol{t}_{n}^{\star}\right\Vert_{\infty} \geq n\nu\right\rbrace \;.
\end{equation}
For $0<\nu<\Delta_{\boldsymbol{\tau}^{\star}}$, we have:
\begin{eqnarray}\label{eq:ineq_alpha}
P\left(\left\Vert \boldsymbol{\widehat{t}}_n(z, \overline{\rho}_n) - \boldsymbol{t}_{n}^{\star} \right\Vert_{\infty} \geq n\nu\right)  & \leq & P \left( \min_{\boldsymbol{t}\in\mathcal{C}_{n,m,\nu}} J_n \left(\overline{w}, \boldsymbol{t} \right) \leq 0\right) \\
 & \leq &P \left( \min_{\boldsymbol{t}\in\mathcal{C}_{n,m,\nu}} \left(J_n \left(\overline{w}, \boldsymbol{t} \right) - J_n \left(w^{\star}, \boldsymbol{t}\right) \right)\leq -\nu\underline{\lambda}^2 \right)\nonumber \\
 & + & P \left( \min_{\boldsymbol{t}\in\mathcal{C}_{n,m,\nu}} \left(V_n\left(w^{\star}, \boldsymbol{t}\right) + W_n\left(w^{\star}, \boldsymbol{t}\right)\right) \leq  -\nu\underline{\lambda}^2 \right)\nonumber \\
 & \leq & P \left( \min_{\boldsymbol{t}\in\mathcal{C}_{n,m,\nu}} \left(J_n \left(\overline{w}, \boldsymbol{t} \right) - J_n \left(w^{\star}, \boldsymbol{t} \right)\right)\leq -\nu\underline{\lambda}^2 \right)\nonumber\\
 & + & P \left( \max_{1\leq s\leq n}{\left(\sum_{i=1}^s \epsilon_i\right)^2} + \max_{1 \leq s \leq n} \left(\sum_{i=n-s}^{n} \epsilon_i\right)^2 \geq c\underline{\lambda}^2 n\Delta_n \nu \right)\nonumber\\
 & + & P \left( \max_{1 \leq s \leq n}{\left\vert\sum_{i=1}^s \epsilon_i\right\vert} + \max_{1 \leq s \leq n} \left\vert\sum_{i=n-s}^{n} \epsilon_i\right\vert \geq c\underline{\lambda}^2 n \nu \overline{\lambda}^{-1} \right) \nonumber
\end{eqnarray}
\\
for some positive constant $c$. The last two terms of this sum go to $0$ when $n$ goes to infinity (see \cite[proof of Theorem 3]{LM}). To show that the first term shares the same property, it suffices to show that $J_n \left(\overline{w}, \boldsymbol{t}\right) - J_n\left(w^{\star},\boldsymbol{t}\right)$ is bounded uniformly in $\boldsymbol{t}$ by a sequence of random variables which converges to $0$ in probability. This result holds by Lemma \ref{lem:BoundedUnifBound}.
\end{proof}

\begin{proof}[Proof of Lemma \ref{lem:almost72LM}]
Using \cite[(64),(65) and (66)]{LM}, one can show the bound \cite[(73)]{LM} on 
$$P\left(\min_{\boldsymbol{t}\in \mathcal{C}_{\nu, \gamma, n, m}'\left(\mathcal{I}\right)} \left(K_n\left(w^{\star},\boldsymbol{t}\right) + V_n\left(w^{\star},\boldsymbol{t}\right) + W_n\left(w^{\star},\boldsymbol{t}\right)\right)\leq 0 \right).$$
 Using the same arguments, we have the same bound on 
$$P\left(\min_{\boldsymbol{t} \in \mathcal{C}_{\nu, \gamma, n, m}'\left(\mathcal{I}\right)} \left(\frac{1}{2} K_n\left(w^{\star},\boldsymbol{t} \right) + V_n\left(w^{\star},\boldsymbol{t} \right) + W_n\left(w^{\star},\boldsymbol{t} \right)\right)\leq 0 \right).$$
We conclude using \cite[(67)-(71)]{LM}.
\end{proof}

\begin{proof}[Proof of Lemma \ref{lem:almost72LM2}]
By (\ref{eq:Jn_decomp}),
\begin{multline*}
P\left(\min_{\boldsymbol{t}\in \mathcal{C}_{\nu, \gamma, n, m}'\left(\mathcal{I}\right)} J_n\left(\overline{w},\boldsymbol{t}\right)\leq 0\right)\\
\leq P\left(\min_{\boldsymbol{t}\in \mathcal{C}_{\nu, \gamma, n, m}'\left(\mathcal{I}\right)} \left(J_n\left(\overline{w},\boldsymbol{t}\right) - J_n\left(w^{\star},\boldsymbol{t}\right) + \frac{1}{2} K_n\left(w^{\star},\boldsymbol{t}\right)\right)\leq 0\right)\\
 +  P\left(\min_{\boldsymbol{t}\in \mathcal{C}_{\nu, \gamma, n, m}'\left(\mathcal{I}\right)} \left(\frac{1}{2} K_n\left(w^{\star},\boldsymbol{t}\right) + V_n\left(w^{\star},\boldsymbol{t}\right) + W_n\left(w^{\star},\boldsymbol{t}\right)\right)\leq 0 \right)\;.
\end{multline*}
By Lemma \ref{lem:almost72LM}, the conclusion thus follows if
\begin{equation*}
P \left( \min_{ \boldsymbol{t} \in \mathcal{C}_{\nu, \gamma, n, m}' \left(\mathcal{I}\right) } 
\left(J_n \left( \overline{w}, \boldsymbol{t} \right) 
- J_n \left( w^{\star}, \boldsymbol{t} \right) 
+ \frac{1}{2} K_n \left( w^{\star} , \boldsymbol{t} \right)\right) \leq 0 \right ) 
 \longrightarrow 0\;, \textrm{ as } n\to\infty\;.
\end{equation*}
Since $\underset{\boldsymbol{t}\in \mathcal{C}_{\nu, \gamma, n, m}'\left(\mathcal{I}\right)}{\min} K_n\left(w^{\star},\boldsymbol{t}\right) \geq \left(1-\gamma\right)\Delta_{\boldsymbol{\tau}^{\star}}\nu$ (see \cite[(65)]{LM}), 
$$P\left( \min_{\boldsymbol{t}\in \mathcal{C}_{\nu, \gamma, n, m}'\left(\mathcal{I}\right)} \left( J_n\left(\overline{w},\boldsymbol{t}\right)-J_n\left(w^{\star},\boldsymbol{t}\right)+\frac{1}{2} K_n\left(w^{\star},\boldsymbol{t}\right)\right)\leq 0\right)$$
$$ \leq P\left(\min_{\boldsymbol{t}\in \mathcal{C}_{\nu, \gamma, n, m}'\left(\mathcal{I}\right)} \left(J_n\left(\overline{w},\boldsymbol{t} \right)-J_n\left(w^{\star},\boldsymbol{t}\right)\right) \leq \frac{1}{2} \left(\gamma-1\right)\Delta_{\boldsymbol{\tau}^{\star}}\nu\right)\;, $$
 and we conclude by Lemma \ref{lem:BoundedUnifBound}.
\end{proof}

\begin{proof}[Proof of Lemma \ref{lem:rateT}]For notational simplicity, $\boldsymbol{\widehat{t}}_n(z, \overline{\rho}_n)$ will be replaced by $\boldsymbol{\overline{t}}_n$ in this proof.
Since for any $\nu>0$,
$$
P\left(\Vert  \boldsymbol{\overline{t}}_n - \boldsymbol{t}_n^{\star}\Vert_\infty< \nu \underline{\lambda}^{-2}\right)
=P( \Vert\boldsymbol{\overline{t}}_n - \boldsymbol{t}_n^{\star} \Vert_\infty\leq n \gamma \Delta_{\boldsymbol{\tau}^{\star}} ) - P( \boldsymbol{\overline{t}}_n \in \mathcal{C}_{\nu , \gamma ,n , m })\;,
$$
it is enough, by Lemma \ref{lem:consistency}, to prove that
\begin{equation*}
P\left(\boldsymbol{\overline{t}}_n \in \mathcal{C}_{\nu, \gamma, n, m}\right) \longrightarrow 0\;,\; \textrm{as } n\to\infty\;,
\end{equation*}
for all $\nu>0$ and $0<\gamma<1/2 $.
Since
$$  \mathcal{C}_{\nu, \gamma, n, m} = \underset{\mathcal{I}\subset\lbrace 1,\dots ,m\rbrace}{\bigcup} \mathcal{C}_{\nu, \gamma, n, m}\cap \left\lbrace \boldsymbol{t}\in\mathcal{A}_{n,m}; \forall k \in \mathcal{I}, t_k\geq t_{n,k}^{\star}\right\rbrace \; , $$
we shall only study one set in the union without loss of generality and prove that
\begin{equation*}
P\left(\boldsymbol{\overline{t}}_n \in \mathcal{C}_{\nu, \gamma, n, m}'\right) \longrightarrow 0\;,\; \textrm{as } n\to\infty\;,
\end{equation*}
where $\mathcal{C}_{\nu, \gamma, n, m}'$ is defined in (\ref{eq:C'_alpha_gamma_n}). Since
$\mathcal{C}_{\nu, \gamma, n, m}' = \underset{\mathcal{I}\subset\lbrace 1,\dots ,m\rbrace}{\bigcup} \mathcal{C}_{\nu, \gamma, n, m}'\left(\mathcal{I}\right)$,
we shall only study one set in the union without loss of generality and prove that
\begin{equation*}
P\left(\boldsymbol{\overline{t}}_n \in \mathcal{C}_{\nu, \gamma, n, m}'\left(\mathcal{I}\right)\right) \longrightarrow 0\;,\; \textrm{as } n\to\infty\;.
\end{equation*}
Since 
$$ P\left(\boldsymbol{\overline{t}}_n \in \mathcal{C}_{\nu, \gamma, n, m}'\left(\mathcal{I}\right)\right) \leq P\left(\min_{\boldsymbol{t} \in \mathcal{C}_{\nu, \gamma, n, m}'\left(\mathcal{I}\right)} J_n\left(\overline{w},\boldsymbol{t} \right)\leq 0\right)\;,$$
the proof is complete by Lemma \ref{lem:almost72LM2}.

\end{proof}

\begin{proof}[Proof of Proposition \ref{Prop:Segment}]For notational simplicity, $\boldsymbol{\widehat{\delta}}_n(z, \overline{\rho}_n)$ will be replaced by $\boldsymbol{\overline{\delta}}_n$ in this proof.
By Lemma \ref{lem:rateT}, the last result to show is
\begin{equation*}
\Vert \boldsymbol{\overline{\delta}_n}-\boldsymbol{\delta}^{\star}\Vert = O_P\left(n^{-1/2}\right)\;,
\end{equation*}
that is, for all $k$,
$
\overline{\delta}_{n,k} - \delta_k^{\star} = O_P\left(n^{-1/2}\right).
$
By \eqref{eq:decor} and \eqref{eq:overline_x},
\begin{eqnarray*}
\overline{\delta}_{n,k} & = & \frac{1}{\overline{t}_{n,k+1}-\overline{t}_{n,k}} \sum_{i=\overline{t}_{n,k}+1}^{\overline{t}_{n,k+1}}\overline{w}_i\\
 & = & \frac{1}{n\left(\overline{\tau}_{n,k+1}-\overline{\tau}_{n,k}\right)}\left( \sum_{i=\overline{t}_{n,k}+1}^{\overline{t}_{n,k+1}}w_i^{\star} + \left(\rho^{\star} 
-\overline{\rho}_n\right)\sum_{i=\overline{t}_{n,k}+1}^{\overline{t}_{n,k+1}}z_{i-1} \right)\;.
\end{eqnarray*}
By the Cauchy-Schwarz inequality,
\begin{eqnarray*}
\left\vert \sum_{i=\overline{t}_{n,k}+1}^{\overline{t}_{n,k+1}}z_{i-1} \right\vert & \leq & \left(\overline{t}_{n,k+1}-\overline{t}_{n,k}\right)^{1/2}
\left(z_{\overline{t}_{n,k}}^2+ \dots +z_{\overline{t}_{n,k+1}-1}^2\right)^{1/2} 
  \leq  n^{1/2} \left\Vert Bz \right\Vert
  =  O_P\left(n\right)\;,
\end{eqnarray*}
where the last equality comes from Lemma \ref{lem:rateXY}. Hence by (\ref{eq:hypRhoRate}) and Lemma \ref{lem:rateT},
\begin{eqnarray*}
\overline{\delta}_{n,k} & = & \frac{1}{n\left(\overline{\tau}_{n,k+1}-\overline{\tau}_{n,k}\right)} \sum_{i=\overline{t}_{n,k}+1}^{\overline{t}_{n,k+1}}w_i^{\star} + O_P\left(n^{-1/2}\right)\\
 & = & \frac{1}{n\left(\overline{\tau}_{n,k+1}-\overline{\tau}_{n,k}\right)} \left(\sum_{i=\overline{t}_{n,k}+1}^{\overline{t}_{n,k+1}}\mathbb{E} w_i^{\star} + \sum_{i=\overline{t}_{n,k}+1}^{\overline{t}_{n,k+1}} \epsilon_i\right) + O_P\left(n^{-1/2}\right),
\end{eqnarray*}
where the last equality comes from (\ref{eq:modele_matriciel}) and (\ref{eq:decor}).
Let us now prove that 
\begin{equation}\label{eq:clt_aleatoire}
\frac{1}{n\left(\overline{\tau}_{n,k+1}-\overline{\tau}_{n,k}\right)} \sum_{i=\overline{t}_{n,k}+1}^{\overline{t}_{n,k+1}} \epsilon_i = O_P\left(n^{-1/2}\right)\;.
\end{equation}
By Lemma \ref{lem:consistency}, $n^{-1}\left(\overline{\tau}_{n,k+1}-\overline{\tau}_{n,k}\right)^{-1}=O_P(n^{-1})$. Moreover,
\begin{equation}\label{eq:eps_decomp}
\sum_{i=\overline{t}_{n,k}+1}^{\overline{t}_{n,k+1}} \epsilon_i = \sum_{i=t_{n,k}^{\star}+1}^{t_{n,k+1}^{\star}} \epsilon_i \pm \sum_{i=\overline{t}_{n,k}+1}^{t_{n,k}^{\star}} \epsilon_i \pm \sum_{i=t_{n,k+1}^{\star}}^{\overline{t}_{n,k+1}+1} \epsilon_i\;.
\end{equation}
By the {Chebyshev inequality}, the first term in the rhs of \eqref{eq:eps_decomp} is $O_P(n^{1/2})$. By using the Cauchy-Schwarz inequality, we get that 
the second term of \eqref{eq:eps_decomp} satisfies:
$|\sum_{i=\overline{t}_{n,k}+1}^{t_{n,k}^{\star}} \epsilon_i|\leq |t_{n,k}^{\star}-\overline{t}_{n,k}|^{1/2} \left(\sum_{i=1}^n  \epsilon_i^2\right)^{1/2}=O_P(1) O_P(n^{1/2})=O_P(n^{1/2})$, by Lemma 
\ref{lem:rateT}. The same holds for the last term in the rhs of \eqref{eq:eps_decomp}, which gives \eqref{eq:clt_aleatoire}.
Hence,
\begin{eqnarray*}
\overline{\delta}_{n,k} - \delta_k^{\star} & = & \frac{1}{n\left(\overline{\tau}_{n,k+1}-\overline{\tau}_{n,k}\right)} \sum_{i=\overline{t}_{n,k}+1}^{\overline{t}_{n,k+1}}\left(\mathbb{E} w_i^{\star} - \delta_k^{\star}\right) + O_P\left(n^{-1/2}\right)\\
 & = &  \frac{1}{n\left(\overline{\tau}_{n,k+1}-\overline{\tau}_{n,k}\right)} \sum_{i\in\left\lbrace\overline{t}_{n,k}+1,\dots ,\overline{t}_{n,k+1}\right\rbrace \setminus \left\lbrace t_{n,k}^{\star}+1,\dots ,t_{n,k+1}^{\star}\right\rbrace}\left(\mathbb{E} w_i^{\star} - \delta_k^{\star}\right) + O_P\left(n^{-1/2}\right)\;,
\end{eqnarray*}
and then
\begin{eqnarray*}
\left\vert \overline{\delta}_{n,k} - \delta_k^{\star} \right\vert & \leq & \frac{1}{n\left(\overline{\tau}_{n,k+1}-\overline{\tau}_{n,k}\right)} \sharp\left\lbrace\overline{t}_{n,k}+1,\dots ,\overline{t}_{n,k+1}\right\rbrace \setminus \left\lbrace t_{n,k}^{\star}+1,\dots ,t_{n,k+1}^{\star}\right\rbrace \max_{l= 0,\dots,m}\left\vert \delta_l^{\star} - \delta_k^{\star}\right\vert \\
 & + & O_P\left(n^{-1/2}\right)\;.
\end{eqnarray*}

We conclude by using Lemma \ref{lem:rateT} to get $\sharp\left\lbrace\overline{t}_{n,k}+1,\dots ,\overline{t}_{n,k+1}\right\rbrace \setminus \left\lbrace t_{n,k}^{\star}+1,\dots ,t_{n,k+1}^{\star}\right\rbrace = O_P\left(1\right)$ and Lemma \ref{lem:consistency} to get $\left(\overline{\tau}_{n,k+1}-\overline{\tau}_{n,k}\right)^{-1} = O_P\left(1\right)$.

\end{proof}

\subsection{Proof of Proposition \ref{Prop:Segment2}}\label{subsec:prop:Segment2}

The connection between models \eqref{eq:modele_new} and \eqref{eq:bkw} is made by the following lemmas.

\begin{lemma} \label{Lem:YZ}
 Let $(y_0, \dots y_n)$ be defined by \eqref{eq:modele_new} and let
 \begin{eqnarray}
 v^{\star}_i & = & y_i - \rho^{\star} y_{i-1}\label{eq:v_star},\\
 \Delta^{\star}_i & = & \begin{cases}
-\rho^{\star} \left(\mu^{\star}_k  - \mu^{\star}_{k-1}\right) \textit{ if } i = t_{n,k}^{\star}+1,\\
0, \textit{ otherwise, }
\end{cases} \label{eq:delta_star}
 \end{eqnarray} where the $\mu_k^{\star}$'s are defined in \eqref{eq:modele_new}, then the process 
\begin{equation}\label{eq:vw_star}
 w^{\star}_i = v^{\star}_i + \Delta^{\star}_i
\end{equation}
has the same distribution as $z_i - \rho^{\star} z_{i-1}$ where $(z_0, \dots z_n)$ is defined by \eqref{eq:bkw}. Such a process $(z_0, \dots z_n)$ can be constructed recursively as

\begin{equation}\label{eq:z_rec}
\begin{cases}
z_0 & =  y_0\\
z_i & =  w^{\star}_i + \rho^{\star} z_{i-1} \textit{ for } i>0.
\end{cases}
\end{equation}
\end{lemma}

\begin{lemma}\label{Lem:YZbar}
 Let $(y_0, \dots y_n)$ be defined by \eqref{eq:modele_new} and let $z$ be defined by (\ref{eq:v_star}-- \ref{eq:z_rec}). Then
\begin{equation}\label{eq:vw}
\overline{w}_i = \overline{v}_i + \overline{\Delta}_i
\end{equation}
where
\begin{eqnarray}
\overline{v}_i & = & y_i - \overline{\rho}_n y_{i-1}\label{eq:v_bar}\\
\overline{w}_i & = & z_i - \overline{\rho}_n z_{i-1}\label{eq:w_bar}\\
\overline{\Delta}_i & = & \Delta^{\star}_i + \left(\rho^{\star} - \overline{\rho}_n\right)\left(z_{i-1}-y_{i-1}\right)\; .\label{eq:Delta_bar}
\end{eqnarray}
\end{lemma}

\begin{lemma} \label{Lem:Delta_order}
 Let $\overline{\Delta}=(\overline{\Delta}_i)_{0\leq i\leq n}$ as defined in \eqref{eq:Delta_bar}. 
Then $\left\Vert\overline{\Delta}\right\Vert = O_P\left(1\right)$.
\end{lemma}


\begin{proof}[Proof of Lemma \ref{Lem:YZ}]
Let $z$ being defined by \eqref{eq:z_rec}. Using \eqref{eq:vw_star}, we get, for all $0\leq k \leq m , t_{n,k}^{\star}<i\leq t_{n,k+1}^{\star}$
\begin{multline*}
\left(z_i - \mu_k^{\star} \right) - \rho^{\star} \left(z_{i-1} - \mu_k^{\star}\right) = \left(y_i - \mu_k^{\star}\right) - \rho^{\star} \left(y_{i-1} - \mu_k^{\star}\right) + \Delta_i^{\star} \\
 = 
\begin{cases}
\left(y_i - \mu_k^{\star}\right) - \rho^{\star} \left(y_{i-1} - \mu_{k-1}^{\star}\right) \textit{ if } i = t_{n,k}^{\star}+1\\
\left(y_i - \mu_k^{\star}\right) - \rho^{\star} \left(y_{i-1} - \mu_k^{\star}\right) \textit{ otherwise.}
\end{cases}
\end{multline*}
This expression equals $\left(y_i - \PE \left(y_i\right)\right) - \rho^{\star} \left(y_{i-1} - \PE\left(y_{i-1}\right)\right) = \eta_i - \rho^{\star} \eta_{i-1} = \epsilon_i$ by \eqref{eq:modele_new} and \eqref{eq:ar1}. Then $z$ satisfies \eqref{eq:bkw}.
\end{proof}

\begin{proof}[Proof of Lemma \ref{Lem:YZbar}]
The proof of Lemma \ref{Lem:YZbar} is straightforward.
\end{proof}

\begin{proof}[Proof of Lemma \ref{Lem:Delta_order}]
\eqref{eq:Delta_bar} can be written as
$$
\overline{\Delta} = \Delta^{\star} + \left(\rho^{\star} - \overline{\rho}_n\right)\left(By-Bz\right)
$$
where $\Delta^{\star}=\left(\Delta_i^{\star}\right)_{1\leq i\leq n}$, $By = \left(y_{i-1}\right)_{1\leq i\leq n}$  and $Bz$ is defined in \eqref{eq:YXE}. By the triangle inequality,
\begin{equation}\label{eq:triangular_Delta_bar}
\left\Vert\overline{\Delta}\right\Vert \leq \left\Vert\Delta^{\star}\right\Vert + \left\vert\rho^{\star} - \overline{\rho}_n\right\vert \left(\left\Vert By\right\Vert+\left\Vert Bz\right\Vert\right).
\end{equation}
Since $\left\Vert\Delta^{\star}\right\Vert$ is constant it is bounded.
The conclusion follows from \eqref{eq:triangular_Delta_bar}, \eqref{eq:hypRhoRate} and Lemma \ref{lem:rateXY}.
\end{proof}


\begin{proof}[Proof of Proposition \ref{Prop:Segment2}]
Let $y$, $z$, $\overline{v}$, $\overline{w}$ and $\overline{\Delta}$ be defined in Lemma \ref{Lem:YZbar}.

Using \eqref{eq:Jnm} and Lemma \ref{Lem:YZbar}, we get
\begin{equation}
J_n\left(\overline{v},\boldsymbol{t}\right) = J_n\left(\overline{w},\boldsymbol{t}\right) + J_n\left(\overline{\Delta},\boldsymbol{t}\right) - \frac{2}{n}\left(\left\langle \pi_{E_{\boldsymbol{t}_n^{\star}}}\left(\overline{w}\right),\pi_{E_{\boldsymbol{t}_n^{\star}}}\left(\overline{\Delta}\right)\right\rangle - \left\langle \pi_{E_{\boldsymbol{t}}}\left(\overline{w}\right),\pi_{E_{\boldsymbol{t}}}\left(\overline{\Delta}\right)\right\rangle\right).
\end{equation}
By the Cauchy-Schwarz inequality and the $1$-Lipschitz property of projections, we have
\begin{eqnarray}\label{eq:CSLipschitz}
\left\vert J_n\left(\overline{\Delta},\boldsymbol{t} \right) \right\vert & \leq & \frac{2}{n} \Vert \overline{\Delta} \Vert^2 \label{eq:CSLipschitz1}, \\
\left\vert \left\langle \pi_{E_{\boldsymbol{t}_n^{\star}}}\left(\overline{w}\right),\pi_{E_{\boldsymbol{t}_n^{\star}}}\left(\overline{\Delta}\right)\right\rangle - \left\langle \pi_{E_{\boldsymbol{t}}}\left(\overline{w}\right),\pi_{E_{\boldsymbol{t}}}\left(\overline{\Delta}\right)\right\rangle \right\vert & \leq & 2 \Vert \overline{\Delta} \Vert \Vert \overline{w} \Vert. \label{eq:CSLipschitz2}
\end{eqnarray}
Note that $\overline{w} = z - \overline{\rho}_n Bz$ thus by the triangle inequality

\begin{equation}\label{eq:triangular}
\Vert \overline{w}\Vert \leq \Vert z \Vert + \vert \overline{\rho}_n \vert \Vert Bz \Vert .
\end{equation}

Since $\vert \overline{\rho}_n \vert = O_P\left(1\right)$, we deduce from Lemma \ref{lem:rateXY}
that $\Vert \overline{w}\Vert=O_P \left(n^{1/2}\right)$. Since, by Lemma \ref{Lem:Delta_order}, $\Vert \overline{\Delta} \Vert = O_P\left(1\right)$, we
obtain that
{\begin{equation}\label{eq:controleunif}
\sup_{\boldsymbol{t}} \left| J_n\left(\overline{\Delta},\boldsymbol{t}\right) - \frac{2}{n}\left(\left\langle \pi_{E_{\boldsymbol{t}_n^{\star}}}\left(\overline{w}\right),\pi_{E_{\boldsymbol{t}_n^{\star}}}\left(\overline{\Delta}\right)\right\rangle - \left\langle \pi_{E_{\boldsymbol{t}}}\left(\overline{w}\right),\pi_{E_{\boldsymbol{t}}}\left(\overline{\Delta}\right)\right\rangle\right) \right| =O_P\left(n^{-1/2}\right).
\end{equation}}
For $0<\nu<\Delta_{\boldsymbol{\tau}^{\star}}$, using \eqref{eq:Jn_decomp} and \eqref{eq:Cn_alpha}, we get:
\begin{eqnarray*}
P\left(\left\Vert \boldsymbol{\overline{t}}_n - \boldsymbol{t}^{\star} \right\Vert_{\infty} \geq \nu\right)  & \leq & P \left( \min_{\boldsymbol{t}\in\mathcal{C}_{n,m,\nu}} J_n \left(\overline{v}, \boldsymbol{t} \right) \leq 0\right)\\
  & \leq & P \left( \min_{\boldsymbol{t} \in\mathcal{C}_{n,m,\nu}} \left\lbrace J_n\left(\overline{w},\boldsymbol{t}\right) + J_n\left(\overline{\Delta},\boldsymbol{t}\right)  \right.\right. \\
  & & \ \ - \left.\left. \frac{2}{n}\left(\left\langle \pi_{E_{\boldsymbol{t}_n^{\star}}}\left(\overline{w}\right),\pi_{E_{\boldsymbol{t}_n^{\star}}}\left(\overline{\Delta}\right)\right\rangle - \left\langle \pi_{E_{\boldsymbol{t}}}\left(\overline{w}\right),\pi_{E_{\boldsymbol{t}}}\left(\overline{\Delta}\right)\right\rangle\right)\right\rbrace \leq 0 \right)\\
 & \leq & P \left( \min_{\boldsymbol{t}\in\mathcal{C}_{n,m,\nu}} \left\lbrace K_n\left(\overline{w},\boldsymbol{t}\right) + V_n\left(\overline{w},\boldsymbol{t}\right) + W_n\left(\overline{w},\boldsymbol{t}\right)+ J_n\left(\overline{\Delta},\boldsymbol{t}\right) \right. \right. \\
  & & \ \ - \left. \left. \frac{2}{n}\left(\left\langle \pi_{E_{\boldsymbol{t}_n^{\star}}}\left(\overline{w}\right),\pi_{E_{\boldsymbol{t}_n^{\star}}}\left(\overline{\Delta}\right)\right\rangle - \left\langle \pi_{E_{\boldsymbol{t}}}\left(\overline{w}\right),\pi_{E_{\boldsymbol{t}}}\left(\overline{\Delta}\right)\right\rangle\right)\right\rbrace \leq 0 \right)\\
  & \leq & P \left( \min_{\boldsymbol{t}\in\mathcal{C}_{n,m,\nu}} \left\lbrace \frac{1}{2}K_n\left(\overline{w},\boldsymbol{t}\right) + V_n\left(\overline{w},\boldsymbol{t}\right) + W_n\left(\overline{w},\boldsymbol{t}\right)\right\rbrace \leq 0 \right)\\
   & & + P \left( \min_{\boldsymbol{t}\in\mathcal{C}_{n,m,\nu}} \left\lbrace \frac{1}{2}K_n\left(\overline{w},\boldsymbol{t}\right) + J_n\left(\overline{\Delta},\boldsymbol{t}\right) \right. \right. \\
   & & \ \ - \left. \left. \frac{2}{n}\left(\left\langle \pi_{E_{\boldsymbol{t}_n^{\star}}}\left(\overline{w}\right),\pi_{E_{\boldsymbol{t}_n^{\star}}}\left(\overline{\Delta}\right)\right\rangle - \left\langle \pi_{E_{\boldsymbol{t}}}\left(\overline{w}\right),\pi_{E_{\boldsymbol{t}}}\left(\overline{\Delta}\right)\right\rangle\right)\right\rbrace \leq 0 \right).
\end{eqnarray*}
Following the proof of Lemma \ref{lem:consistency}, one can prove that 
\begin{equation}\label{eq:demi}
P \left( \underset{\boldsymbol{t}\in\mathcal{C}_{n,m,\nu}}{\min} \left\lbrace\frac{1}{2}K_n\left(\overline{w},\boldsymbol{t}\right) + V_n\left(\overline{w},\boldsymbol{t}\right) + W_n\left(\overline{w},\boldsymbol{t}\right)\right\rbrace \leq 0\right) \underset{n\rightarrow\infty}{\longrightarrow} 0 \; .
\end{equation}

Using \eqref{eq:LMbounds}, we get that
\begin{eqnarray}\label{eq:HalfBound}
P \left( \min_{\boldsymbol{t}\in\mathcal{C}_{n,m,\nu}} \left\lbrace\frac{1}{2}K_n\left(\overline{w},\boldsymbol{t}\right) +  J_n\left(\overline{\Delta},\boldsymbol{t}\right) - \frac{2}{n}\left(\left\langle \pi_{E_{\boldsymbol{t}_n^{\star}}}\left(\overline{w}\right),\pi_{E_{\boldsymbol{t}_n^{\star}}}\left(\overline{\Delta}\right)\right\rangle - \left\langle \pi_{E_{\boldsymbol{t}}}\left(\overline{w}\right),\pi_{E_{\boldsymbol{t}}}\left(\overline{\Delta}\right)\right\rangle\right)\right\rbrace \leq 0\right) & & \\
 \leq P \left( \frac{1}{2}\underline{\lambda}^2 \nu + \min_{\boldsymbol{t}\in\mathcal{C}_{n,m,\nu}}  \left\{ J_n\left(\overline{\Delta},\boldsymbol{t}\right) - \frac{2}{n}\left(\left\langle \pi_{E_{\boldsymbol{t}_n^{\star}}}\left(\overline{w}\right),\pi_{E_{\boldsymbol{t}_n^{\star}}}\left(\overline{\Delta}\right)\right\rangle - \left\langle \pi_{E_{\boldsymbol{t}}}\left(\overline{w}\right),\pi_{E_{\boldsymbol{t}}}\left(\overline{\Delta}\right)\right\rangle\right)\right\} \leq 0\right) & & \nonumber
\end{eqnarray}
which goes to zero when $n$ goes to infinity by \eqref{eq:controleunif}.
Then Lemma \ref{lem:consistency} still holds if $y$ is defined by \eqref{eq:modele_new}.
To show the rate of convergence, we use the same decomposition. As in the proof of Lemma \ref{lem:rateT}, $P\left(\underset{\boldsymbol{t} \in \mathcal{C}_{\nu, \gamma, n, m}'\left(\mathcal{I}\right)}{\min} J_n\left(\overline{v},\boldsymbol{t}\right)\leq 0\right)\underset{n\rightarrow \infty}{\longrightarrow} 0$ for all $\nu>0$ and $0<\gamma <1/2$ is a sufficient condition 
for proving that $P\left(\boldsymbol{\widehat{t}}_n(y,\overline{\rho}_n)\in \mathcal{C}_{\nu, \gamma, n, m}\right) \longrightarrow_{n\rightarrow\infty} 0$, which allows us to conclude on the rate of convergence of the estimated change-points. Note that
\begin{eqnarray*}
P\left(\min_{\boldsymbol{t} \in \mathcal{C}_{ \nu, \gamma, n, m}'\left(\mathcal{I}\right)} J_n\left(\overline{v},\boldsymbol{t}\right) \leq 0\right)  & \leq & P \left( \min_{\boldsymbol{t} \in \mathcal{C}_{\nu, \gamma, n, m}'} 
\left\lbrace\frac{1}{2}K_n\left( \overline{w}, \boldsymbol{t} \right) + V_n \left( \overline{w}, \boldsymbol{t}\right) + W_n \left( \overline{w}, \boldsymbol{t} \right)\right\rbrace \leq 0\right)\\
   & + & P \left( \frac{1}{2}\underline{\lambda}^2 \nu + J_n\left(\overline{\Delta},\boldsymbol{t}\right) \right. \\
   & - & \left. \frac{2}{n}\left(\left\langle \pi_{E_{\boldsymbol{t}_n^{\star}}}\left(\overline{w}\right),\pi_{E_{\boldsymbol{t}_n^{\star}}}\left(\overline{\Delta}\right)\right\rangle - \left\langle \pi_{E_{\boldsymbol{t}}}\left(\overline{w}\right),\pi_{E_{\boldsymbol{t}}}\left(\overline{\Delta}\right)\right\rangle\right) \leq 0 \right).
\end{eqnarray*}
In the latter equation, the second term of the rhs goes to zero as $n$ goes to infinity by \eqref{eq:controleunif}.
The first term of rhs goes to zero when $n$ goes to infinity by following the same line of reasoning as the one of Lemma \ref{lem:almost72LM2}. 
This concludes the proof of Proposition \ref{Prop:Segment2}.

\end{proof}


\subsection{Proof of Proposition \ref{Prop:SelBeta}}\label{subsec:bicbardet}
We shall used in this section the notations introduced in Sections \ref{subsec:prop:Segment} and \ref{subsec:beta}.
 The result derives directly from Lemmas \ref{Lem:infSelBeta} and \ref{Lem:supSelBeta}.
\begin{lemma}\label{Lem:infSelBeta}
Under the assumptions of Proposition \ref{Prop:SelBeta}, $P\left(\widehat{m}=m\right)\underset{n\rightarrow\infty}{\longrightarrow} 0$ if $m<m^{\star}$.
\end{lemma}
\begin{lemma}\label{Lem:supSelBeta}
Under the assumptions of Proposition \ref{Prop:SelBeta}, $P\left(\widehat{m}=m\right)\underset{n\rightarrow\infty}{\longrightarrow} 0$ if $m>m^{\star}$.
\end{lemma}
%
\begin{proof}[Proof of Lemma \ref{Lem:infSelBeta}]
If $\widehat{m}=m<m^{\star}$, then
\begin{equation*}
\frac1n SS_m(z, \overline{\rho}_n) + \beta_n m \leq \frac1n SS_{m^{\star}}(z, \overline{\rho}_n) + \beta_n m^{\star} \; ,
\end{equation*}
where $SS_m$ is defined in \eqref{Eq:SSm}. In particular, there exists $\boldsymbol{t}\in \mathcal{A}_{n,m}$ such that
\begin{equation*}
\frac1n \underset{\boldsymbol{\delta}}{\min} SS_m(z, \overline{\rho}_n, \boldsymbol{\delta}, \boldsymbol{t}) + \beta_n m \leq \frac1n \underset{\boldsymbol{\delta}}{\min} SS_m(z, \overline{\rho}_n, \boldsymbol{\delta}, \boldsymbol{t}_n^{\star}) + \beta_n m^{\star} \; .
\end{equation*}
From \eqref{eq:Jnm}, we get
\begin{equation*}
J_n \left(\overline{w},\boldsymbol{t}\right) \leq \beta_n\left(m^{\star} - m\right) \; .
\end{equation*}
Since $\left(\beta_n\right)$ converges to zero, for any $\varepsilon>0$, $\beta_n\left(m^{\star} - m\right) \leq \varepsilon$ for a large
enough $n$, and so
\begin{equation*}
J_n \left(\overline{w},\boldsymbol{t}\right) \leq \varepsilon \; .
\end{equation*}
One can check that there exist $0<\nu<\Delta_{\boldsymbol{\tau}^{\star}}$ such that, for a large enough $n$, there exists $\boldsymbol{t}'\in \mathcal{C}_{n,m^{\star},\nu}$ such that $E_{\boldsymbol{t}} \subset E_{\boldsymbol{t}'}$ (that is the change-points of $\boldsymbol{t}$ are change-points of $\boldsymbol{t}'$) for all $\boldsymbol{t}\in \mathcal{A}_{n,m}$, where $\mathcal{C}_{n,m^{\star},\nu}$ is defined in \eqref{eq:Cn_alpha}. From \eqref{eq:Jnm} and $E_{\boldsymbol{t}} \subset E_{\boldsymbol{t}'}$, we get $J_n \left(\overline{w},\boldsymbol{t}'\right) \leq J_n \left(\overline{w},\boldsymbol{t}\right)$. Then, the following inequality holds for all $\varepsilon>0$ and any large enough $n$:
\begin{equation}\label{eq:ineq_inf_beta}
P\left(\widehat{m}=m\right) \leq P\left(\exists \boldsymbol{t}'\in \mathcal{C}_{n,m^{\star},\nu} , J_n \left(\overline{w},\boldsymbol{t}'\right) \leq \varepsilon\right) \; .
\end{equation}
We then follow the steps of \eqref{eq:ineq_alpha}, $-\nu\underline{\lambda}^2$ being replaced by $\varepsilon-\nu\underline{\lambda}^2$. The convergence of $P\left(\exists \boldsymbol{t}'\in \mathcal{C}_{n,m^{\star},\nu} , J_n \left(\overline{w},\boldsymbol{t}'\right) \leq \varepsilon\right)$ to zero holds with $\varepsilon < \nu\underline{\lambda}^2$. We can conclude with \eqref{eq:ineq_inf_beta}.
\end{proof}
\begin{proof}[Proof of Lemma \ref{Lem:supSelBeta}]
Following the proof of Lemma \ref{Lem:infSelBeta}, if $\widehat{m}=m>m^{\star}$, there exists $\boldsymbol{t}\in\mathcal{A}_{n,m}$ such that $J_n \left(\overline{w},\boldsymbol{t}\right) \leq \beta_n\left(m^{\star} - m\right)$ and then $ J_n \left(\overline{w},\boldsymbol{t}\right) + \beta_n \leq 0 $ since $m>m^{\star}$. Then
\begin{equation}\label{eq:ineq_sup_beta}
P\left(\widehat{m}=m\right) \leq P\left(\exists \boldsymbol{t}\in \mathcal{A}_{n,m} , J_n \left(\overline{w},\boldsymbol{t}\right) + \beta_n \leq 0\right) \; .
\end{equation}
Adding the change-points of $\boldsymbol{t}_n^{\star}$ to those of such a $\boldsymbol{t}$, one can get $t'\in\mathcal{A}_{n,m'}$ with $m^{\star}<m\leq m' \leq m+m^{\star}$ such that $E_{\boldsymbol{t}}\cup E_{\boldsymbol{t}_n^{\star}}\subset E_{\boldsymbol{t}'}$, provided that $\left(m+m^{\star}\right)\left\lceil\Delta_n\right\rceil\leq n$, where $\lceil\cdot \rceil$ is the ceiling function, this condition being fulfilled for any sufficiently large $n$ under the assumptions of Proposition \ref{Prop:SelBeta} since $n^{-1}\Delta_n$ converges to zero. Since $E_{\boldsymbol{t}}\subset E_{\boldsymbol{t}'}$, we derive $J_n \left(\overline{w},\boldsymbol{t}'\right)+ \beta_n \leq J_n \left(\overline{w},\boldsymbol{t}\right) + \beta_n$ from \eqref{eq:Jnm}. Then, from \eqref{eq:ineq_sup_beta}, we get
\begin{equation}\label{eq:sublemma_sup_beta}
\forall m' > m^{\star} , P\left(\exists \boldsymbol{t}'\in \mathcal{A}_{n,m'} , E_{\boldsymbol{t}_n^{\star}}\subset E_{\boldsymbol{t}'},  J_n \left(\overline{w},\boldsymbol{t}'\right)+ \beta_n \leq 0\right) \underset{n\rightarrow\infty}{\longrightarrow} 0
\end{equation}
is a sufficient condition to prove the lemma. Let us prove \eqref{eq:sublemma_sup_beta}. Let $m'>m^{\star}$ and such a $\boldsymbol{t}'$. We compare $J_n \left(\overline{w},\boldsymbol{t}'\right)$ to $J_n \left(w^{\star},\boldsymbol{t}'\right)$. Since $\mathbb{E}w^{\star}\in E_{\boldsymbol{t}_n^{\star}}\subset E_{\boldsymbol{t}'}$, $K_n \left(w^{\star} , \boldsymbol{t}'\right) = 0$ by \eqref{eq:Kn}. By \eqref{eq:Wn} and $\mathbb{E}w^{\star}\in E_{\boldsymbol{t}_n^{\star}}\subset E_{\boldsymbol{t}'}$,
\begin{eqnarray*}
W_n\left(w^{\star},\boldsymbol{t}'\right) & = & \frac{2}{n}\left( \left\langle \pi_{E_{\boldsymbol{t}_n^{\star}}} \left( w^{\star}-\mathbb{E}w^{\star}\right), \pi_{E_{\boldsymbol{t}_n^{\star}}} \left( \mathbb{E}w^{\star} \right)\right\rangle - \left\langle \pi_{E_{\boldsymbol{t}'}} \left( w^{\star}-\mathbb{E}w^{\star}\right), \pi_{E_{\boldsymbol{t}'}} \left( \mathbb{E}w^{\star} \right)\right\rangle  \right)\\
 & = & \frac{2}{n}\left\langle \pi_{E_{\boldsymbol{t}_n^{\star}}} \left( w^{\star}-\mathbb{E}w^{\star}\right) - \pi_{E_{\boldsymbol{t}'}} \left( w^{\star}-\mathbb{E}w^{\star}\right), \pi_{E_{\boldsymbol{t}_n^{\star}}} \left( \mathbb{E}w^{\star} \right)\right\rangle \\
  & = & - \frac{2}{n}\left\langle \pi_{E_{\boldsymbol{t}_n^{\star}}^\bot} \pi_{E_{\boldsymbol{t}'}} \left( w^{\star}-\mathbb{E}w^{\star}\right), \pi_{E_{\boldsymbol{t}_n^{\star}}} \left( \mathbb{E}w^{\star} \right)\right\rangle \\
   & = & 0 \; ,
\end{eqnarray*}
where $E^\bot$ is the (Euclidian) orthogonal complement of the vector subspace $E$. Then $J_n \left(w^{\star},\boldsymbol{t}'\right) = V_n \left(w^{\star},\boldsymbol{t}'\right)$ and
\begin{equation}\label{eq:decomp_sup_beta}
J_n \left(\overline{w},\boldsymbol{t}'\right) = V_n \left(w^{\star},\boldsymbol{t}'\right) + \left(J_n \left(\overline{w},\boldsymbol{t}'\right) - J_n \left(w^{\star},\boldsymbol{t}'\right)\right) \; .
\end{equation}
Using \eqref{eq:LMboundV}, $V_n\left(w^{\star},\boldsymbol{t}\right)  \geq  -\frac{2\left(m'+1\right)}{n\Delta_n}M_n$, where
\begin{eqnarray*}
M_n & = & M_{n,1} + M_{n,2}\;,\\
 M_{n,1} & = & \max_{1\leq s\leq n} \left(\sum_{i=1}^s \epsilon_i\right)^2\;, \\
 M_{n,2} & = &  \max_{1\leq s\leq n} \left(\sum_{i=n-s}^{n} \epsilon_i\right)^2\;.
\end{eqnarray*}
 We define $D_n = \underset{\boldsymbol{t}' \in\mathcal{A}_{n,m'}}{\sup} \left| J_n \left(\overline{w},\boldsymbol{t}'\right) - J_n \left(w^{\star},\boldsymbol{t}'\right) \right|$. Then, using \eqref{eq:decomp_sup_beta},
\begin{equation*}
J_n \left(\overline{w},\boldsymbol{t}'\right) \geq -\frac{2\left(m+1\right)}{n\Delta_n}M_n - D_n \; ,
\end{equation*}
which implies
\begin{eqnarray*}
P\left(\exists \boldsymbol{t}'\in \mathcal{A}_{n,m'} , E_{\boldsymbol{t}_n^{\star}}\subset E_{\boldsymbol{t}'},  J_n \left(\overline{w},\boldsymbol{t}'\right)+ \beta_n \leq 0\right) & \leq & P\left( -\frac{2\left(m'+1\right)}{n\Delta_n}M_n - D_n + \beta_n \leq 0\right)\\
 & \leq & P\left( \frac{2\left(m'+1\right)}{n\Delta_n}M_n \geq \frac{\beta_n}{2} \right) + P\left(  D_n \geq \frac{\beta_n}{2} \right) \; .
\end{eqnarray*}
By Lemma \ref{lem:BoundedUnifBound}, $D_n = O_P \left(n^{-1/2}\right)$ and then $P\left(  D_n \geq \frac{\beta_n}{2} \right)$ tends to zero as $n$ tends to infinity since $n^{1/2}\beta_n\underset{n\rightarrow\infty}{\longrightarrow} + \infty $. Let us now prove that $P\left( \frac{2\left(m+1\right)}{n\Delta_n}M_n \geq \frac{\beta_n}{2} \right)$ tends to zero as $n$ tends to infinty, which concludes the proof. Note that
\begin{equation*}
P\left( \frac{2\left(m'+1\right)}{n\Delta_n}M_n \geq \frac{\beta_n}{2} \right) \leq P\left( M_{n,1} \geq \frac{n\Delta_n \beta_n}{8\left(m'+1\right)} \right) + P\left( M_{n,2} \geq \frac{n\Delta_n \beta_n}{8\left(m'+1\right)} \right) \; .
\end{equation*}
We prove the convergence for each term in the rhs of the above equation. We shall prove it for the first term in the rhs since the arguments for the other term are the same.  From  Kolmogorov's maximal inequality (see for example \cite[Theorem 2.5.2.]{durrett2010probability}), since $\left(\epsilon_i\right)_{i\geq 0}$ is a sequence of independent random variables with zero-mean and finite variance $\sigma^{\star 2}$,
\begin{equation}\label{eq:Hajek}
\forall \delta>0, \;  P \left( M_{n,1} \geq \delta^2 \right) \leq \frac{n\sigma^{\star 2}}{\delta^2} \; . 
\end{equation}
Letting $\delta^2 = \frac{n\Delta_n \beta_n}{8\left(m'+1\right)}$ in \eqref{eq:Hajek}, we get
\begin{equation}
P\left( M_{n,1} \geq \frac{n\Delta_n \beta_n}{8\left(m'+1\right)} \right) 
\leq \frac{8\left(m'+1\right)\sigma^{\star 2}}{\Delta_n \beta_n}\;,
\end{equation}
which goes to $0$ as $n$ tends to infinity because $\Delta_n \beta_n\underset{n\rightarrow\infty}{\longrightarrow} +\infty$. The proof of the convergence of $P\left( M_{n,2} \geq \frac{n\Delta_n \beta_n}{8\left(m'+1\right)} \right)$ follows the same lines.
\end{proof}

\subsection{Proof of Proposition \ref{Prop:SelBeta2}}\label{subsec:mbicbardet_bis}
\begin{lemma}\label{Lem:infSelBeta2}
Under the assumptions of Proposition \ref{Prop:SelBeta2}, $P\left(\widehat{m}=m\right)\underset{n\rightarrow\infty}{\longrightarrow} 0$ if $m<m^{\star}$.
\end{lemma}
\begin{lemma}\label{Lem:supSelBeta2}
Under the assumptions of Proposition \ref{Prop:SelBeta2}, $P\left(\widehat{m}=m\right)\underset{n\rightarrow\infty}{\longrightarrow} 0$ if $m>m^{\star}$.
\end{lemma}
\begin{proof}[Proof of Lemma \ref{Lem:infSelBeta2}]
Following the proof of Lemma \ref{Lem:infSelBeta} and replacing $\overline{w}$ by $\overline{v}$, we get, for any $\varepsilon >0$,
\begin{eqnarray}
P\left(\widehat{m}=m\right) & \leq & P\left(\exists \boldsymbol{t}'\in \mathcal{C}_{n,m^{\star},\nu} , J_n \left(\overline{v},\boldsymbol{t}'\right) \leq \varepsilon\right)\label{eq:ineq_inf_beta2}\\
 & \leq & P\left(\exists \boldsymbol{t}'\in \mathcal{C}_{n,m^{\star},\nu} , \frac{1}{2}K_n \left(\overline{w},\boldsymbol{t}'\right) + V_n\left(\overline{w},\boldsymbol{t}'\right) + W_n\left(\overline{w},\boldsymbol{t}'\right) \leq \frac{\varepsilon}{2}\right) \label{eq:ineq_inf_beta2_decompo}\\
 & + & P\left(\exists \boldsymbol{t}'\in \mathcal{C}_{n,m^{\star},\nu} , \frac{1}{2}K_n \left(\overline{w},\boldsymbol{t}'\right) + J_n \left(\overline{v},\boldsymbol{t}'\right) - J_n \left(\overline{w},\boldsymbol{t}'\right) \leq \frac{\varepsilon}{2}\right) \;, \nonumber
\end{eqnarray}
since
$$ J_n \left(\overline{v},\boldsymbol{t}'\right) = 
  \frac{1}{2}K_n \left(\overline{w},\boldsymbol{t}'\right) + V_n\left(\overline{w},\boldsymbol{t}'\right) + W_n\left(\overline{w},\boldsymbol{t}'\right)
  + 
  \frac{1}{2}K_n \left(\overline{w},\boldsymbol{t}'\right) + J_n \left(\overline{v},\boldsymbol{t}'\right) - J_n \left(\overline{w},\boldsymbol{t}'\right). $$
From \eqref{eq:demi} and \eqref{eq:ineq_inf_beta2_decompo}, it suffices to prove that
$$P\left(\exists \boldsymbol{t}'\in \mathcal{C}_{n,m^{\star},\nu} , \frac{1}{2}K_n \left(\overline{w},\boldsymbol{t}'\right) + J_n \left(\overline{v},\boldsymbol{t}'\right) - J_n \left(\overline{w},\boldsymbol{t}'\right) \leq \frac{\varepsilon}{2}\right)\underset{n\to\infty}{\longrightarrow}0  $$
to conclude the proof. It follows from \eqref{eq:controleunif} and \eqref{eq:HalfBound}, $\frac{1}{2}\underline{\lambda}^2 \nu$ being replaced by $\frac{1}{2}\left(\underline{\lambda}^2 \nu - \varepsilon\right)$, which is positive if $\varepsilon < \underline{\lambda}^2 \nu$.
\end{proof}
\begin{proof}[Proof of Lemma \ref{Lem:supSelBeta2}]
As in the Proof of Lemma \ref{Lem:supSelBeta}, it suffices to show that 
$$P\left(\exists \boldsymbol{t}\in\mathcal{A}_{n,m}, J_n\left(\overline{v},\boldsymbol{t}\right) + \beta_n \leq 0\right)\underset{n\to\infty}{\longrightarrow} 0 \; .$$
Since $$J_n\left(\overline{v},\boldsymbol{t}\right)\geq J_n\left(\overline{w},\boldsymbol{t}\right) - \underset{t}{\sup}\left|J_n\left(\overline{v},\boldsymbol{t}\right) - J_n\left(\overline{w},\boldsymbol{t}\right)\right| \; , $$
the result follows from
\begin{eqnarray}
P\left(\exists \boldsymbol{t}\in\mathcal{A}_{n,m}, J_n\left(\overline{w},\boldsymbol{t}\right) + \frac{1}{2}\beta_n \leq 0\right) & \underset{n\to\infty}{\longrightarrow} & 0 \label{eq:HalfBeta}\\
P\left(\underset{t}{\sup}\left|J_n\left(\overline{v},\boldsymbol{t}\right) - J_n\left(\overline{w},\boldsymbol{t}\right)\right| \geq \frac{1}{2}\beta_n\right)& \underset{n\to\infty}{\longrightarrow} & 0 \label{eq:Diffvwbeta}
\end{eqnarray}
\eqref{eq:HalfBeta} follows from the Proof of Lemma \ref{Lem:supSelBeta}, replacing $\beta_n$ by $\frac{1}{2}\beta_n$. \eqref{eq:Diffvwbeta} follows from \eqref{eq:controleunif} and from $n^{1/2}\beta_n\underset{n\to\infty}{\longrightarrow} +\infty$.
\end{proof}

\subsection{Proof of Proposition \ref{Prop:mBICBardet}} 

We first give some lemmas which are useful for the proof of Proposition \ref{Prop:mBICBardet}.

\begin{lemma} \label{Lem:SSm} Under the assumptions of 
Proposition \ref{Prop:mBICBardet} with $SS_m$ given by \eqref{Eq:SSm}, we have, for any positive $m$,
$$
SS_m(z, \overline{\rho}_n) = SS_m(z, \rho^{\star}) + O_P(1),\textrm{ as } n\to\infty\;.
$$
\end{lemma}

\begin{lemma} \label{Lem:SSmstar} Under the assumptions of Proposition
\ref{Prop:mBICBardet} with $SS_m$ given by \eqref{Eq:SSm}, we have,
for any positive $m$,
$$
SS_m(z, \rho^{\star})^{-1} = O_P(n^{-1}),\textrm{ as } n\to\infty\;.
$$
\end{lemma}

\begin{proof}[Proof of Lemma \ref{Lem:SSm}]
 The proof of this Lemma follows exactly this of Lemma \ref{Lem:SSmY}. The difference is that, in \eqref{eq:bkw}, the term $\Delta^{\star}$ appearing in the decomposition \eqref{eq:dec:second:terme} vanishes.
\end{proof}

\begin{proof}[Proof of Lemma \ref{Lem:SSmstar}]
We first define 
$$
SS_m \left(z, \rho, \boldsymbol{t}\right) = \underset{\delta}{\argmin } SS_m \left(z, \rho, \delta, \boldsymbol{t} \right).
$$ 
We have, for any positive $M$, 
\begin{eqnarray*}
 P\left(\frac{n}{SS_m(z, \rho^{\star})} > M \right) 
 & \leq & P\left( \left\{\frac{SS_m(z, \rho^{\star})}{SS_m(z, \rho^{\star}, \boldsymbol{t^{\star}})} > 1 \right\} \bigcap \left\{\frac{n}{SS_m(z, \rho^{\star})} > M\right\} \right) \\
  & & + P\left( \left\{\frac{SS_m(z, \rho^{\star})}{SS_m(z, \rho^{\star}, \boldsymbol{t^{\star}})} < 1 \right\} \bigcap \left\{\frac{n}{SS_m(z, \rho^{\star})} > M\right\} \right) \\
  & \leq & P\left(\frac{n}{SS_m(z, \rho^{\star}, \boldsymbol{t^{\star}})} > M \right) + P\left(\frac{SS_m(z, \rho^{\star})}{SS_m(z, \rho^{\star}, \boldsymbol{t^{\star}})} < 1 \right).
\end{eqnarray*}
Under the assumptions of Proposition \ref{Prop:Segment}, a by product of the proof of Theorem 3 in \cite{LM} is that
$$
P\left(\frac{SS_m(z, \rho^{\star})}{SS_m(z, \rho^{\star}, \boldsymbol{t^{\star}})} < 1 \right) = P\left(SS_m(z, \rho^{\star}) - SS_m(z, \rho^{\star}, \boldsymbol{t^{\star}}) < 0 \right) 
\leq \kappa n^{-\alpha},
$$
where $\kappa$ is a positive constant depending on $\boldsymbol{\delta^{\star}}$ and $\boldsymbol{t^{\star}}$, and $\alpha$ is a positive constant. 
Furthermore, as $\sigma^{\star -2}SS_m(z, \rho^{\star}, \boldsymbol{t^{\star}})$ has a $\chi^2_{n-m-1}$ distribution, 
$n^{-1}SS_m(z, \rho^{\star}, \boldsymbol{t^{\star}})=\sigma^{\star 2}+o_P(1)$ and thus {$n^{-1} SS_m(z, \rho^{\star}, \boldsymbol{t^{\star}})=O_P(1)$}, which concludes the proof.
\end{proof}

\begin{proof}[Proof of Proposition \ref{Prop:mBICBardet}]
We have to prove that, for a given positive $m$, $C_m(z, \rho^{\star}) - C_m(z, \overline{\rho}_n) = O_P(1)$. 
Observe that, since $\widehat{\tau}_k(z, \rho) = \widehat{t}_k(z, \rho) /n$, 
\begin{eqnarray} \label{Eq:DecompNkZ}
& & \sum_{k=0}^m \log n_k(\widehat{t}(z, \overline{\rho}_n)) - \sum_{k=0}^m \log
n_k(\widehat{t}(z, {\rho}^{\star})) \nonumber \\
& = & \sum_{k=0}^m \log (\widehat{\tau}_{k+1}(z, \overline{\rho}_n)-\widehat{\tau}_{k}(z, \overline{\rho}_n))- 
\sum_{k=0}^m \log (\widehat{\tau}_{k+1}(z, {\rho}^{\star})-\widehat{\tau}_{k}(z, {\rho}^{\star})).
\end{eqnarray}
By Proposition  \ref{Prop:Segment}, both quantities of the previous equation converge in probability
to $$\sum_{k=0}^m \log (\tau^{\star}_{k+1}-\tau^{\star}_{k}) \,$$
thus
\begin{equation} \label{Eq:CvgceNkZ}
\sum_{k=0}^m \log n_k(\widehat{t}(z, \overline{\rho}_n))- \sum_{k=0}^m \log
n_k(\widehat{t}(z, {\rho}^{\star}))=O_P(1). 
\end{equation}
Further note that
$$
\log {SS}_m(z, \overline{\rho}_n) - \log {SS}_m(z, {\rho}^{\star})
=\log\left(\frac{{SS}_m(z, \overline{\rho}_n)}{{SS}_m(z, {\rho}^{\star})}\right) 
= R\left(\frac{{SS}_m(z, \overline{\rho}_n)-{SS}_m(z, {\rho}^{\star})}{{SS}_m(z, {\rho}^{\star})}\right),
$$
where $R(x)=\log(1+x)$.
Lemma \ref{Lem:SSm} states that ${SS}_m(z, \overline{\rho}_n)-{SS}_m(z, {\rho}^{\star}) = O_P(1)$ and Lemma \ref{Lem:SSmstar} that $[{SS}_m(z, {\rho}^{\star})]^{-1} = O_P(n^{-1})$ so, by \cite[Lemma 2.12]{van}, we get that
$$
\log {SS}_m(z, \overline{\rho}_n) - \log {SS}_m(z, {\rho}^{\star})
= O_P(n^{-1}).
$$
Hence
$$
\frac{n-m+1}{2} \log {SS}_m(z, \overline{\rho}_n) - \frac{n-m+1}{2} \log {SS}_m(z, {\rho}^{\star}) =
O_P(1),
$$
which with (\ref{Eq:CvgceNkZ}) concludes the proof of Proposition \ref{Prop:mBICBardet}.
\end{proof}

\subsection{Proof of Proposition \ref{Prop:mBIC}}
We first give some lemmas which are useful for the proof of Proposition \ref{Prop:mBIC}.

\begin{lemma} \label{Lem:SSmY} Under the assumptions of 
Proposition \ref{Prop:mBICBardet} with $SS_m$ given by \eqref{Eq:SSm}, we have, for any positive $m$,
$$
SS_m(y, \overline{\rho}_n) = SS_m(y, \rho^{\star}) + O_P(1),\textrm{ as } n\to\infty\;.
$$
\end{lemma}

\begin{lemma} \label{Lem:SmYZ}
  If $(y_0, \dots y_n)$ is defined by \eqref{eq:modele_new} and $(z_0, \dots z_n)$ is defined as in Lemma \ref{Lem:YZ}, then 
  $$
  SS_m(y, {\rho}^{\star}) = SS_m(z, {\rho}^{\star}) + O_P(1),\textrm{ as } n\to\infty\;.
  $$
\end{lemma}

\begin{lemma} \label{Lem:OP1}
  Let  $(X_n)$ and $(Y_n)$ be two sequences of random variables such that $X_n - Y_n = O_P(1)$. If $Y_n ^{-1} = O_P(n^{-1})$ then $X_n ^{-1} = O_P(n^{-1})$.
\end{lemma}

\begin{proof}[Proof of Lemma \ref{Lem:SSmY}]
Using the matrix notations from the proof of Lemma \ref{Lem:Delta_order}, we have
$$
SS_m(y, \rho^{\star}) = \min_{T, \delta} \Vert y - \rho^{\star} By - T \delta \Vert^2, 
\qquad
SS_m(y, \overline{\rho}_n) = \min_{T, \delta} \Vert y - \overline{\rho}_n  By - T \delta \Vert^2,
$$
where all minimizations are achieved over all segmentations with $m$ change points belonging to $\mathcal{A}_{n,m}$. 
Let us define $(\widehat{T}^{\star}, \widehat{\delta}^{\star})$
and $(\overline{T}, \overline{\delta})$ by 
\begin{equation*}
(\widehat{T}^{\star}, \widehat{\delta}^{\star}) = \arg\min_{T, \delta} \Vert y - \rho^{\star}By - T\delta \Vert, 
\qquad
(\overline{T}, \overline{\delta}) = \arg\min_{T, \delta} \Vert y - \overline{\rho}_n By - T\delta \Vert.
\end{equation*}
Note that $\widehat{T}^{\star}$ and $\overline{T}$ refer to $\widehat{t}(y, \rho^{\star})$ and $\widehat{t}(y, \overline{\rho}_n)$, respectively. We have 
\begin{eqnarray} \label{Eq:UpperBoundSSdiff}
 \left| SS_m(y, \overline{\rho}_n) - SS_m(y, \rho^{\star}) \right| 
 & = &  \left| \min_{T, \delta} \Vert y - \overline{\rho}_n By - T\delta \Vert^2 - \min_{T, \delta} \Vert y - \rho^{\star} By - T\delta \Vert^2 \right| \nonumber \\
 & \leq & 
 \max \left( \left| \Vert y - \overline{\rho}_n By - \widehat{T}^{\star}\widehat{\delta}^{\star} \Vert^2 - \Vert y - \rho^{\star}By - \widehat{T}^{\star}\widehat{\delta}^{\star} \Vert^2 \right|, \right. \nonumber\\
  & & \qquad \left. \left| \Vert y - \overline{\rho}_n By - \overline{T}\,\overline{\delta} \Vert^2 - \Vert y - \rho^{\star}By - \overline{T}\,\overline{\delta} \Vert^2 \right| \right).
\end{eqnarray}
We now have to prove that this upper bound is $O_P(1)$. We first prove it for the second term of in the rhs of \eqref{Eq:UpperBoundSSdiff}. To do so, observe that
$\Vert y - \overline{\rho}_n By - \overline{T}\,\overline{\delta} \Vert^2
=\Vert y - \rho^{\star}By - \overline{T}\,\overline{\delta}+(\rho^{\star}-\overline{\rho}_n) By\Vert^2$. Thus,
$$
\Vert y - \overline{\rho}_n By - \overline{T}\overline{\delta} \Vert^2 - \Vert y - \rho^{\star}By - \overline{T}\overline{\delta} \Vert^2 
=(\overline{\rho}_n  - \rho^{\star})^2 \Vert By\Vert^2 +2(\rho^{\star}-\overline{\rho}_n)\langle By,y - \rho^{\star}By - \overline{T}\,\overline{\delta} \rangle.
$$
Since, by (\ref{eq:modele_matriciel}) and Lemma \ref{Lem:YZ}, $y - \rho^{\star}By - \overline{T}\,\overline{\delta}
=\epsilon - \Delta^{\star} + (T^{\star}\delta^{\star}-\overline{T}\,\overline{\delta})=\epsilon- \Delta^{\star}
+T^{\star}(\delta^{\star}-\overline{\delta})+(T^{\star}-\overline{T})\overline{\delta}$, where $\Delta^{\star}$ is the $n$-dimensional vector with entries $\Delta^{\star}_i$, we get
\begin{multline}\label{eq:dec:second:terme}
\Vert y - \overline{\rho}_n By - \overline{T}\,\overline{\delta} \Vert^2 - \Vert y - \rho^{\star}By - \overline{T}\,\overline{\delta} \Vert^2\\
=(\overline{\rho}_n  - \rho^{\star})^2 \Vert By\Vert^2 
+ 2 (\rho^{\star}- \overline{\rho}_n ) \left(\langle By,\epsilon  \rangle + \langle By,T^{\star} (\delta^{\star} - \overline{\delta} )  \rangle 
+ \langle By,(T^{\star}-\overline{T}) \overline{\delta}  \rangle - \langle By, \Delta^{\star}\rangle \right).
\end{multline}
Let us now prove that each term in the rhs of (\ref{eq:dec:second:terme}) is $O_P(1)$.
\begin{enumerate}[($a$)]
\item Let us study the first term of (\ref{eq:dec:second:terme}).
Using Lemma \ref{lem:rateXY} and (\ref{eq:hypRhoRate}) we get that 
  \begin{equation}\label{Eq:rhoXnorm}
  (\overline{\rho}_n -\rho^{\star})^2 \Vert By\Vert^2 = O_P(1).
  \end{equation}
\item Let us now study the second term of (\ref{eq:dec:second:terme}). Observe that 
$
\left\langle By,\epsilon\right\rangle = \sum_{i=1}^n y_{i-1}\epsilon_i=\sum_{i=1}^n (y_{i-1}-\PE(y_{i-1}))\epsilon_i+\sum_{i=1}^n \PE(y_{i-1})\epsilon_i.
$
By using the central limit theoreom for i.i.d. random variables and since there is a finite number of change-points, the second term
is $O_P(\sqrt{n})$. As for the first term, since $(y_{i-1}-\PE(y_{i-1}))$ is a causal AR(1) process, then by using the beginning of the proof of 
\cite[Proposition 8.10.1]{brockwell}, we get
that $\sum_{i=1}^n (y_{i-1}-\PE(y_{i-1}))\epsilon_i=O_P (\sqrt{n})$. Thus,
  \begin{equation} \label{Eq:CrossProd1}  
  \left\langle By,\epsilon\right\rangle 
= O_P (\sqrt{n}).
 \end{equation}
  Furthermore, we have $\Vert T^{\star} ( \delta^{\star} - \overline{\delta} ) \Vert^2 = \sum_{k=0}^{m}\left(t_{k+1}^{\star}-t_k^{\star}\right) (\delta^{\star}_k -\overline{\delta}_k)^2$ 
 where each term of the sum is $O_P(1)$, thanks to Proposition \ref{Prop:Segment2}, and so is the sum. 
Now using Lemma \ref{lem:rateXY} and the Cauchy-Schwarz inequality, we get 
\begin{equation} \label{Eq:CrossProd2}  
  \langle By, T^{\star} ( \delta^{\star} - \overline{\delta} )  \rangle = O_P(\sqrt{n}).
  \end{equation}
  The convergence rate of $\widehat{t}(y,\overline{\rho}_n)$ given in Proposition \ref{Prop:Segment2} ensures that, for any $\varepsilon>0$ there exists a positive $M$ 
such that each column of $(T^{\star} - \overline{T})$ has at most $M$ non-zero coefficients with probability greater than $1-\varepsilon$. 
By using Proposition \ref{Prop:Segment2}, we obtain that with probability greater than $1-\varepsilon$
  \begin{equation}\label{eq:T*-Tbar}
  \Vert(T^{\star} - \overline{T}) \overline{\delta} \Vert^2 \leq M \sum_k \overline{\delta}_k^2 =2 M \sum_k (\overline{\delta}_k-\delta^{\star}_k)^2
+2M \sum_k {\delta^{\star}_k}^2\leq MM',
\end{equation}
where $M'$ is a positive constant. 
By the Cauchy-Schwarz inequality, (\ref{eq:T*-Tbar}) and Lemma \ref{lem:rateXY}, we get 
  \begin{equation} \label{Eq:CrossProd3}  
  \langle By, (T^{\star} - \overline{T}) \overline{\delta}  \rangle = O_P (\sqrt{n}).
  \end{equation}
  As $\Delta^{\star}$ has only $m$ non-zero entries, $\langle By, \Delta^{\star}\rangle$ is the sum of $m$ Gaussian rv's and is therefore $O_P(1)$.
  
  Thus, combining \eqref{Eq:CrossProd1}, \eqref{Eq:CrossProd2} and \eqref{Eq:CrossProd3} with (\ref{eq:hypRhoRate}), we get
  $$
  (\rho^{\star}- \overline{\rho}_n ) \left(\langle By,\epsilon  \rangle + \langle By,T^{\star} (\delta^{\star} - \overline{\delta} )  \rangle + \langle By,(T^{\star} - \overline{T}) \overline{\delta}  \rangle -  \langle By, \Delta^{\star}\rangle \right) = O_P(1).
  $$
\end{enumerate}

To complete the proof, we need to consider the first term of \eqref{Eq:UpperBoundSSdiff}.
As $\rho^{\star}$ satisfies the same assumptions as $\overline{\rho}_n$, using the same line of reasoning as for the second term holds so we get
$$
\Vert y - \overline{\rho}_n By - \widehat{T}^{\star}\widehat{\delta}^{\star} \Vert^2 - \Vert y - \rho^{\star}By - \widehat{T}^{\star}\widehat{\delta}^{\star} \Vert^2 = O_P(1).
$$
\end{proof}

\begin{proof}[Proof of Lemma \ref{Lem:SmYZ}] 
The proof follows the same line of reasoning as the proof of Lemma \ref{Lem:SSmY}.
Let us define $(\widehat{T}^y, \widehat{\delta}^y)$
and $(\widehat{T}^z, \widehat{\delta}^z)$ by 
\begin{equation*}
(\widehat{T}^y, \widehat{\delta}^y)
= \arg\min_{T, \delta} 
\Vert y - \rho^{\star} By - T \delta \Vert^2,
\qquad
(\widehat{T}^z, \widehat{\delta}^z)
= \arg\min_{T, \delta} 
\Vert z - \rho^{\star} Bz - T \delta \Vert^2.
\end{equation*}
We have 
\begin{multline*} 
\left|SS_m(y, \rho^{\star}) - SS_m(z, \rho^{\star}) \right\| 
 \leq 
 \max \left( \left| \Vert y - \rho^{\star} By - \widehat{T}^y \widehat{\delta}^y \Vert^2 - \Vert z - \rho^{\star}Bz - \widehat{T}^y \widehat{\delta}^y \Vert^2 \right|, \right. \\
 \left. \left| \Vert y - \rho^{\star} By - \widehat{T}^z \widehat{\delta}^z \Vert^2 - \Vert z - \rho^{\star}Bz - \widehat{T}^z \widehat{\delta}^z \Vert^2 \right| \right).
\end{multline*}
According to Lemma \ref{Lem:YZ}, we have $y - \rho^{\star} By = z - \rho^{\star} Bz - \Delta^{\star}$ where $\Delta^{\star}=(\Delta_i^{\star})$. As for the first term
\begin{multline*}
\Vert y - \rho^{\star} By - \widehat{T}^y \widehat{\delta}^y \Vert^2 - \Vert z - \rho^{\star}Bz - \widehat{T}^y \widehat{\delta}^y \Vert^2 \\
= \Vert \Delta^{\star} \Vert^2 
- 2 \left(\langle \Delta^{\star},\epsilon \rangle + \langle \Delta^{\star}, T^{\star} (\delta^{\star} - \widehat{\delta}^y )  \rangle 
+ \langle \Delta^{\star},(T^{\star}-\widehat{T}^y) \widehat{\delta}^y \rangle \right),
\end{multline*}
the first term of which is a constant and all other terms being $O_P(1)$, which can be proved following the same line as the proof of Lemma \ref{Lem:SSmY}. The control of $\Vert y - \rho^{\star} By - \widehat{T}^z \widehat{\delta}^z \Vert^2 - \Vert z - \rho^{\star}Bz - \widehat{T}^z \widehat{\delta}^z \Vert^2$ follows the same lines.
  \end{proof}

\begin{proof}[Proof of Lemma \ref{Lem:OP1}]
Observe that 
$$X_n^{-1}=\left(Y_n+(X_n-Y_n)\right)^{-1}=Y_n^{-1}\left(1+Y_n^{-1}(X_n-Y_n)\right)^{-1}\; .$$
Since, by assumption, $Y_n^{-1}(X_n-Y_n)=O_P(n^{-1})$, the terms inside
the parentheses converges in probability to one. Thus, $\left(1+Y_n^{-1}(X_n-Y_n)\right)^{-1}$ is in particular $O_P(1)$ which concludes the proof.
\end{proof}

\begin{proof}[Proof of Proposition \ref{Prop:mBIC}]
   As for the proof of Proposition \ref{Prop:mBICBardet}, denoting $\widehat{\tau}_k(y, \rho) = \widehat{t}_k(y, \rho) /n$, the decomposition \eqref{Eq:DecompNkZ} still holds, replacing $z$ with $y$. Then, by proposition \ref{Prop:Segment2}, we have
   $$
   \sum_{k=0}^m \log n_k(\widehat{t}(y, \overline{\rho}_n))- \sum_{k=0}^m \log
   n_k(\widehat{t}(y, {\rho}^{\star}))=O_P(1). 
   $$
   For a process $y$ under model \eqref{eq:modele_new}, we construct a process $z$ under model \eqref{eq:bkw} using Lemma \ref{Lem:YZ}. The proof relies on the fact that $y$ inherits some properties of $z$. Again, we note that
   $$
   \log {SS}_m(y, \overline{\rho}_n) - \log {SS}_m(y, {\rho}^{\star})
    = R\left(\frac{{SS}_m(y, \overline{\rho}_n)-{SS}_m(y, {\rho}^{\star})}{{SS}_m(y, {\rho}^{\star})}\right).
   $$
  Lemma \ref{Lem:SSmY} states that ${SS}_m(y, \overline{\rho}_n)-{SS}_m(y, {\rho}^{\star}) = O_P(1)$. 
  To conclude the proof we need to further show that $[{SS}_m(y, {\rho}^{\star})]^{-1} = O_P(n^{-1})$. We first show that $[{SS}_m(y, {\rho}^{\star}) - {SS}_m(z, {\rho}^{\star})] = O_P(1)$ in Lemma \ref{Lem:SmYZ} and, because $[{SS}_m(z, {\rho}^{\star})]^{-1} = O_P(n^{-1})$, we conclude using Lemma \ref{Lem:OP1}.
\end{proof}



\end{document}